\documentclass[a4paper, 12pt]{book}	
\usepackage[T1]{fontenc}
\usepackage{amssymb}
\usepackage{amsmath}
\usepackage{amsthm}
\usepackage[english]{babel}
\usepackage{graphicx}
\usepackage{epstopdf}
\usepackage{epsfig}
\usepackage{color}
\usepackage{latexsym}
\usepackage[all]{xy}
\usepackage{fancyhdr}
\usepackage{geometry}
\usepackage{graphics}

\usepackage{verbatim} 

\makeatletter
\renewcommand{\@makechapterhead}[1]{%
\vspace*{50 pt}%
{\setlength{\parindent}{0pt} \raggedright \normalfont
\bfseries\Huge
\ifnum \value{secnumdepth}>1 
   \if@mainmatter\thechapter\hspace{5mm}\ \fi%
\fi
#1\par\nobreak\vspace{40 pt}}}
\makeatother

\newcommand\nc{\newcommand}

\theoremstyle{plain}

\newtheorem{theorem}{Theorem}
\newtheorem{corollary}[theorem]{Corollary}
\newtheorem{proposition}[theorem]{Proposition}
\newtheorem{lemma}[theorem]{Lemma}
\newtheorem*{lemma*}{Lemma}
\newtheorem*{theorem*}{Theorem}

\theoremstyle{definition}

\newtheorem{example}[theorem]{Example}
\newtheorem{remark}[theorem]{Remark}
\newtheorem*{remark*}{Remark}
\newtheorem*{example*}{Example}

\theoremstyle{plain}
\newtheorem{Atheorem}{Theorem}[section]
\newtheorem{Alemma}[Atheorem]{Lemma}
\newtheorem{Aproposition}[Atheorem]{Proposition}
\newtheorem{Acorollary}[Atheorem]{Corollary}

\def\be#1\ee{\begin{equation}#1\end{equation}}

\newcommand{\sn}{\operatorname{sn}}

\newcommand{\con}{\equiv}
\newcommand{\tr}{\operatorname{tr}}
\newcommand{\End}{\operatorname{End}}
\newcommand{\cl}{\operatorname{cl}}
\newcommand{\id}{\operatorname{id}}
\newcommand{\gr}{\operatorname{gr}}
\newcommand{\ord}{\operatorname{ord}}
\newcommand{\rk}{\operatorname{rk}}
\newcommand{\ev}{\operatorname{ev}}
\newcommand{\lk}{\operatorname{lk}}

\newcommand{\inlim}{\lim\limits_{\overleftarrow{\hspace{2mm}n\hspace{2mm}}}} 
\newcommand{\inverselim}[1]{\lim\limits_{\overleftarrow{\hspace{2mm}#1\hspace{2mm}}}} 

\nc{\homeo}{\cong}
\nc{\iso}{\simeq}

\def\Habiro{\widehat{\Z[q]}}

\nc{\qbinom}[2]{\left[\begin{array}{c}#1\\#2 \end{array}\right]}

\def\cR{\mathcal R}
\def\cL{\mathcal L}
\def\cT{\mathcal T}
\def\cS{\mathcal S}

\newcommand{\U}{\mathcal{U}}
\def\F{\mathcal F}

\newcommand{\tF}{\tilde{F}}
\newcommand{\tcS}{\tilde {\mathcal S}}

\def\fg{\mathfrak{g}}

\def\bk{\mathbf{k}}

\def\bp{{\mathbf p}}
\def\bn{{\mathbf n}}
\def\bj{{\mathbf j}}
\def\bv{{\mathbf v}}

\def\N{\mathbb{N}}
\def\Q{\mathbb{Q}}
\def\R{\mathbb{R}}
\def\C{\mathbb{C}}
\def\Z{\mathbb{Z}}

\nc\modb {{\mathsf b}}
\nc\modB {{\mathsf B}}

\nc\ve{\varepsilon}

\nc\FIG[3]{
    \begin{figure}
    \includegraphics[#3]{#1}
    \caption{#2}
    \label{fig:#1}
    \end{figure}}

\nc\incl[2]{{\includegraphics[height=#1]{#2}}}

\newcommand{\bz} {{\bar 0}}

\nc{\hsp}{\hspace{5mm}}
\nc{\vsp}{\vspace{5mm}}

\begin{document}	
\pagestyle{plain}

\titlepage{
\begin{centering}

\vspace*{2.5\baselineskip}

{\LARGE {\bfseries Unified Quantum $SO(3)$ and $SU(2)$ Invariants \\[1ex] for Rational Homology 3--Spheres}}\\[12ex]
Irmgard B\"uhler\\[5ex]
Z\"urich, 2010

\end{centering}}

\newpage
\thispagestyle{empty}
\begin{center}
\Large {\bf Abstract} 
\normalsize
\end{center}


Inspired by E. Witten's work, N. Reshetikhin and V. Turaev introduced in 1991
important invariants for 3--manifolds and links in 3--manifolds, the
so--called quantum (WRT) $SU(2)$ invariants. Short after, R. Kirby and
P.  Melvin defined a modification of these invariants, called the
quantum (WRT) $SO(3)$ invariants.  Each of these invariants depends on
a root of unity.

In this thesis, we give a unification of these invariants. Given a
rational homology 3--sphere $M$ and a link $L$ inside, we define the
unified invariants $I^{SU(2)}_{M,L}$ and $I^{SO(3)}_{M,L}$, such that
the evaluation of these invariants at a root of unity equals the
corresponding quantum (WRT) invariant. In the $SU(2)$ case, we assume
the order of the first homology group of the manifold to be odd.
Therefore, for rational homology 3--spheres, our invariants dominate the whole set of $SO(3)$ quantum
(WRT) invariants and, for manifolds with the order of the first homology group odd, the whole set of $SU(2)$ quantum (WRT)
invariants. We further show, that the unified invariants have a strong integrality property, i.e. that they lie in modifications of the Habiro ring, which is a cyclotomic completion of
the polynomial ring $\Z[q]$. 

We also give a complete computation of the quantum (WRT) $SO(3)$
and $SU(2)$ invariants of lens spaces with a colored unknot inside.

\pagestyle{fancy}
\setlength{\headheight}{15.4pt}

\fancyhead{}
\fancyhead[LE]{\thepage}
\fancyhead[RO]{\thepage}
\fancyfoot{}
\renewcommand{\headrulewidth}{0pt} 
\renewcommand{\footrulewidth}{0pt}

\tableofcontents
\let\cleardoublepage\clearpage

\chapter*{Introduction}

In 1984, V. Jones \cite{Jo} discovered the famous Jones polynomial, a strong link invariant which led to a rapid development of knot theory. Many new link invariants were defined short after, including the so--called colored Jones polynomial which uses representations of a ribbon Hopf algebra acting as colors attached to each link component. The whole collection of invariants of this spirit are called quantum link invariants.
\\

In the 60' and 70' of the last century, Likorish \cite{Li62}, Wallace \cite{Wallace} and Kirby \cite{Kirby} showed, that there is a one--to--one correspondence via surgery between closed oriented 3--manifolds up to homeomorphisms and knots in the $3$--dimensional sphere modulo Kirby--moves. This gives the possibility to study 3--manifolds using knot theory. 
\\

In 1989, E. Witten \cite{Wi} considered quantum field theory defined by the noncommutative Chern--Simons action to define (on a physical level of rigor) certain invariants of closed oriented 3--manifolds and links in 3--manifolds. Inspired by this work, N. Reshetikhin and V. Turaev \cite{RT2, Tu} constructed in 1991 new topological invariants of 3--manifolds and of links in 3--manifolds. 
The construction goes as follows. Let $M$ be a closed, oriented 3--manifold and $L_M$ its corresponding surgery link. The quantum group $U_q(\mathfrak{sl}_2)$ is a deformation of the Lie algebra $\mathfrak{sl}_2$ and has the structure of a ribbon Hopf algebra. One now takes the sum of the colored Jones polynomial of $L_M$, normalized in an appropriate way, over all colors, i.e over all finite--dimensional irreducible representations of $U_q(\mathfrak{sl}_2)$. Evaluating at a root of unity $\xi$ makes the sum finite and well--defined. 
These invariants are denoted by $\tau_M(\xi)$. Together they form a sequence of complex numbers parameterized by complex roots of unity and are known either as the Witten--Reshetikhin--Turaev invariants, short WRT invariants, or as the quantum invariants of $3$--manifolds. 
Since the irreducible representations of the quantum group $U_q(\mathfrak{sl}_2)$ correspond to the irreducible representations of the Lie group $SU(2)$, they are sometimes also called the quantum (WRT) $SU(2)$ invariants.

R. Kirby and P. Melvin \cite{KM} defined the $SO(3)$ version of the quantum (WRT) invariants by summing only over representations of $U_q(\mathfrak{sl}_2)$ of \emph{odd} dimension and evaluating at roots of unity of \emph{odd} order. These invariants are known as the quantum (WRT) $SO(3)$ invariants.
They have very nice properties. For example, A. Beliakova and T. Le \cite{BL} showed that they are algebraic integers, i.e. $\tau^{SO(3)}_M(\xi)\in \Z[\xi]$ for any closed oriented 3--manifold $M$ and any root of unity $\xi$ (of odd order). Similar results where also proven for the $SU(2)$ invariants with some restrictions on either the manifold $M$ or the order of the root of unity $\xi$ (see \cite{Ha}, \cite{BBL}, \cite{GM}, \cite{MR}). The full integrality result is conjectured and work in progress.
\\

The integrality results are based on a unification of the quantum (WRT) invariant.
For any integral homology $3$--sphere $M$, K. Habiro \cite{Ha} constructed a unified invariant $J_M$ whose evaluation at any root of unity coincides with the value of the quantum (WRT) $SU(2)$ invariant at that root. The unified invariant is an element of a certain cyclotomic completion of a polynomial ring, also known as the Habiro ring. This ring has beautiful properties. For example, we can think of its elements as analytic functions at roots of unity \cite{Ha}. Therefore, the unified invariant belonging to the Habiro ring means that the collection of the quantum (WRT) invariants is far from a random collection of algebraic integers: together they form a nice function. 
\\

In this thesis, we give a similar unification result for rational homology $3$--spheres which includes Habiro's result for integral homology $3$--spheres. More precisely, for a rational homology $3$--sphere $M$, we define unified invariants $I_M^{SO(3)}$ and $I_M^{SU(2)}$ such that the evaluation at a root of unity $\xi$ gives the corresponding quantum (WRT) invariant (up to some renormalization). In the $SU(2)$ case, we assume the order of the first homology group of the manifold to be odd -- the even case turns out to be quite different from the odd case and is part of ongoing research. Further, new rings, similar to the Habiro ring, are constructed which have the unified invariants as their elements. We show that these rings have similar properties to those of the Habiro ring. We also give a complete computation of the quantum (WRT) $SO(3)$ and $SU(2)$ invariants for lens spaces with a colored unknot inside at all roots of unity. 
\\

Additionally to the techniques developed by Habiro, we use deep results coming from number theory, commutative algebra, quantum group and knot theory. The new techniques developed in Chapters \ref{cyc} and \ref{RootsInSp} about cyclotomic completions of polynomial rings could be of separate interest for analytic geometry (compare \cite{Ma}), quantum topology, and representation theory. 
Further, even though integrality of the quantum (WRT) invariants does \emph{not} in general follow directly from the unification of the quantum (WRT) invariants, it does help proving it and a conceptual solution of the integrality problem is of primary importance for any attempt of a categorification of the quantum (WRT) invariants (compare \cite{Kho}).
Our results are also a step towards the unification of quantum (WRT) $\mathfrak{g}$ invariants of any semi--simple Lie algebra $\mathfrak{g}$ (see \cite{Tu} for a definition of quantum (WRT) $\mathfrak{g}$ invariants). K. Habiro and T. Le announced such unified $\mathfrak{g}$ invariants for integral homology 3--spheres. We expect that the techniques introduced here will help to generalize their results to rational homology 3--spheres.

\subsubsection*{Plan of the thesis} 
In Chapter \ref{ColoredJonesPolynomial}, we give the definition of the colored Jones polynomial and  state that it has a cyclotomic expansion with integral coefficients. The proof of this integrality result is postponed to the Appendix. This expansion is used for the definition of the unified invariant (Chapter \ref{UnifiedInvariant}). 
In Chapter \ref{WRTInvariant}, the quantum (WRT) invariants are defined	and important facts about (generalized) Gauss sums are stated. 
Chapter \ref{cyc} is devoted to the theory of cyclotomic completions of polynomial rings. For a given $b$, we define the rings $\cR_b$ and $\cS_b$ and discuss the evaluation at a root of unity in these rings.
In Chapter \ref{UnifiedInvariant}, the unified invariants $I_M^{SO(3)}$ and $I_M^{SU(2)}$ of a rational homology 3--sphere $M$ are defined and the main results of this thesis, i.e. the invariance of $I_M^{SO(3)}$ and $I_M^{SU(2)}$ and that their evaluation at a root of unity equals the corresponding quantum (WRT) invariant, are proven. Here we use (technical) results from Chapters \ref{LensSpaces} and \ref{laplace}.
In Chapter \ref{RootsInSp}, we prepare Chapters \ref{LensSpaces} and \ref{laplace} by showing that certain roots appearing in the unified invariants exist in the rings $\cR_b$ and $\cS_b$.
In Chapter \ref{LensSpaces}, we compute the quantum (WRT) invariants of lens spaces with a colored unknot inside and define the unified invariants of lens spaces.
In Chapter \ref{laplace}, we define a Laplace transform which we use to prove the main technical result of this thesis, namely that the unified invariant $I_M^{SO(3)}$ (respectively $I_M^{SU(2)}$) is indeed an element of $\cR_b$ (respectively of $\cS_b$), where $b$ is the order of the first homology group of the rational homology 3--sphere $M$.
\\

The material of Chapters \ref{ColoredJonesPolynomial} and \ref{WRTInvariant} is partly taken from \cite{Ha}, \cite{Li62}, \cite{KM} and \cite{BeBuLe}. Chapter \ref{cyc} includes results of Habiro \cite{Ha, Ha1}. The $SO(3)$ case of the results from Chapters \ref{cyc} to \ref{laplace} as well as the Appendix appeared in our joint paper with A. Beliakova and T. Le \cite{BeBuLe}. The $SU(2)$ case has not yet been published anywhere else.
\\

\subsubsection*{Acknowledgments}

First and foremost, I want to express my deepest gratitude to my supervisor Anna Beliakova.
Her encouragement, guidance as well as her way of thinking about mathematics influenced and motivated me throughout my studies. 

Further, I would like to thank Thang T. Q. Le for sharing with me his immense mathematical knowledge, his way of explaining, discussing and doing mathematics and for our joint research work. 

During my thesis, I spent seven months at the CTQM, University of Aarhus, in Denmark. Further, a part of my PhD was funded by the Forschungskredit of the University of Zurich as well as by the Swiss National Science Foundation.

\newpage

\pagestyle{fancy}
\setlength{\headheight}{15.4pt}

\renewcommand{\chaptermark}[1]{\markboth{\MakeUppercase{\chaptername}\ \thechapter. \ #1}{}}
\renewcommand{\sectionmark}[1]{\markright{\thesection. \ #1}{}}

\fancyhead{}
\fancyhead[RE]{\slshape \leftmark}
\fancyhead[LE]{\thepage}
\fancyhead[LO]{\slshape \rightmark}
\fancyhead[RO]{\thepage}
\fancyfoot{}
\renewcommand{\headrulewidth}{0pt} 
\renewcommand{\footrulewidth}{0pt}

\chapter{Colored Jones Polynomial}
 \label{ColoredJonesPolynomial}

In this chapter, we first recall some basic concepts of knot theory and quantum groups. We then define the universal $\mathfrak{sl}_2$ invariant of knots and links which leads us to the definition of the colored Jones Polynomial. In the last section, we state a generalization of Habiro's Theorem 8.2 of \cite{Ha} about a cyclotomic expansion of the colored Jones polynomial which we need for the definition of the unified invariant in Chapter \ref{UnifiedInvariant}. The proof of this theorem is postponed to the Appendix.
\\

Throughout this thesis, we will use the following notation. The $n$--dimensional sphere will be denoted by $S^n$, the $n$--dimensional disc by $D^n$ and the unit interval $[0,1]\subset\R$ by $I$. The boundary of a manifold $M$ is denoted by $\partial M$.
Except otherwise stated, a manifold $M$ is always considered to be closed, oriented and 3--dimensional.

\section{Links, tangles and bottom tangles}

A \emph{link} $L$ with $m$ \emph{components} in a manifold $M$  is an equivalence class by ambient isotopy of smooth embeddings of $m$ disjoint circles $S^1$ into $M$. A one--component link is called a \emph{knot}. The link is \emph{oriented} when an orientation of the components is chosen. 

A \emph{(rational) framing} of a link is an assignment of a rational number to each component of the link. It is called $integral$ when all numbers assigned are integral. A \emph{link diagram} of a framed link is a generic projection onto the plane as depicted in Figure \ref{fig:FramedLink}, where the framing is denoted by numbers next to each component.

\begin{figure}[!ht]
\centering%
\includegraphics[height=2.5cm]{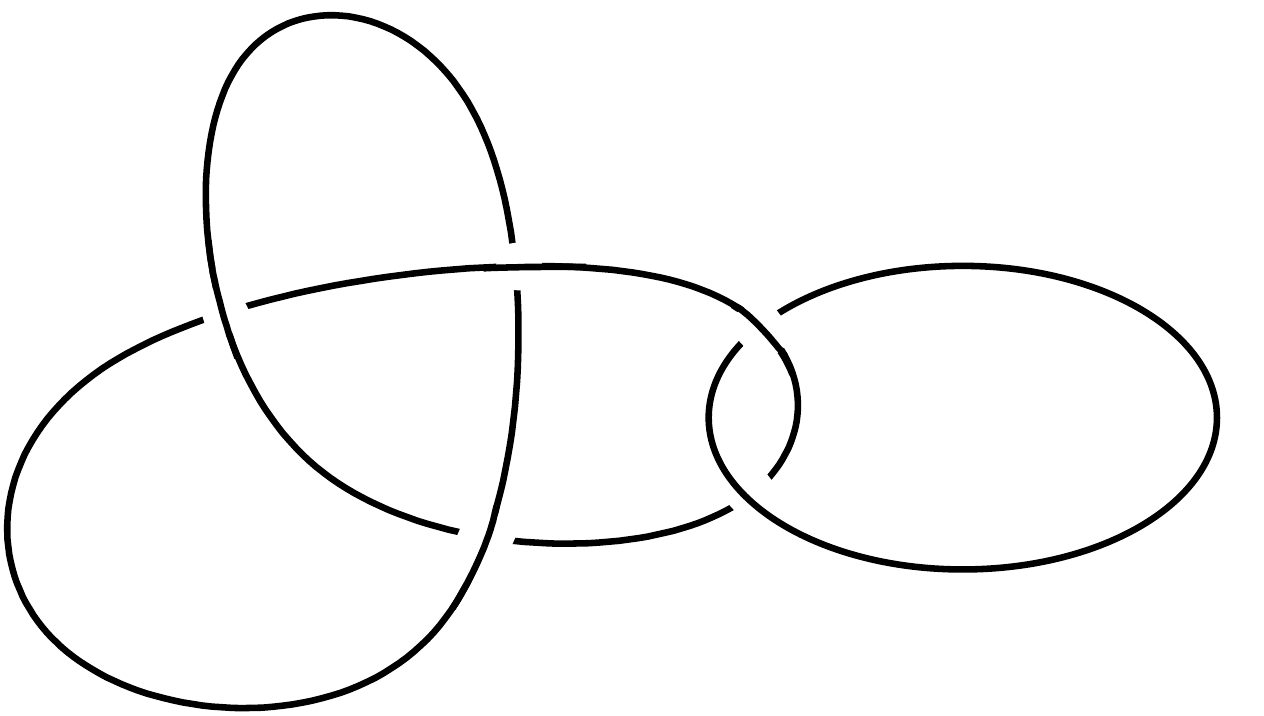}
\caption{A link diagram of a framed link.}
\label{fig:FramedLink}
\end{figure}


The \emph{linking number} of two components $L_1$ and $L_2$ of an oriented link $L$ is defined as follows. Each crossing in a link diagram of $L$ between $L_1$ and $L_2$ counts as $+1$ or $-1$, see Figure \ref{fig:LinkingNumber} for the sign. The sum of all these numbers divided by $2$ is called the linking number $\lk(L_1,L_2)$, which is independent of the diagram chosen for $L$.
The \emph{linking matrix} of a link $L$ with components $L_1,L_2,\ldots,L_n$ is a $n\times n$ matrix $(l_{ij})_{1\leq i,j\leq n}$ with the framings of the $L_i$'s on the diagonal and $l_{ij}=\lk(L_i,L_j)$ for $i\not=j$.
\\

\begin{figure}[!ht]
\begin{center}
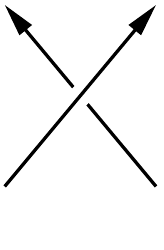\hspace{2cm}
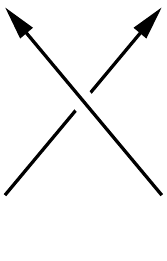
\caption{The assignment of $+1$ and $-1$ to the crossings.}
\label{fig:LinkingNumber}
\end{center}
\end{figure}

A \emph{tangle} $T$ is an equivalence class by ambient isotopy (fixing $\partial I^3\setminus \{\frac{1}{2}\}\times I\times \partial I$) of smooth embeddings of disjoint $1$--manifolds into the unit cube $I^3$ in $\R^3\subset S^3$ with $\partial T \subset \{\frac{1}{2}\}\times I\times \partial I$. We define $\partial_- T=T\cap (I^2\times \{0\})$ and $\partial_+ T=T\cap (I^2\times \{1\})$ and call $T$ a $(m,n)$--tangle if $m=|\partial_+ T|$ and $n=|\partial_- T|$, where $|M|$ denotes the number of connected components of $M$. Thus a link is a $(0,0)$--tangle. 

\begin{figure}[!ht]
\begin{center}
\includegraphics[height=3cm]{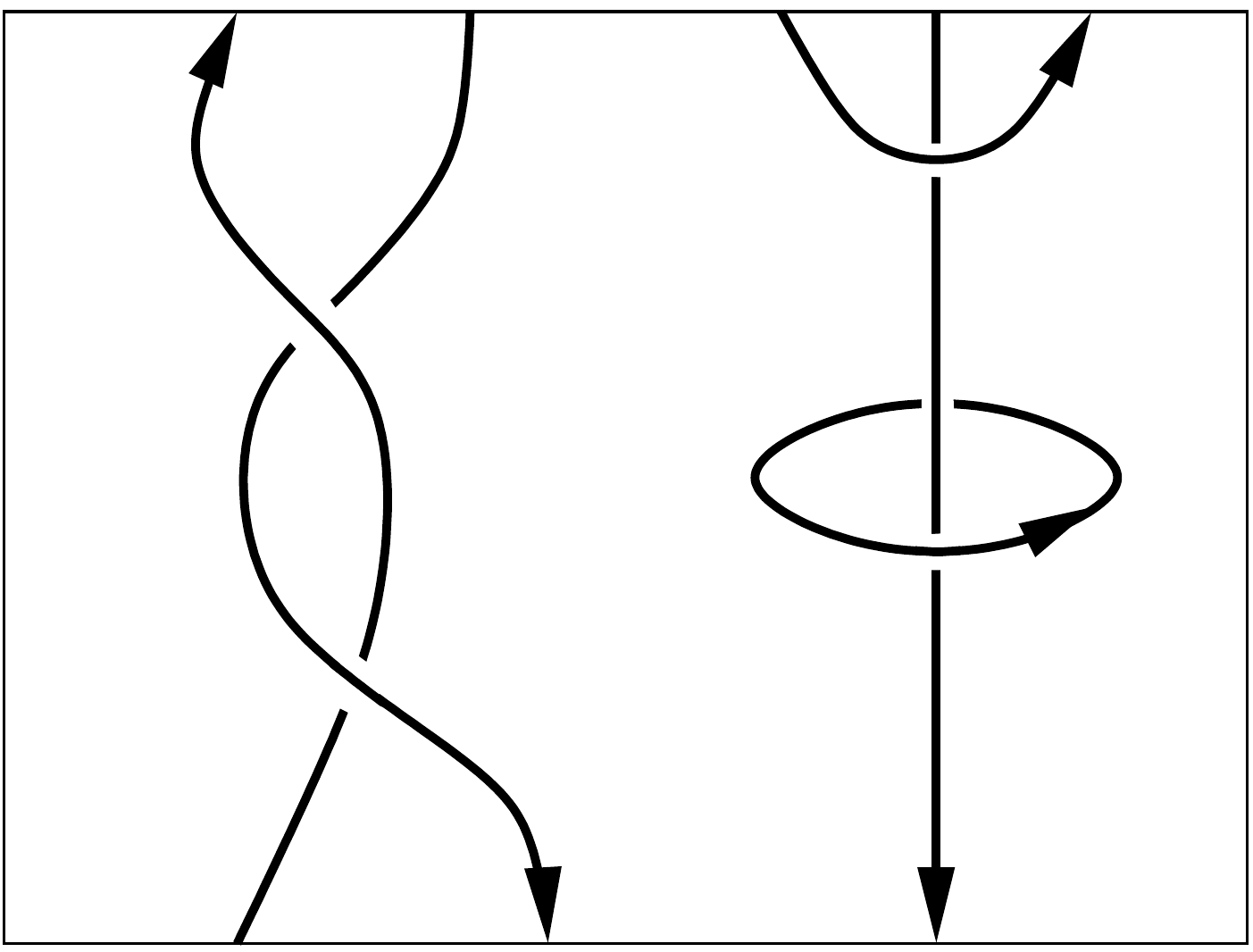}
\caption{A diagram of an oriented $(5,3)$--tangle.}
\label{fig:ExampleTangle}
\end{center}
\end{figure}

Framing, orientation and diagrams of tangles are defined analogously as for links. See Figure \ref{fig:ExampleTangle} for an example of a diagram of an oriented $(5,3)$--tangle. Every (oriented) tangle diagram can be factorized into the elementary diagrams shown in Figure \ref{fig:FundamentalTangles} using composition $\circ$ (when defined) and tensor product $\otimes$ as defined in Figure \ref{fig:CompositionAndTensorProductOfTangles}.
The oriented tangles can therefore be considered as the morphisms of a category $\cT$ with objects $x_1\otimes x_2\otimes\cdots\otimes x_n$, $x_i \in \{\uparrow, \downarrow\}$.
\\

\begin{figure}[!ht]
\begin{center}
\includegraphics[height=1.5cm]{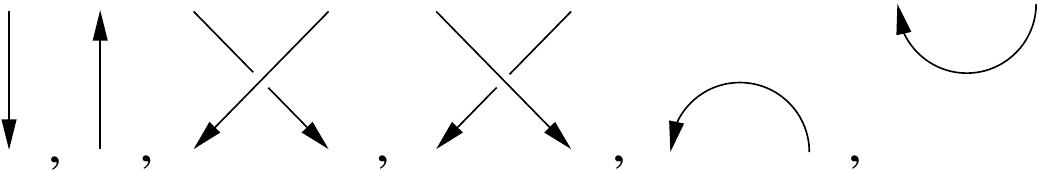}
\caption{The fundamental tangles.}
\label{fig:FundamentalTangles}
\end{center}
\end{figure}

\begin{figure}[!ht]
\begin{center}
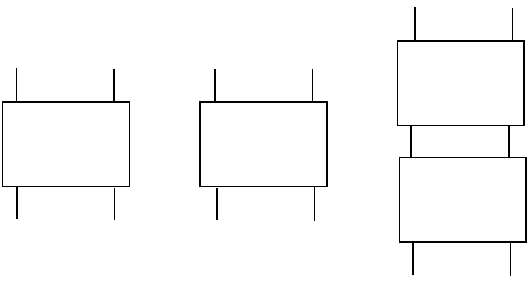\hspace{1.5cm}
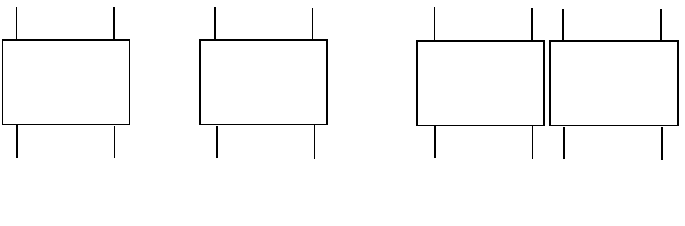
\caption{Composition and tensor product of tangle $T_1$ and tangle $T_2$.}
\label{fig:CompositionAndTensorProductOfTangles}
\end{center}
\end{figure}

In the cube $I^3$, we define the points $p_i:=\{\frac 12, \frac{i}{2n+1},0\}$ for $i=1,\ldots, 2n$, on the bottom of the cube. An \emph{$n$--component bottom tangle} $T=T_1\sqcup \dots \sqcup T_n$ is an oriented $(0,n)$--tangle consisting of $n$ arcs $T_i$ homeomorphic to $I$ and the $i$--th arc $T_i$ starts at point $p_{2i}$ and ends at $p_{2i-1}$. 
For an example, a diagram of the Borromean bottom tangle $B$ is given in Figure \ref{fig:borromeanTangle}. 

\begin{figure}[!ht]
\begin{center}
\includegraphics[height=3cm]{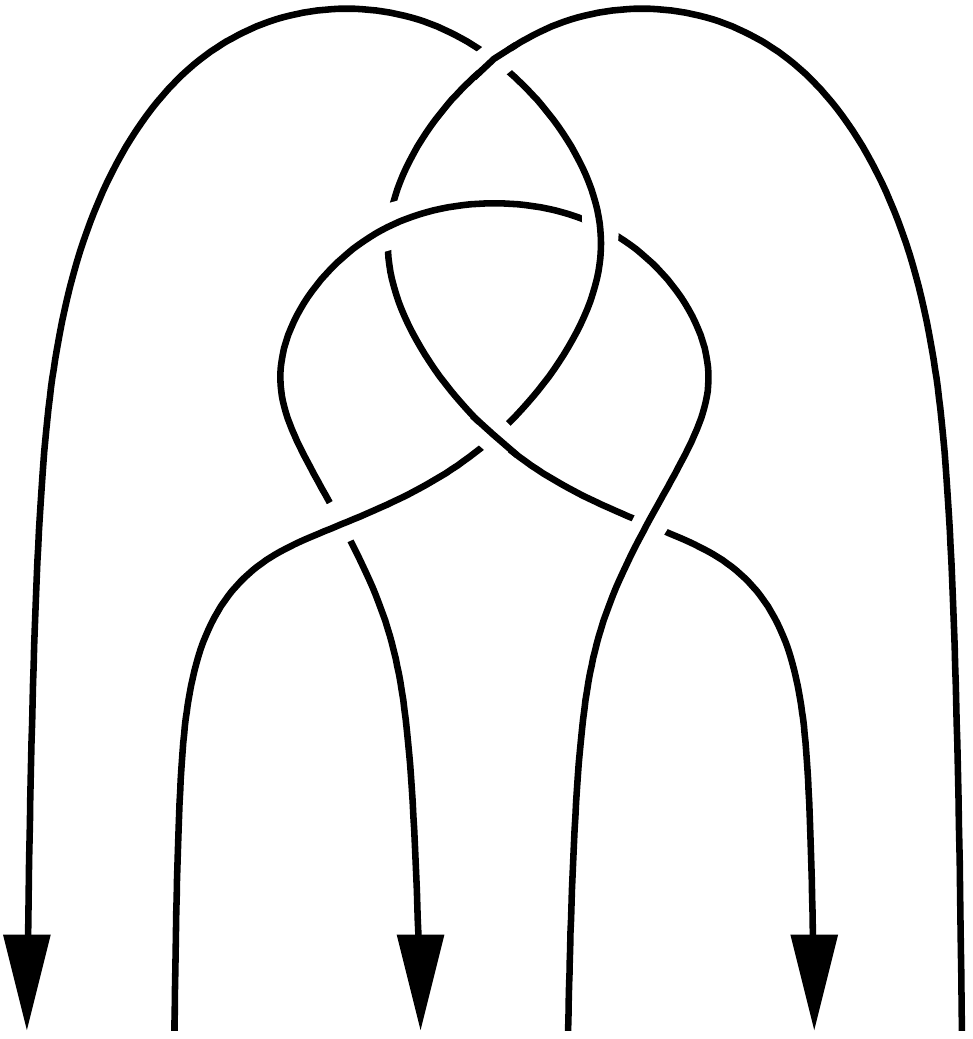}
\caption{Borromean bottom tangle $B$.}
\label{fig:borromeanTangle}
\end{center}
\end{figure}

The \emph{closure $\cl(T)$} of a bottom tangle $T$ is the $(0,0)$--tangle obtained by taking the composition of $T$ with the element  \includegraphics[height=0.3cm]{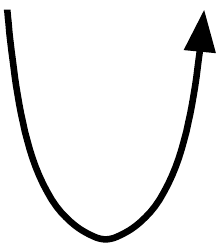}\;\includegraphics[height=0.3cm]{uArrow}
\;$\cdots$\;\includegraphics[height=0.3cm]{uArrow}. See Figure \ref{fig:borromeanLink} for an example.
\\

\begin{figure}[!ht]
\begin{center}
\includegraphics[height=3cm]{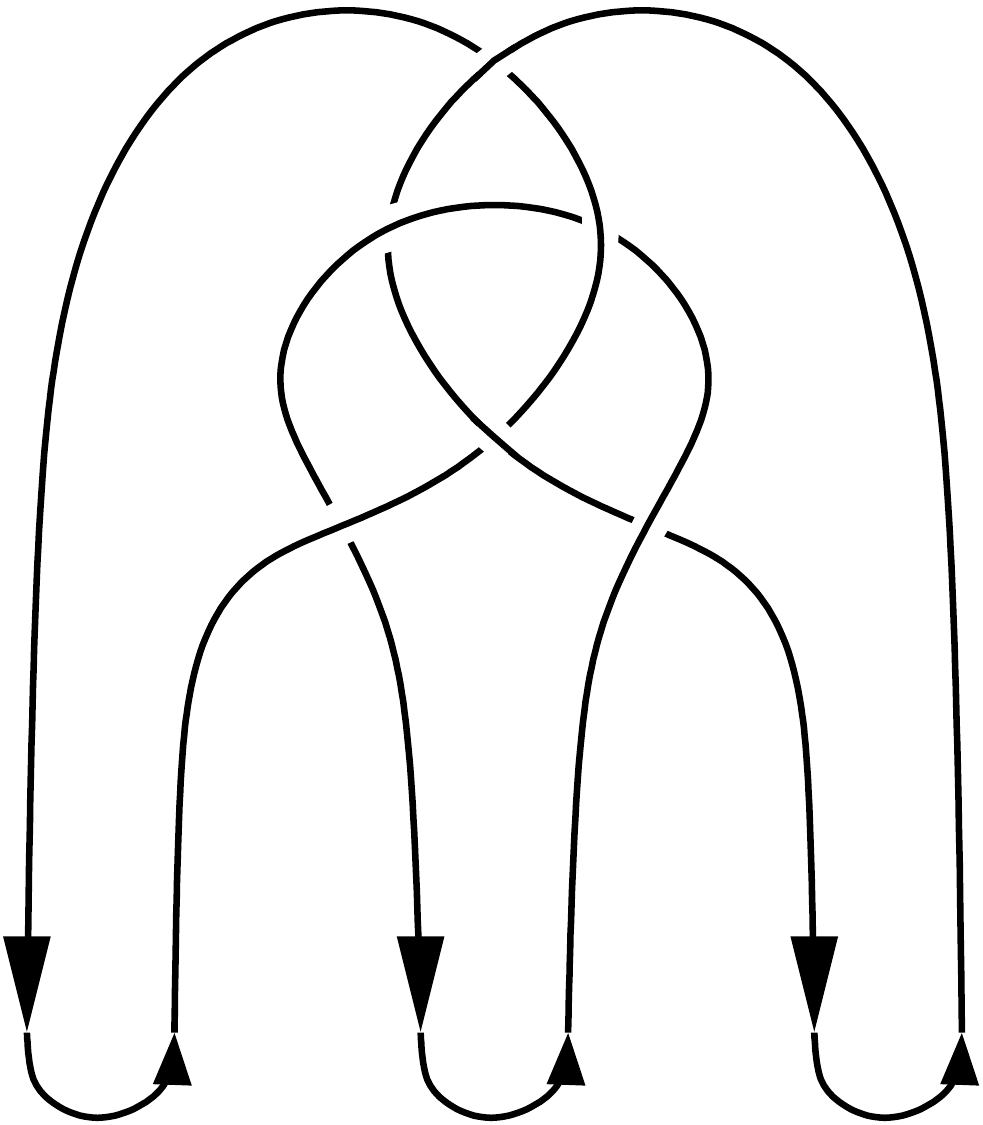}
\caption{The closure $\cl(B)$ of $B$.}
\label{fig:borromeanLink}
\end{center}
\end{figure}

In \cite{Ha-b}, Habiro defined a subcategory $\modB$ of the category of framed, oriented tangles $\cT$. The objects of $\modB $ are the symbols $\modb ^{\otimes n}$, $n\ge 0$,
where $\modb:=\downarrow \uparrow $. A morphism $X$ of $\modB$ is a $(m,n)$--tangle mapping $\modb^{\otimes m}$ to $\modb ^{\otimes n}$ for some $m,n\ge 0$. We can compose such a morphism with $m$--component bottom tangles 
to get $n$--component bottom tangles. Therefore, $\modB$ acts on the bottom tangles by composition. The category $\modB$ is braided: the monoidal 
structure is given by taking the tensor product of the tangles, the braiding for the generating object $\modb$  with itself is given by $\psi _{\modb ,\modb }=\includegraphics[height=1.2em]{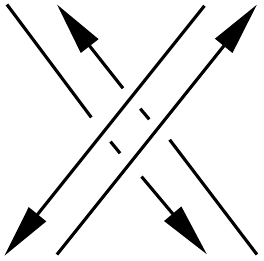}$.

\section{The quantized enveloping algebra $U_h(\mathfrak{sl}_2)$}\label{Uh(sl2)}

We follow the notation of \cite{Ha}.
We consider $h$ as a free parameter and let
\[
v=\exp \frac h2 \in \Q[[h]],\hsp\hsp q=v^2=\exp h,
\]
\[
\{n\} = v^{n}-v^{-n},
 \quad  \{n\}!=
\prod_{i=1}^n \{i\} ,\quad  [n] =\frac{\{n\}}{\{1\}}, \quad
\qbinom{n}{k} = \frac{\{n\}!}{\{k\}!\{n-k\}!}.
\]

The quantized enveloping algebra $U_h:=U_h(\mathfrak{sl}_2)$ is the
quantum deformation of the universal enveloping algebra
$U(\mathfrak{sl}_2)$ of the Lie algebra $\mathfrak{sl}_2$. More precisely, it is the $h$--adically complete $\Q[[h]]$--algebra generated by the elements $H,E$ and $F$ satisfying the relations
\[
HE-EH=2E, \hsp HF-FH=-2F,\hsp EF-FE=\frac{K-K^{-1}}{v-v^{-1}}
\]
where $K:=\exp\frac{hH}{2}$. It has a ribbon Hopf algebra structure with comultiplication $\Delta:U_h\to U_h\hat\otimes U_h$ (where $\hat\otimes$ denotes the $h$--adically complete tensor product), counit $\epsilon:U_h\to \Q[[h]]$ and antipode $S:U_h\to U_h$ defined by
\[
\begin{array}{lllll}
\Delta(H)=H\otimes 1+1\otimes H &\hsp& \epsilon(H)=0 &\hsp&S(H)=-H\\
\Delta(E)=E\otimes 1+K\otimes E &\hsp& \epsilon(E)=0 &\hsp&S(E)=-K^{-1}E\\
\Delta(F)=F\otimes K^{-1}+1\otimes F &\hsp& \epsilon(F)=0 &\hsp&S(H)=-FK.
\end{array}
\]
The universal $R$--matrix and its inverse are given by
\begin{eqnarray*}
R&=&D\left(\sum_{n\geq 0}v^{n(n-1)/2}\frac{(v-v^{-1})^n}{[n]!}F^n\otimes E^n\right)\\
R^{-1}&=&\left(\sum_{n\geq 0}(-1)^{n}v^{-n(n-1)/2}\frac{(v-v^{-1})^n}{[n]!}F^n\otimes E^n\right)D^{-1}
\end{eqnarray*}
where 
\[
D=\exp\left(\frac h4 H\otimes H\right).	
\]
We will use the Sweedler notation $R=\sum \alpha \otimes \beta$ and $R^{-1}=\sum \overline{\alpha}\otimes \overline{\beta}$ when we refer to $R$. 
As always, the ribbon element and its inverse can be defined via the $R$--matrix and the associated grouplike element $\kappa \in U_h$ satisfies $\kappa=K^{-1}$.
\\

By a \emph{finite-dimensional representation} of $U_h$, we mean a left $U_h$--module which is free of finite rank as a $\Q[[h]]$--module. For each $n\geq 0$, there exists exactly one irreducible finite--dimensional representation $V_n$ of rank $n+1$ up to isomorphism. It corresponds to the $(n+1)$--dimensional irreducible representation of the Lie algebra $\mathfrak{sl}_2$.

The structure of $V_n$ is as follows. Let $\bv_0^n \in V_n$ denote a highest weight vector of $V_n$ which is characterized by $E\bv_0^n=0, H\bv_0^n=n\bv_0^n$ and $U_h \bv_0^n=V_n$. Further we define the other basis elements of $V_n$ by 
$\bv_{i}^n:= \frac{F^{i}}{[i]!} \bv_0^n$ for $i=1,\ldots,n$. 
Then the action $\rho_{V_n}$ of $U_h$ on $V_n$ is given by
\[
H \bv_i^n = (n-2i) \bv_i^n,\;\;F \bv_i^n=[i+1]\bv_{i+1}^n,\;\; E \bv_i^n=[n+1-i]\bv_{i-1}^n
\]
where we understand $\bv_i^n=0$ unless $0\leq i\leq n$. It follows that $K^{\pm 1} \bv_i^n =v^{\pm(n-2i)}\bv_i^n$.

If $V$ is a finite--dimensional representation of $U_h$, then the \emph{quantum trace} $\tr_q^V(x)$ in $V$ of an element $x\in U_h$ is given by
\[
 \tr_q^V(x)=\tr^V(\rho_V(K^{-1}x)) \in \Q[[h]]
\]
where $\tr^V:\End(V)\to \Q[[h]]$ denotes the trace in $V$.

\section{Universal $\mathfrak{sl}_2$ invariant}

For every ribbon Hopf algebra exists a \emph{universal invariant} of links and tangles from which one can recover the operator invariants such as the colored Jones polynomial. Such universal invariants have been studied by Kauffman, Lawrence, Lee, Ohtsuki, Reshetikhin, Turaev and many others, see \cite{ Ha}, \cite{Ohtsukibook}, \cite{Tu} and the references therein. Here we need only the case of bottom tangles.

Let $T=T_1\sqcup T_2\sqcup \ldots \sqcup T_n$ be an ordered oriented $n$--component framed bottom tangle. We define the universal $\mathfrak{sl}_2$ invariant $J_T\in U_h^{\hat\otimes n}$ as follows. We choose a diagram for $T$ which is obtained by composition and tensor product of fundamental tangles (see Figure \ref{fig:FundamentalTangles}). On each fundamental tangle, we put elements of $U_h$ as shown in Figure \ref{fig:AssignmentToFundamentalTangles}. Now we read off the elements on the $i$--th component following its orientation. Writing down these elements from right to left gives $J_{(T_i)}$. This is the $i$--th tensorand of the universal invariant $J_T=\sum J_{(T_1)}\otimes J_{(T_2)}\otimes \ldots \otimes J_{(T_n)}$. Here the sum is taken over all the summands of the $R$--matrices which appear. The result of this construction does not depend on the choice of diagram and defines an isotopy invariant of bottom tangles. 

\begin{figure}[!ht]
\begin{center}
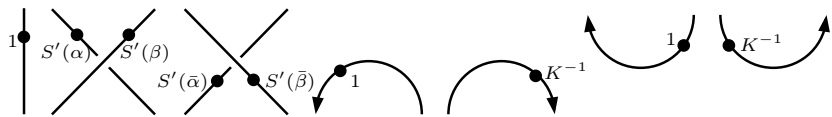
\caption{Assignment to the fundamental tangles. Here $S'$ should be replaced with the identity map if the string is oriented downward and by $S$ otherwise.}
\label{fig:AssignmentToFundamentalTangles}
\end{center}
\end{figure}

\begin{example}
For the Borromean tangle $B$, the assignment of elements of $U_h$ to $B$ are shown in Figure \ref{fig:ExampleUniversalInvariant}. The universal invariant is given by
\[
J_B=\sum
S(\alpha_6)S(\beta_5)S(\overline{\alpha_3})S(\beta_1)\otimes
\alpha_1 \beta_4 S(\alpha_5)S(\beta_2)\otimes
\alpha_2\overline{\beta_3}\alpha_4\beta_6
\]
where we use the Sweelder notation, i.e. we sum over all $\alpha_i,\beta_i$ for $i=1,\ldots,6$.
Compare also with \cite[Proof of Corollary 9.14]{Ha-b} and \cite[Proof of Theorem 4.1]{Ha}. Habiro uses $(S\otimes S)R=R$ and $R^{-1}=(1\otimes S^{-1})R$ therein.
\end{example}

\begin{figure}[!ht]
\begin{center}
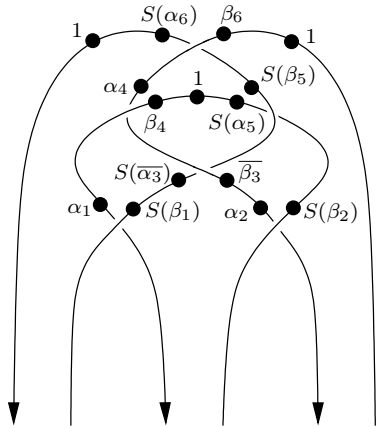
\caption{The assignments to the Borromean tangle.}
\label{fig:ExampleUniversalInvariant}
\end{center}
\end{figure}

\section{Definition of the colored Jones Polynomial}

Let $L=L_1\sqcup L_2\sqcup \ldots\sqcup L_m$ be an $m$--component framed oriented ordered link with associated positive integers $n_1,\ldots,n_m$ called the \emph{colors associated with $L$}. Remember that the $n$--dimensional representation of $U_h$ is denoted by $V_{n-1}$. Let further $T$ be a bottom tangle with $\cl(T)=L$. 
The \emph{colored Jones polynomial}  of $L$ with colors $n_1,\ldots,n_m$ is given by
\[
J_L(n_1,\ldots,n_m)=(\tr_q^{V_{n_1-1}}\otimes \tr_q^{V_{n_2-1}}\otimes\ldots\otimes\tr_q^{V_{n_m-1}})(J_T).
\]
For every choice of $n_1,\ldots,n_m$, this is an invariant of framed links (see e.g. \cite{RT1} and \cite[Section 1.2]{Ha-b}).

\begin{example}\label{unknot}
 Let us calculate $J_U(n)$, where $U$ denotes the unknot with zero framing. For $T=\includegraphics[height=0.8cm]{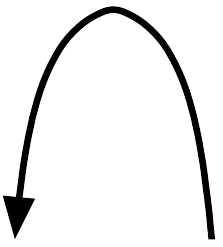}$, we have $\cl(T)=U$. We choose for $V_{n-1}$ the basis $\bv_0^{n-1}, \bv_1^{n-1},\ldots,\bv_{n-1}^{n-1}$ described in Section \ref{Uh(sl2)}. Since 
$J_T=1$, we have
\begin{eqnarray*}
J_U(n)&=&\tr^{V_{n-1}}(\rho_{V_{n-1}}(K^{-1}))=\tr^{V_{n-1}}\left(\left(\begin{array}{ccccc}v^{-n+2} & 0&\ldots&\ldots&0\\0&v^{-n+4}&0&\cdots&0\\ \vdots&&\ddots&&\vdots\\ 0&\ldots&\ldots&0&v^{n-2}\end{array}\right)\right) \\
&&\\
&=& v^{n-2}+v^{n-4}+\ldots+v^{-n+2}=[n].
\end{eqnarray*}
\end{example}

We will need the following two important properties of the colored Jones polynomial.
\begin{lemma}\cite[Lemma 3.27]{KM}\label{framing}

If $L_1$ is obtained from $L$
by increasing the framing of the $i$th component by 1, then
\[
J_{L_1}(n_1,\dots,n_m) = q^{(n_i^2-1)/4} J_{L}(n_1,\dots,n_m).
\]
\end{lemma}

\begin{lemma}\cite[Strong integrality Theorem 2.2 and Corollary 2.4]{LeDuke}

There exists a number $p\in \Z$, depending only on the linking matrix of $L$, such that \linebreak $J_{L}(n_1,\dots,n_m) \in q^{\frac{p}{4}}\Z[q^{\pm 1}]$. Further, if all the colors $n_i$ are odd, $J_{L}(n_1,\dots,n_m) \in \Z[q^{\pm 1}]$.
\end{lemma}

\section{Cyclotomic expansion of the colored Jones polynomial}\label{CyclotomicJones}
Let  $L$ and $L'$ have $m$ and $l$ components. Let us color $L'$ by fixed $\bj=(j_1,\dots,j_l)$ and vary the colors $\bn=(n_1,\dots,n_m)$ of $L$.

For non--negative integers $n,k$ we define
\[ 
A(n,k) := \frac{\prod^{k}_{i=0}
\left(q^{n}+q^{-n}-q^i -q^{-i}\right)}{(1-q) \, (q^{k+1};q)_{k+1}}
\]
where we use from $q$--calculus the definition
\[ 
(x;q)_n := \prod_{j=1}^n (1-x q^{j-1}).
\]
For $\bk=(k_1,\dots,k_m)$ let
\[ 
A(\bn,\bk):= \prod_{j=1}^m \; A(n_j,k_j).
\]
Note that $A(\bn,\bk)=0$ if $k_j \ge n_j$ for some index $j$. Also 
\[ 
A(n,0)= q^{-1} J_U(n)^2. 
\]

The colored Jones polynomial $J_{L\sqcup L'} (\bn, \bj)$, when $\bj$ is fixed, can be repackaged into the invariant $C_{L\sqcup L'} (\bk, \bj)$ as stated in the following theorem.

\begin{theorem}\label{GeneralizedHabiro} 
Suppose $L\sqcup L'$ is a link in $S^3$, with $L$ having 0 linking matrix. Suppose the components of $L'$ have fixed {\em odd} colors $\bj = (j_1,\dotsm j_l)$. Then  there are invariants
\begin{equation}\label{Jones2}
C_{L\sqcup L'}(\bk,\bj) \in \frac{(q^{k+1};q)_{k+1}}{(1-q)}
\,\,\Z[q^{\pm 1}] \quad \text{where  $k=\max\{k_1,\dots, k_m\}$}
\end{equation}
such that for every $\bn =(n_1,\dots, n_m)$
\begin{equation}\label{Jones}
 J_{L\sqcup L'}
(\bn, \bj)  \,  \prod^m_{i=1}\; [n_i] = 
\sum_{0\le k_i \le n_i-1}C_{L\sqcup L'}(\bk,\bj)\;  A(\bn, \bk).
\end{equation}
\end{theorem}

When $L'=\emptyset$, this was proven by K. Habiro, see Theorem 8.2 in \cite{Ha}. This generalization can be proved similarly as in \cite{Ha}. For completeness, we give a proof in the Appendix. Note that the existence of $C_{L\sqcup L'}(\bk,\bj)$ as rational functions in $q$ satisfying \eqref{Jones} is easy to establish. The difficulty here is to show the integrality of \eqref{Jones2}.

\begin{remark}\label{FiniteSum}
Since $A(\bn, \bk) =0$ unless $ \bk < \bn$, in the sum on the right hand side of \eqref{Jones} one can assume that $\bk$ runs over the set of all $m$--tuples $\bk$ with non--negative integer components. We will use this fact later.
\end{remark}

\newpage

\chapter{Quantum (WRT) invariant}\label{WRTInvariant}

In this chapter, we describe in Section \ref{Surgery} a one--to--one correspondence between 3--dimensional manifolds up to orientation preserving homeomorphisms and links up to Fenn--Rourke moves. We then state in Section \ref{GaussSums} results about generalized Gauss sums and define a variation therefrom. We use this in Section \ref{QuantumInvariant} where we give the definition of the quantum (WRT) invariants and, for rational homology 3--spheres, a renormalization of these invariants. Finally, we describe the connection between the quantum (WRT) $SU(2)$ and $SO(3)$ invariants.

\section{Surgery on links in $S^3$}\label{Surgery}

Let $K$ be a knot in $S^3$ and $N(K) = K\times D^2$ its tubular neighbourhood. The \emph{knot exterior} $E$ is defined as the closure of $S^3\backslash N(K)$. 

A 3--manifold $M$ is obtained from $S^3$ by a \emph{rational 1--surgery} along a framed knot $K\subset S^3$ with framing $\frac{p}{q}$, when $N(K)$ is removed from $S^3$ and a copy of $D^2\times S^1$ is glued back in using a homeomorphism $h: \partial D^2\times S^1 \to \partial E$.
If $q=1$, the $1$--surgery is called \emph{integral}.
The homeomorphism $h$ is completely determined by the image of any meridian $m:=\partial D^2 \times \{*\}$ of $\partial D^2\times S^1$. To describe this image it is enough to specify a canonical longitude $l$ of $\partial E$ and an orientation on $m$ and $l$. The image will then be a simple closed curve on $\partial E$ isotopic to a curve of the form $c=p \cdot m+ q\cdot l$, where $p$ and $q$ are given by the framing of the knot. 
The canonical longitude $l$ is, up to isotopy, uniquely defined as the curve homologically trivial in $E$ and with $\lk(l,K)=0$. For the orientation on $m$ and $l$, we choose the standard orientation on $S^3$ which induces an orientation on $E$. The two curves $m$ and $l$ are then oriented such that the triple $\langle m,l,n\rangle$ is positively oriented. Here $n$ is a normal vector to $\partial E$ pointing inside $E$, see Figure \ref{fig:MeridianLongitudeNormal}.

\begin{figure}[ht]
\begin{center}
 \includegraphics[height=3cm]{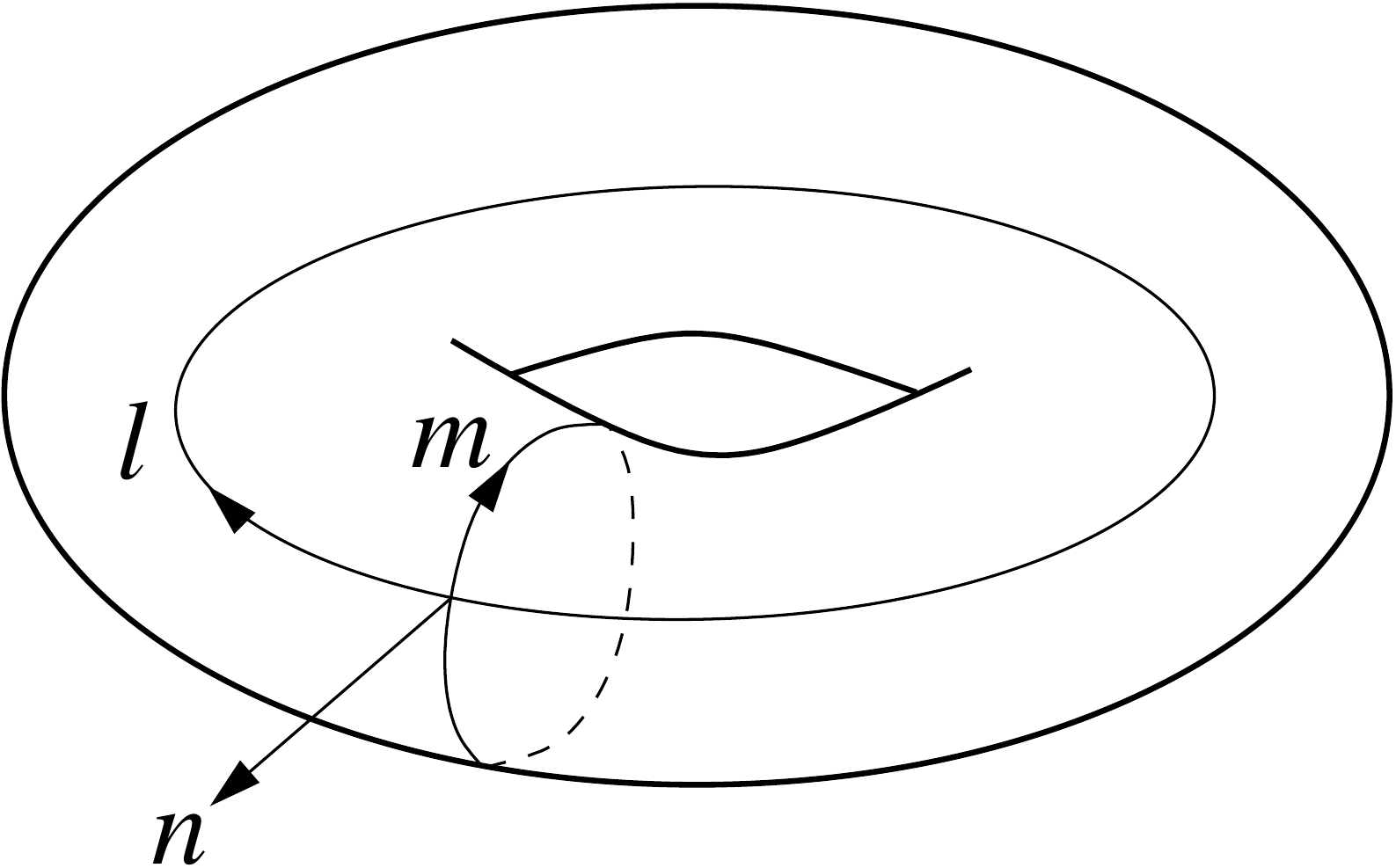}
 \caption{Orientation of meridian $m$ and longitude $l$ in $\partial E$.} 
\end{center}
\label{fig:MeridianLongitudeNormal}
\end{figure}

\begin{theorem}[Likorish, Wallace]
 Any closed connected orientable 3--manifold $M$ can be obtained from $S^3$ by a collection of integral 1--surgeries.
\end{theorem}

\begin{proof}
See for example \cite{Li62} or \cite{Wallace}. 
\end{proof}

Therefore any ordered framed link $L$ gives a description of a collection of 1-surgeries and the manifold obtained in this way will be denoted by $M=S^3(L)$.
Let $L'$ be an other link in $S^3$. Surgery along $L$ transforms $(S^3,L')$ into  $(M,L')$. We use the same notation $L'$ to denote the link in $S^3$ and the corresponding one in $M$.

\begin{example}
 The $(b,a)$ lens space $L(b,a)$ is obtained by surgery along an unknot with rational framing $\frac{b}{a}$. Further we have $S^3=S^3(U^1)$, where $U^1$ denotes the unknot with framing 1. 
\end{example}

In \cite{Kirby}, R. Kirby proved a one--to--one correspondence between 3--manifolds up to homeomorphisms and framed links up to the two so--called Kirby moves. In \cite{FR}, R. Fenn and C. Rourke showed that these two moves are equivalent to the one Fenn--Rourke move (see Figure \ref{fig:FennRourke}) and proved the following.

\begin{theorem}[Fenn--Rourke]
 Two framed links in $S^3$ give, by surgery, the same oriented 3--manifold if and only if they are related by a sequence of Fenn Rourke moves. A Fenn Rourke move means replacing in the link locally $T$ by $T_+$ or $T_-$ as shown in Figure \ref{fig:FennRourke} where the non--negative integer $m$ can be chosen arbitrary.
The framings $j$ and $j_{\pm}$ on corresponding components $J$ and $J_{\pm}$ (before and after a move) are related by $j_{\pm}=j\pm \lk(K,J_{\pm})^2$.

\begin{figure}[h]
\begin{center}
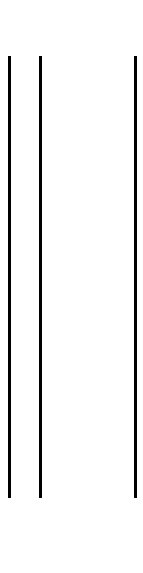 \hspace{2.2cm}
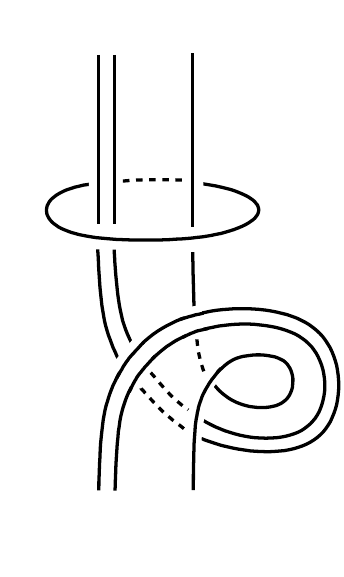 \hspace{1.4cm}
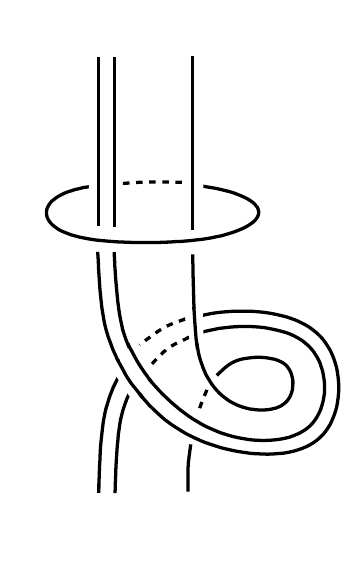
\caption{The positive and the negative Fenn--Rourke move.}
\label{fig:FennRourke}
\end{center}
\end{figure}

\end{theorem}

\begin{proof}
 See \cite{FR}.
\end{proof}

\section{Gauss sums}\label{GaussSums}

We use the following notation. The greatest common divisor of two integers $a$ and $b$ is denoted by $(a,b)$. If $a$ does (respectively does not) divide $b$, we write $a \mid b$ (respectively $a\nmid b$). \\

Further, for $b$ odd and positive, the Jacobi symbol, denoted by $\left(\frac{a}{b}\right)$, is defined as follows. First, $\left(\frac{a}{1}\right)=1$. Then for $p$ prime, $\left(\frac{a}{p}\right)$ represents the Legendre symbol, which is defined for all integers $a$ and all odd primes $p$ by
\[
    \left(\frac{a}{p}\right) = \begin{cases} \;\;\,0\mbox{ if } a \equiv 0 \pmod{p} \\+1\mbox{ if }a \not\equiv 0\pmod{p} \mbox{ and for some integer }x, \;a\equiv x^2\pmod{p} \\-1\mbox{ else.} \end{cases}
\]
Finally,  if $b\not=1$, we put
\[
     \left(\frac{a}{b}\right) = \prod_{i=1}^{m} \left(\frac{a}{p_i}\right)^{\alpha_i} \quad\text{where}\quad b=\prod_i^{m}p_i^{\alpha_i} \quad\text{for } p_i \text{ prime}.
\]

Let $e_r:=\exp(\frac{2\pi i}{r})$ where $i$ denotes the positive primitive $4$th root of $1$.
The generalized Gauss sum is defined as
\[
G(r,x,y):=\sum_{j=0}^{r-1} e_r^{xj^2+yj}.
\]
The values of $G(r,x,y)$ are well known:

\begin{lemma}\label{GaussSumFormulas}
For $r,x,y\in \N$ we have
\[
G(r,x,y)=\begin{cases}
		0 &\text{if } (r,x) \nmid y\\
		(r,x)\cdot G\left(\frac{r}{(r,x)},\frac{x}{(r,x)},\frac{y}{(r,x)}\right) &\text{else}
          \end{cases}
\]
and for $(r,x)=1$ 
\begin{eqnarray*}
 G(r,x,y)=\begin{cases}
           \epsilon(r)\left(\frac{x}{r}\right)\sqrt{r}e_r^{-\frac{x_{*r}y^2}{4}(r+1)^2} & \text{\, if } r \text{ odd}\\
	   0 &\begin{array}{l}
		\text{if } r\equiv 2 \pmod{4} \text{ and } y \text{ even, or}\\
		\text{if } r\equiv 0\pmod{4} \text{ and } y \text{ odd}
	      \end{array}\\
           \epsilon\left(\frac{r}{2}\right)\left(\frac{2x}{r/2}\right)\sqrt{2r}e_r^{-\frac{x_{*r}y^2}{4}\left(\frac{r+2}{2}\right)^3}&\text{\, if }r\equiv 2\pmod{4} \text{ and } y \text{ odd}\\
           \overline{\epsilon(x)}\left(\frac{r}{x}\right)(1+i)\sqrt{r}e_r^{-\frac{x_{*r}y^2}{4}}&\text{\, if }r\equiv 0\pmod{4} \text{ and } y \text{ even}
          \end{cases}
\end{eqnarray*}
where $x_{*r}$ is defined such that $xx_{*r}\equiv 1 \pmod{r}$ and $\epsilon(x)=1$ if $x\equiv 1\pmod{4}$ and $\epsilon(x)=i$ if $x\equiv 3\pmod{4}$.
\end{lemma}

\begin{proof}
See e.g. \cite{Lang} or any other text book in basic number theory.
\end{proof}

Let $R$ be a unitary ring and $f(q;n_1,\dots,n_m)\in R[q^{\pm 1}, n_1,n_2,\ldots,n_m]$.
For each root of unity $\xi$ of order $r$, in quantum topology, the following sum plays an important role:
\[
{\sum_{n_i}}^{\xi, G} f := \sum_{n_i\in N_G}  f(\xi; n_1,\dots, n_m).
\]
Here $G$ stands for either the Lie group $SU(2)$ or the Lie group $SO(3)$ and $N_{SU(2)} :=\{n\in\Z\mid 0\leq n\leq 2r-1\} $ and $N_{SO(3)}:=\{n\in\Z\mid 0\leq n\leq 2r-1, n\text{ odd}\}$. If $G=SO(3)$, $r$ is always assumed to be odd.
\\

Let us explain the meaning of the $N_G$'s. Roughly speaking, the set $N_G$ corresponds to the set of irreducible representations of $U_q(\mathfrak{sl}_2)$ where $q$ is a root of unity of order $r$.

In fact, the quantum invariants (see next Section \ref{QuantumInvariant}) were originally defined by N. Re\-she\-ti\-khin and V. Turaev in \cite{RT1, RT2} by summing over all irreducible representations of the quantum group $U_q(\mathfrak{sl}_2)$ where $q$ is chosen to be a root of unity of order $r$. The quantum group $U_q(\mathfrak{sl}_2)$ is defined similarly as $U_h(\mathfrak{sl}_2)$ (see Section \ref{Uh(sl2)}) where $q$ stands for $\exp(h)$. Similarly as for $U_h(\mathfrak{sl}_2)$, there exists exactly one free finite--dimensional irreducible representation of $U_q(\mathfrak{sl}_2)$ in each dimension. In the case when $q$ is chosen to be a primitive $r$th root of unity, only the representations of dimension $\leq r$ are irreducible (see e.g. \cite[Theorem 2.13]{KM}). Since for every irreducible representation of $U_q(\mathfrak{sl}_2)$ there is a corresponding irreducible representation of the Lie group $SU(2)$, the invariant is sometimes also called the \emph{quantum (WRT) $SU(2)$ invariant}.

Kirby and Melvin showed in \cite[Theorem 8.10]{KM}, that summing over all irreducible representations of $U_q(\mathfrak{sl}_2)$ of \emph{odd} dimension also gives an invariant.
Since the Lie group $SO(3)$ is isomorphic to $SU(2)/\{\pm I\}$, where $I$ stands for the identity matrix, a representation $\rho_{n-1}$ of dimension $n-1$ of $SU(2)$ is a representation of $SO(3)$ if and only if $\rho_{n-1}(-I)$ is the identity map. This is true if and only if $n$ is odd. Therefore it makes sense to call the invariant introduced by Kirby and Melvin the \emph{quantum (WRT) $SO(3)$ invariant}. 

As already mentioned above, if  $q$ is chosen to be a root of unity of order $r$, the irreducible representations of $U_q(\mathfrak{sl}_2)$ are actually of dimension $\leq r$, and not $\leq 2r-1$ which we use as upper bound in $N_G$. But summing up to $2r-1$ makes all calculations much simpler and due to the first symmetry principle of Le \cite[Theorem 2.5]{LeDuke}, summing up to $r$ or up to $2r-1$ does change the invariant only by some constant factor. 
\\

For $\xi$ a root of unity, we define the following  variation  of the Gauss sum:
\[
\gamma^G_b(\xi):= {\sum_{n}}^{\xi,G} q^{b\frac{n^2-1}{4}}.
\]
Notice, that for $G=SO(3)$, $\gamma^{SO(3)}_b(\xi)$ is well--defined in $\Z[q]$ since, for odd $n$, $4\mid n^2-1$.
In the case $G=SU(2)$, the Gauss sum is dependent on a $4$th root of $\xi$ which we denote by $\xi^{\frac{1}{4}}$.
\\

For an arbitrary primitive $r$th root of unity $\xi$, we define the Galois transformation
\begin{eqnarray*}
\varphi: \Q(e_r)&\to& \Q(\xi)\\
e_r&\mapsto&\xi
\end{eqnarray*}
which is a ring isomorphism.

\begin{lemma}\label{GammabNonzero}
Let $\xi$ be a primitive $r$th root of unity and $b\in\Z$. The following holds.
\begin{eqnarray*}
\gamma_b^{SO(3)}(\xi) &=& \varphi(G(r,b,b))\\
\gamma_b^{SU(2)}(\xi) &=& \xi^{\frac{-b}{4}} \varphi(G(r,b,0))+\varphi(G(r,b,b)).
\end{eqnarray*}
In particular, $\gamma_b^{SU(2)}(\xi)$ is zero for $\frac{r}{(r,b)}$ odd, $\frac{b}{(r,b)}\equiv 2\pmod{4}$ and $\gamma_b^{G}(\xi)$ is nonzero in all other cases.
\end{lemma}

\begin{proof}
It is enough to prove the claim at the root of unity $e_r$ and then apply the Galois transformation to get the general result.

For $G=SO(3)$, we have $r$ odd and 
\[
\gamma_b^{SO(3)}(e_r)
=\sum_{\substack{n=0\\n \text{ odd}}}^{2r-1}e_r^{b \frac{n^2-1}{4}}
=\sum_{n=0}^{r-1} e_r^{b(n^2+n)}=G(r,b,b)
\]
which is nonzero for all $b$ and odd $r$.

For $G=SU(2)$, we split the sum into the even and the odd part and get
\begin{eqnarray*}
\gamma_b^{SU(2)}(e_r)
&=&\sum_{n=0}^{2r-1}e_r^{b\frac{n^2-1}{4}}
=\sum_{n=0}^{r-1} e_r^{b(n^2+n)} + e_r^{-\frac{b}{4}}\sum_{n=0}^{r-1}e_r^{bn^2}
=G(r,b,b)+e_r^{-\frac{b}{4}}G(r,b,0)\\
&=&c\cdot\begin{cases} 
    \epsilon(r')\left(\frac{b'}{r'}\right)(e_{r'}^{-\frac{b'}{4}}+e_{r'}^{-\frac{b'}{4}(r'+1)^2})
				&\text{if } r' \text{ odd}\\
    G(r',b',b')                   &\text{if } r' \equiv 2\pmod{4} \\
    e_{r'}^{-\frac{b'}{4}}G(r',b',0) &\text{if } r' \equiv 0\pmod{4}
   \end{cases}
\end{eqnarray*}
where $c=(r,b)$ and $r'=\frac{r}{c}$ and $b'=\frac{b}{c}$. Therefore, $\gamma_b^{SU(2)}(e_r)$ can only be zero if $r'$ odd and $e_{r'}^{-\frac{b'}{4}}+e_{r'}^{-\frac{b'}{4}(r'+1)^2}$ equal zero, i.e. $e_{r'}^{-\frac{b'}{4}r'(r'+2)}=-1$. Since $e_{r'}^{\frac{r'}{4}}$ is a primitive $4$th root of unity, this is true if and only if $b'\equiv 2\pmod{4}$.
\end{proof}

\begin{example}
For $b=1$, we have
\be\label{Gamma1SO3}
 \gamma^{SO(3)}_1(e_r)=G(r,1,1)=\epsilon(r)\sqrt{r} e_r^{-4_*}
\ee
where $4\cdot 4_*\equiv 1\pmod{r}$. Further, for the $SU(2)$ case, fixing the $4$th root of $e_r$ as $e_r^{\frac{1}{4}}:=e_{4r}$, we get
\be\label{Gamma1SU2}
 \gamma^{SU(2)}_1(e_r)=e_{4r}^{-1}G(r,1,0)+G(r,1,1)=(1+i)\sqrt{r}e_{4r}^{-1}.
\ee
\end{example}

\section{Definition of the quantum (WRT) invariant}\label{QuantumInvariant}
Suppose the components of $L'$ are colored by fixed integers $j_1,\dots,j_l$. Let
\be\label{F}
 F^G_{L\sqcup L'}(\xi):= {\sum_{n_i}}^{\xi,G}\;
\left \{  J_{L\sqcup L'}(n_1,\dots,n_m, j_1,\dots,j_l)\prod_{i=1}^m [n_i]\right \}.
\ee

\begin{example}\label{TauNonZero}
An important special case is when $L=U^b$, the unknot with framing $b \neq 0$, and $L'=\emptyset$. Then $J_{U^b}(n)=q^{b\frac{n^2-1}{4}}[n]$. Applying Lemma \ref{GaussSumFormulas} we get \begin{equation}\label{eq:F_U_b}
F^G_{U^b}(\xi)= {\sum_{n}}^{\xi,G}\; q^{b\frac{n^2-1}{4}}[n]^2=2 \gamma_b^{G}(\xi) \ev_{\xi}\left(\frac{(1-q^{-b_{*r}})^{\chi(c)}}{(1-q)(1-q^{-1})}\right)
\end{equation}
where $\chi(c)=1$ if $c=1$ and zero otherwise. Another (shorter) way to calculate $F^{G}_{U^b}(\xi)$ uses the Laplace transform method introduced in Section \ref{laplace}, see Equation \eqref{1122}. From Formula (\ref{eq:F_U_b}) and Lemma \ref{GammabNonzero} we can see that $F^G_{U_b}(\xi)$ is non--zero except if $G=SU(2)$, $\frac{r}{(r,b)}$ odd, and $\frac{b}{(r,b)} \equiv 2\pmod{4}$.
\end{example}

For a better readability, we omit from now on the index $G$ in all notations. If the dependency on $G$ is not indicated, the notation should be understood in generality, in the sense that $G$ can be $SU(2)$ or $SO(3)$. Otherwise the Lie group will be indicated as a superscript. 
\\

Let $M=S^3(L)$ and $\sigma_+ $ (respectively $\sigma_-$) be the number of positive (respectively negative) eigenvalues of the linking matrix of $L$. Further, let $L'$ be an ordered $l$--component framed link colored by fixed integers $j_1, j_2,\ldots,j_l$. 
We define 
\begin{equation}\label{eq:DefinitionOfWRT}
\tau_{M,L'}(\xi) :=
\frac{F_{L\sqcup L'}(\xi)}
{(F_{U^{+1}}(\xi))^{\sigma_+}\,(F_{U^{-1}}(\xi))^{\sigma_-} }
\end{equation}
where $U^{\pm 1}$ is the unknot with framing $\pm 1$. 

\begin{theorem}[Reshetikhin--Turaev]\label{InvarianceOfWRT}
Assume the order of the root of unity $\xi$ is odd if $G=SO(3)$ and arbitrary otherwise. Then $\tau_{M,L'}(\xi)$ is invariant under orientation preserving homeomorphisms of $M$ and ambient isotopies of $L'$ and is called the \emph{quantum (WRT) invariant} of the pair $(M,L')$. 
\end{theorem}

\begin{proof}
 See \cite{KM} and \cite{RT2,Tu}.
\end{proof}

\begin{example}\label{WRTOfLensSpace}
The quantum invariant of the lens space $L(b,1)$,
obtained by surgery along $U^b$,  is
\begin{equation} \label{eq:WRTOfLensSpace}
\tau_{L(b,1)} (\xi)= \frac{ F_{U^b}(\xi)}{F_{U^{\sn(b)}}(\xi)}=
\frac{\gamma_b(\xi)}{\gamma_{\sn(b)}(\xi)}\cdot
\frac{(1-\xi^{-b_{*r}})^{\chi(c)}}{1-\xi^{-\sn(b)}}
\end{equation}
where $\sn(b)$ is the sign of the integer $b$. Since $S^3(U^1)=S^3$, we have $\tau_{S^3} (\xi)=1$.
\end{example}

\begin{theorem}[Reshetikhin--Turaev, Kirby--Melvin]\label{PropertiesOfWRTInvariant}
 The quantum invariant satisfies the following properties:
\begin{enumerate}
 \item Multiplicativity: $\tau_{M \# N}(\xi)=\tau_{M}(\xi)\tau_{N}(\xi)$.
 \item Orientation: $\tau_{-M}(\xi)=\overline{\tau_{M}(\xi)}$.
 \item Normalization: $\tau_{S^3}(\xi)=1$.
\end{enumerate}
Here, $M\#N$ is the connected sum of the two manifolds $M$ and $N$ and $-M$ denotes $M$ with orientation reversed. By $\overline{z}$ for $z\in \C$, we denote the complex conjugate of $z$.
\end{theorem}

\begin{proof}
We have proven the normalization property in Example \ref{WRTOfLensSpace}. For the multiplicativity and orientation property see \cite{KM} and \cite{RT2,Tu}.
\end{proof}

\subsection{Renormalization of quantum (WRT) invariant}\label{RenormalizedWRTInvariants}
Suppose that $M$ is a rational homology 3--sphere.
There is a unique decomposition $H_1(M;\Z)=\bigoplus_i \Z/b_i\Z$, where each $b_i$ is a prime power. We put $b=\prod b_i =|H_1(M;\Z)|$, the order of the first homology group $H_1(M;\Z)$.
\\

For the rest of this thesis, we allow $b$ to be any number in the $SO(3)$ case but assume $b$ to be \emph{odd} in the $SU(2)$ case. 
\\

Let the absolute value of an integer $x$ be denoted by $|x|$. 
We renormalize the quantum invariant of the pair $(M,L')$ as follows:
\be\label{eq:DefinitionRenormalizedWRT}
\tau'_{M,L'}(\xi):=\frac{\tau_{M,L'}(\xi)}{\prod\limits_i \tau_{L(|b_i|,1)}(\xi)}.
\ee
Notice that due to Example \ref{TauNonZero}, $\tau^{SO(3)}_{L(b,1)}(\xi)$ is always nonzero and $\tau^{SU(2)}_{L(b,1)}(\xi)$ is nonzero for $b$ odd  and therefore the renormalization is well--defined. 
\\

Let us focus on the special case when the linking matrix of
$L$ is diagonal, with $b_1, b_2, \dots, b_m$ on the diagonal.
Assume each $b_i$ is a power of a prime up to sign.
Then $H_1(M,\Z) = \oplus_{i=1}^m \Z/|b_i|\Z$, and
\[
\sigma_+ = {\rm card}\, \{ i\mid b_i >0\}, \quad \sigma_- =
{\rm card}\, \{ i \mid b_i < 0\}. 
\]
Thus from Definitions \eqref{eq:DefinitionOfWRT} and \eqref{eq:DefinitionRenormalizedWRT} and Equation \eqref{eq:WRTOfLensSpace} we have
\begin{equation}
 \tau'_{M,L'}(\xi) = \left( \prod_{i=1}^m  \tau'_{L(b_i,1)}(\xi)
 \right)\,
\frac{F_{L\sqcup L'}(\xi)}
{\prod_{i=1}^m F_{U^{b_i}}(\xi) }
\,  ,
\label{0077}
\end{equation}
with
\[
\tau'_{L(b_i,1)}(\xi)= \frac{\tau_{L(b_i,1)}(\xi)}{\tau_{L(|b_i|,1)}(\xi)}\, . 
\]

For the renormalized quantum invariant, multiplicativity and normalization follows from Theorem \ref{PropertiesOfWRTInvariant}, but reversing the orientation of the manifold does not induce complex conjugacy for the renormalized quantum invariant. For example $L(-b,1)$ is homeomorphic to $L(b,1)$ with orientation reversed. But for $b>0$,
$\tau'_{L(b,1)}(\xi)=1$ but $\tau'_{L(-b,1)}(\xi)\not=1$ (see Section \ref{LensSpaces} for the exact calculation). To achieve a better behavior under orientation reversing, we could instead renormalize the invariant as
\[
\tilde{\tau}_{M,L'}(\xi):=\frac{\tau_{M,L'}(\xi)}{\prod\limits_{i=1}^m\sqrt{\tau_{L(b_i,1)\#L(-b_i,1)}(\xi)}}.
\]
The disadvantage of this normalization is that to be allowed to take the square root of 
\linebreak
$\tau_{L(b_i,1)\# L(-b_i,1)}(\xi)$ we need to fix a $4$th root of $\xi$. We have this in the $SU(2)$ case but not in the $SO(3)$ case.

\subsection{Connection between $SU(2)$ and $SO(3)$ invariant}\label{SU2SO3Connection}

\begin{theorem}[Kirby--Melvin]
For $\ord(\xi)$ odd, we have
\[
\tau^{SU(2)}_{M,L'}(\xi)=\tau^{SO(3)}_{M,L'}(\xi)\cdot \tau^{SU(2)}_{M,L'}(e_3).
\]
Therefore, for $\ord(\xi)$ odd, the $SO(3)$ invariant is sometimes stronger than the $SU(2)$ invariant since $\tau^{SU(2)}_{M,L'}(e_3)$ is sometimes zero while the $SO(3)$ invariant is not. 
\end{theorem}

\begin{proof}
See \cite[Corollary 8.9]{KM}. 
\end{proof}

\begin{remark}
 In \cite{Le3}, T. Le proved a similar result for arbitrary semi--simple Lie algebras $\fg$. 
\end{remark}

The result for the renormalized quantum invariant follows immediately:

\begin{corollary}
For $\ord(\xi)$ odd, we have
\[
\tau'^{SU(2)}_{M,L'}(\xi)=\tau'^{SO(3)}_{M,L'}(\xi)\cdot \tau'^{SU(2)}_{M,L'}(e_3).
\]
\end{corollary}

\chapter{Cyclotomic completions of polynomial rings}\label{cyc}

In \cite{Ha1}, Habiro develops a theory for cyclotomic completions of polynomial rings. 
In this chapter, we first recall some important results about cyclotomic polynomials and inverse limits before summarizing some of Habiro's results. We then define the rings $\cS_b$ and $\cR_b$ in which the unified invariant defined in Chapter \ref{UnifiedInvariant} is going to lie. We also describe the evaluation in these rings. Most important, we prove that when the evaluation of two elements of the ring $\cS_b$ coincide at all roots of unity, the two elements are actually identical in $\cS_b$. A similar statement holds in the ring $\cR_b$ for roots of unity of odd order.

\section{On cyclotomic polynomial}
Recall that $e_n := \exp(\frac{2\pi i}{n})$ and denote by
$\Phi_n(q)$ the cyclotomic polynomial
\[
\Phi_n(q) = \prod_{\substack{(j,n)=1\\0<j\leq n}} (q - e_n^j).
\]
Since $q^n-1=\prod_{j=0}^{n-1}(q-e_n^j)$, collecting together all terms belonging to roots of unity of the same order, we have 
\[
q^n-1=\prod_{d\mid n}\Phi_d(q).
\]
The degree of $\Phi_n(q)\in \Z[q]$ is given by the Euler function $\varphi(n)$.
Suppose $p$ is a prime and $n$ an integer. Then (see e.g. \cite{Lang})
\begin{equation} \Phi_n(q^p)= \begin{cases}    \Phi_{np}(q)  & \text{ if } p \mid n \\
\Phi_{np}(q) \Phi_n(q)  & \text{ if } p \nmid n.
\end{cases}
\end{equation}
It follows that  $\Phi_n(q^p)$ is always divisible by $\Phi_{np}(q)$.
The ideal of $\Z[q]$ generated by $\Phi_n(q)$ and $\Phi_m(q)$ is well--known,
see e.g. \cite[Lemma 5.4]{Le}:

\begin{lemma}
\label{IdealGeneratedByPhinPhim}
$\text{ }$

\begin{itemize}
\item[(a)] If $\frac{m}{n} \neq p^e$ for any  prime $p$ and any integer $e\neq 0$, then
$(\Phi_n(q))+ (\Phi_m(q))=(1)$ in $\Z[q]$.

\item[(b)] If $\frac{m}{n} = p^e$ for a prime $p$ and some integer $e \neq 0$, then $(\Phi_n(q))+ (\Phi_m(q))=(1)$ in $\Z[1/p][q]$.
\end{itemize}
\end{lemma}

\begin{remark}\label{xytoxkyl}
Note that in a commutative ring $R$,  $(x) + (y) =(1)$ if and only if
$x$ is invertible in $R/(y)$. Therefore $(x) + (y) =(1)$ implies $(x^k) +
(y^l) =(1)$ for any integers $k,l \ge 1$.
\end{remark}

\section{Inverse limit}

Let $(I,\leq)$ be a partially ordered set, $R_i$, $i\in I$, unital commutative rings and for $i\leq j$ let $f_{ij}:R_j \to R_i$ be ring homomorphisms. We call $(R_i,f_{ij})$ an \emph{inverse system} of rings and ring homomorphisms if $f_{ii}$ is the identity map in $R_i$ and $f_{ik}=f_{ij} \circ f_{jk}$ for all $i\leq j\leq k$. 

The \emph{inverse limit} of an inverse system $(R_i,f_{ij})$ is defined to be the ring
\[
\inverselim{i\in I} R_i =\left\{(r_i) \in \prod_{i\in I}R_i \;\Big|\; r_i = f_{ij}(r_j) \mbox{ for all } i \leq j\right\}.
\]
The following is well--known.
\begin{lemma}\label{cofinal}
 If the set $J\subset I$ is cofinal to $I$, i.e. for every element $i\in I$ exists an element $j\in J$ such that $j\geq i$, we have
\[
 \inverselim{i\in I}R_i \iso \inverselim{j \in J} R_j.
\]
\end{lemma}

\begin{proof}
Let $\varphi: \inverselim{i\in I}R_i \to \inverselim{j\in J}R_j$ be the canonical projection. We have to show that $\varphi$ is injective. Assume $\varphi((r_i))=(0)$. Since $J$ is cofinal to $I$, for every $i\in I$ we can choose $j_i \in J$ such that $j_i\geq i$. But $f_{ij_i}(r_{j_i})=r_i$ and $r_{j_i}=0$ for all $j_i$. Therefore we have $r_i=0$ for all $i\in I$.
\end{proof}

For a ring $R$ and $I\subset R$ an ideal, the inverse limit $\inverselim{j} R/I^j$ is called the \emph{$I$-adic completion of $R$}.
There is a map from this ring to the formal sums of elements of $R$. Namely, every element $r$ in $\inverselim{j} R/I^j$ can be expressed in the form
\[
r=\sum_{j\geq 0} s_j i_j
\]
where $s_j \in R/I$ and $i_j\in I^j$. This decomposition is not unique. 

\begin{example}
The $(3)$--adic completion of $\Z$ corresponds to the $3$--adic expansion of $\Z$. For example, the number $124$ corresponds to the element 
\[
(r_n)=(0,1,7,16,43,124,124,124,\ldots)
\]
in $\inverselim{n\in \N}\Z/(3^n)$ and as $3$--adic number we write it as
\[
124=1+2\cdot 3+1\cdot 9+1\cdot 27+1\cdot 81=\sum_{n\geq 0} s_n 3^n
\]
where $s_n=\frac{r_{n+1}-r_n}{3^n}$.
\end{example}

\begin{example}
In \cite{Ha1}, Habiro defined the so--called \emph{Habiro ring}
\[
\Habiro:=\inlim \Z[q]/((q;q)_n).
\]
Every element $f(q)\in \Habiro$ can be written as an infinite sum
\[
f(q)= \sum_{n\ge 0} f_n(q)\, (1-q)(1-q^2)...(1-q^n)
\]
with $f_n(q)\in \Z[q]$. When evaluating $f(q)$ at a root of unity $\xi$, only a finite number of terms on the right hand side are not zero, hence the right hand side gives a well--defined value. Since $f_n(q)\in \Z[q]$, the evaluation of $f(q)$ is an algebraic integer, i.e. lies in $\Z[\xi]$.
\end{example}

\section{Cyclotomic completions of polynomial rings}

We now summarize further results of Habiro on cyclotomic completions of polynomial rings \cite{Ha1}. Let  $R$ be a commutative integral domain of characteristic zero and $R[q]$  the polynomial ring over $R$.
We consider $\N$ as a directed set with respect to the divisibility relation.
For any $S\subset \N$, the $S$--cyclotomic completion ring $R[q]^S$ is defined as 
\be\label{rs} 
R[q]^S:=\lim_{\overleftarrow{f(q)\in \Phi^*_S}} \;\;\frac{R[q]}{(f(q))} 
\ee
where $\Phi^*_S$ denotes the multiplicative set in $\Z[q]$ generated by $\Phi_S=\{\Phi_n(q)\mid n\in S\}$ and directed with respect to the divisibility relation.

\begin{example}
Since the sequence $(q;q)_n$, $n\in \N$, is cofinal to $\Phi^*_\N$, Lemma \ref{cofinal} implies
\be
\Habiro\simeq\Z[q]^\N.
\ee
\end{example}

Note that if $S$ is finite, $R[q]^S$ is identified with the $(\prod \Phi_S)$--adic completion of $R[q]$. In particular,
\[
R[q]^{\{1\}}\simeq R[[q-1]], \quad
R[q]^{\{2\}}\simeq R[[q+1]].
\]
\\

Two positive integers $n, n'$ are called {\em adjacent} if $n'/n=p^e$ with a nonzero $e\in \Z$ and  a  prime $p$, such that the ring $R$ is $p$--adically separated, i.e. $\bigcap_{n=1}^\infty (p^n) =0$ in $R$. A set of positive integers is {\em $R$--connected} if for any two distinct elements $n,n'$ there is a sequence $n=n_1, \,n_2, \dots,\, n_{k-1},\, n_k= n'$ in the set, such that any two consecutive numbers of this sequence are adjacent. 
\\

Suppose $S' \subset S$, then $\Phi^*_{S'}\subset \Phi^*_S$, hence there is a natural map 
\[
\rho^R_{S, S'}: R[q]^S \to R[q]^{S'}.
\]

\begin{theorem}[Habiro]\label{HabiroTheorem4.1}
If $S$ is $R$--connected, then for any subset $S'\subset S$ the natural map 
\[
\rho^R_{S,S'}: R[q]^S \hookrightarrow R[q]^{S'}
\] 
is an embedding.
\end{theorem}

\begin{proof}
 See \cite[Theorem 4.1]{Ha1}.
\end{proof}

\section{The rings $\cS_b$ and $\cR_b$}
For any positive integer $b$, we  define
\be
\label{ab} \cR_b:=\inverselim{k}\frac{\Z[1/b][q]}{\left((q;q^2)_k\right)}
\quad\quad\text{and}\quad\quad
\cS_b :=\inverselim{k}\frac{\Z[1/b][q]}{((q;q)_k)}\; .
\ee

For every integer $a$, we put $\N_a := \{ n \in \N \mid (a,n)=1\}$. Since the sets $\Phi_{\N}^{*}$ and $\{(q;q)_n\mid n\in \N\}$, as well as $\Phi_{\N_2}^*$ and $\{(q;q^2)_n \mid n\in \N\}$, are cofinal we have due to Lemma \ref{cofinal}
\[
\cR_b\iso \Z[1/b][q]^{\N_2} \quad \quad\text{and}\quad\quad \cS_b\iso \Z[1/b][q]^{\N}.
\]

\begin{remark}
For $b=1$, we have $\cS_1=\Habiro$. Further, if $p$ is a prime divisor of $b$, we have $\cR_p \subset \cR_b$ and $\cS_p\subset \cS_b$.
\end{remark}

\subsection{Splitting of $\cS_b$ and $\cR_b$}

Observe that $\N$ is not $\Z[1/b]$--connected for $b>1$. In fact, for a prime $p$ one has 
$\N =\amalg_{j=0}^\infty \; p^j \N_p$, where each $p^j\N_p$ is $\Z[1/p]$--connected.

Suppose $p$ is a prime divisor of $b$. Let us define
\[ 
\cS_{b}^{p,0} := \Z[1/b][q]^{\N_p}\;, \qquad 
\cS_{b}^{p,\bz} := 
\prod_{j>0}\Z[1/b][q]^{p^j\N_p}\;,
\quad \text {and } \;\;
\cS_{p,j}:= \Z[1/p][q]^{p^j \N_p}.
\]
Notice that $\cS_{p}^{p,0}=\cS_{p,0}$ and $\cS_p^{p,\bz}=\prod_{j> 0} \cS_{p,j}$. Further is $\cS_p^{p,\epsilon}\subset \cS_b^{p,\epsilon}$ for $\epsilon$ either $0$ or $\bz$.

\begin{proposition} 
For $p$ a prime divisor of $b$, we have
\begin{equation} \label{SplittingSb}
\cS_b \simeq \cS_{b}^{p,0} \times \cS_{b}^{p,\bz}
\end{equation}
and therefore there are canonical projections
\[
\pi^p_0   : \cS_b \to \cS_b^{p,0}\;\; \text{ and } \;\;
\pi^p_\bz : \cS_b \to \cS_b^{p,\bz}\;. 
\]
In particular, for every prime $p$ one has
\[
\cS_p\simeq \prod_{j=0}^\infty \cS_{p,j}
\]
and canonical projections $\pi_j : \cS_p \to \cS_{p,j}$.
\end{proposition}

\begin{proof}
Suppose $n_i \in p^{j_i}\N_p$ for $ i=1,\dots, m$, with  distinct $j_i$'s. Then $n_{i}/n_{s}$, with $i \neq s$, is either not a power of a prime or a non--zero power of $p$. Hence by Lemma \ref{IdealGeneratedByPhinPhim} (and Remark \ref{xytoxkyl}), for any positive integers $k_1,\dots, k_m$, we have 
\[
(\Phi_{n_i}^{k_i}(q))+ (\Phi_{n_s}^{k_s}(q))=(1) \quad \text{ in }\Z[1/b][q].
\]
By the Chinese remainder theorem, we have
\[
\frac{ \Z[1/b][q]}{\left(\prod_{i=1}^m \Phi_{n_i}^{k_i}(q)\right)}\;\simeq\;
\prod_{i=1}^m \frac{\Z[1/b][q]}{\left(\Phi_{n_i}^{k_i}(q)\right)} .
\]
Taking the inverse limit, we get \eqref{SplittingSb}.
\end{proof}

We will also use the notation $\cS_{b,0}:=\Z[1/b][q]^{\N_b}$ and as above one can see that we have the projection $\pi_0 : \cS_b \to \cS_{b,0}$.
\\

A completely similar splitting exists for $\cR_b$, where $\cR_b^{p,\epsilon}$, $\epsilon \in \{0,\bar{0}\}$, are defined analogously by replacing $\N_p$ by $\N_{2p}$ (only odd numbers coprime to $p$) as
\[ 
\cR_{b}^{p,0} := \Z[1/b][q]^{\N_{2p}}\;
\quad \text {and } \;\;
\cR_{b}^{p,\bz} := 
\prod_{j>0}\Z[1/b][q]^{p^j\N_{2p}}\;
\]
and the projections $\pi_{\epsilon}^p$ are defined analogously as above in the $\cS_b$ case. If $2\mid b$, then $\cR^{2,0}_b$ coincides with $\cR_b$. 
\\

We get the following.

\begin{corollary}\label{Splitting}
For any odd divisor $p$ of $b$, an element $ x \in \cR_b$ (or $\cS_b$) determines and is totally determined  by the pair $(\pi^p_0(x), \pi^p_\bz(x))$. If $p=2$ divides $b$, then for any $x\in \cR_b$, $x=\pi^p_0(x)$.
\end{corollary}

\subsection{Further splitting of $\cS_b$ and $\cR_b$}\label{FurtherSplitting}

Let $\{p_i\,|\, i=1,\dots,m\}$ be the set of all distinct {\em odd} prime divisors of $b$. For $\bn=(n_1,\dots,n_m)$, a tuple of numbers $n_i\in \N$, let $\bp^{\bn}=\prod_{i}p_i^{n_i}$. Let $A_{\bn}=\bp^{\bn}\N_b$ and $O_{\bn}:=\bp^{\bn}\N_{2b}$. Then $\N_2=\amalg_{\bn}\, O_\bn$ and, if $b$ odd, $\N=\amalg_{\bn} A_\bn$. 
Moreover, for $a\in O_{\bn},$ $a'\in O_{\bn'}$, we have $(\Phi_{a}(q),\Phi_{a'}(q))=(1)$ in $\Z[1/b]$ if $\bn\neq\bn'$. The same holds for $a\in A_{\bn}$ and $a'\in A_{\bn'}$. In addition, each $O_\bn$ and $A_{\bn}$ is $\Z[1/b]$--connected. An argument similar  to that for Equation \eqref{SplittingSb} gives
\[
\cR_b\simeq\prod_{\bn}\Z[1/b][q]^{O_\bn} \quad\quad\text{and}\quad\text{if $b$ odd }\;\;
\cS_b\simeq\prod_{\bn}\Z[1/b][q]^{A_\bn}.
\]

\begin{proposition}\label{Sb20}
For odd $b$, the natural homomorphism $\rho_{\N,\N_2}:\cS_b\to \cR_b$ is injective. If $2\mid b$, then the natural homomorphism $\cS_{b}^{2,0}\to \cR_{b}$ is an isomorphism.
\end{proposition}

\begin{proof}
By Theorem \ref{HabiroTheorem4.1} of Habiro, the map
\[ 
\Z[1/b][q]^{A_\bn}\hookrightarrow \Z[1/b][q]^{O_{\bn}}
\]
is an embedding. Taking the inverse limit we get the result. If $2\mid b$, then $\cS^{2,0}_b:=\Z[1/b][q]^{\N_2}\simeq\cR_b$.
\end{proof}

\section{Evaluation}
Let $\xi$ be a root of unity and $R$ be a ring. We define the evaluation map 
\begin{eqnarray*}
 \ev_{\xi} : R[q]&\to& R[\xi]\\
q&\mapsto& \xi.
\end{eqnarray*}
Since for $r=\ord(\xi)$
\[
R[\xi]\simeq \frac{R[q]}{(\Phi_r(q))}, 
\]
the evaluation map $\ev_{\xi}$ can be defined analogously on $R[q]^S$ if $\xi$ is a root of unity of order in $S$.

For a set $\Xi$ of roots of unity whose orders form a subset $\cT\subset S$, one defines the evaluation 
\[
\ev_\Xi: R[q]^S \to \prod_{\zeta \in \Xi} R[\zeta].
\]

\begin{theorem}[Habiro]\label{HabiroTheorem6.1}
If $R\subset \Q$, $S$ is $R$--connected and there exists $n\in S$ that is adjacent to infinitely many elements in $\cT$, then $\ev_\Xi$ is injective.
\end{theorem}

\begin{proof}
 See \cite[Theorem 6.1]{Ha1}.
\end{proof}

For a prime $p$, while for every $f \in \cS_p$ the evaluation $\ev_\xi(f)$ can be defined for every root of unity $\xi$, for $f\in \cS_{p,j}$ the evaluation $\ev_\xi(f)$ can only be defined  when $\xi$ is a root of unity of order in $p^j\N_p$. Actually we have the following.

\begin{lemma} \label{1100}
Suppose $\xi$ is a root of unity of order $r= p^j r'$, with $(r',p)=1$. Then for any $f\in \cS_p$, one has 
\[ 
\ev_\xi(f) = \ev_\xi(\pi_j(f)). 
\]
If $i\neq j$, then $\ev_\xi(\pi_i(f))=0$.
\end{lemma}

\begin{proof} Note that $\ev_\xi(f)$ is the image of $f$ under the projection 
$\cS_p \to \cS_p/(\Phi_r(q))= \Z[1/p][\xi]$. It remains to notice that $\cS_{p,i}/(\Phi_r(q))=0$ if $i\neq j$.
\end{proof}

Similarly, if $f \in \cS_b$ and $\xi$ is a root of unity of order coprime with $p$, then $\ev_\xi(f) = \ev_\xi(\pi^p_0(f))$. If the order of $\xi$ is divisible by $p$, then $\ev_\xi(f) = \ev_\xi(\pi^p_\bz(f))$. The same holds when $f\in \cR_b$.
\\

Let $T$ be an infinite set  of powers of an odd prime not dividing $b$ and let $P$ be an infinite set of odd primes not dividing $b$. As above in Section \ref{FurtherSplitting}, let $\{p_i\,|\, i=1,\dots,m\}$ be the set of all distinct {\em odd} prime divisors of $b$ and for $\bn=(n_1,\dots,n_m)$, let $\bp^{\bn}=\prod_{i}p_i^{n_i}$.

\begin{proposition}\label{f=gFollowsEvf=Evg} 
For a given $\bn=(n_1,\ldots,n_m)$, suppose $f, g\in \Z[1/b][q]^{O_{\bn}}$ or $f, g\in \Z[1/b][q]^{A_{\bn}}$ such that $\ev_\xi(f)=\ev_\xi(g)$ for any root of unity $\xi$ with $\text{ord}(\xi)\in \bp^\bn T$, then $f=g$. The same holds true if $\bp^\bn T$ is replaced by $\bp^\bn P$.
\end{proposition}

\begin{proof}
Since both sets $T$ and $P$ contain infinitely many
numbers adjacent to $\bp^\bn$, the claims follows from Theorem \ref{HabiroTheorem6.1}.
\end{proof}

We can infer from this directly the following important result.

\begin{corollary}\label{UniquenessAfterEvaluation}
Let $p$ be an odd prime not dividing $b$ and $T$ the set of all integers of the form $p^k b'$ with $k\in \N$ and $b'$ any odd divisor of $b^n$ for some $n$. Any element $f(q)\in\cR_b$ or $f(q)\in \cS_b$ is totally determined by the values at roots of unity with orders in $T$.
\end{corollary}

\chapter{Unified invariant}
\label{UnifiedInvariant}

In \cite{Le,BBL, BL}, T. Le together with A. Beliakova and C. Blanchet defined unified invariants for rational homology 3--spheres with the restriction that one can only evaluate these invariants at roots of unity with order coprime to the order of the first homology group of the manifold. 

In this chapter, we define unified invariants for the quantum (WRT) $SU(2)$ and $SO(3)$ invariants of rational homology 3--spheres with links inside, which can be evaluated at any root of unity. The only restriction which remains is that we assume for the $SU(2)$ case the order of the first homology group of the rational homology 3--sphere to be odd.

To be more precise, let $M$ be a rational homology 3--sphere with $b:=|H_1(M)|$ and let $L\subset M$ be a framed link with $l$ components. Assume that $L$ is colored by fixed $\bj=(j_1,\dots, j_l)$ with $j_i$ odd for all $i$.
The following theorem is the main result of this thesis.

\begin{theorem}\label{MainTheorem}
There exist invariants $I^{SO(3)}_{M,L} \in \cR_b$ and, for $b$ odd, $I^{G}_{M,L} \in \cS_b$, such that
\be\label{EvI=Tau'}
\ev_{\xi}(I^G_{M,L}) = \tau'^G_{M,L}(\xi)
\ee
for any root of unity $\xi$ (of odd order if $G=SO(3)$).
Further, the invariants are multiplicative with respect to the connected sum, i.e. for $L\subset M$ and $L'\subset M'$,
\[
I_{M\# M', L\sqcup L'}=I_{M,L}\cdot I_{M',L'}.
\]
\end{theorem}

The rest of this chapter is devoted to the proof of this theorem using technical results that will be proven later.
\\

The following observation is important. By Corollary \ref{UniquenessAfterEvaluation}, there is {\em at most one} element $f^{SO(3)}(q)\in \cR_b$ and, if $b$ odd, {\em at most one} element $f^{SU(2)}(q)\in \cS_b$, such that for every root $\xi$ of odd order one has  
\[ 
\tau'^G_{M,L} (\xi) = \ev_\xi\left( f^G(q)\right).
\]
That is, if  we can find such an element, it is unique. Therefore, since $\tau'^G$ is an invariant of manifolds with links inside, so must be $f^G$  and we put $I^G_{M,L'} := f^G(q)$. It follows directly from Theorem \ref{PropertiesOfWRTInvariant} that this unified invariant must be multiplicative with respect to the connected sum. 
\\

We say that $M$ is \emph{diagonal} if it can be obtained from $S^3$ by surgery along a framed link $L_M$  with diagonal linking matrix where  the diagonal entries are of the form $\pm p^k$ with $p=0,1$ or a prime.
\\

To define  the unified invariant for a general rational homology 3--sphere $M$, one first adds to $M$ lens spaces to get a diagonal manifold $M'$, for which
the unified invariant $I_{M'}$ will be defined in Section \ref{DefinitionDiagonal}. Then $I_M$ is the quotient of $I_{M'}$ by the unified invariants of the lens spaces (see Section \ref{DefinitionLensSpaces}), which were added. But unlike the simpler case of \cite{Le} (where the orders of the roots of unity are always chosen to be coprime to $|H_1(M)|$), the unified invariant of lens spaces are {\em not} invertible in general.
To overcome this difficulty we insert knots in lens spaces and split the unified invariant into different components.

\section{Unified invariant of lens spaces}\label{DefinitionLensSpaces}
Suppose $b,a,d$ are integers with $(b,a)=1$ and $b\neq 0$. Let $M(b,a;d)$ be the pair of a lens space $L(b,a)$ and a knot $K\subset L(b,a)$, colored by $d$, as described in Figure \ref{fig:LensSpaceKnot}. 
Among these pairs we want to single out some whose quantum invariants are invertible.

\begin{figure}[htbp]
\begin{center}
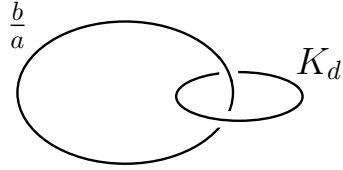
\caption{The lens space  $(L(b,a),K_d)$ is obtained by
$b/a$ surgery on the first component of the link. The second
component is the knot $K$ colored by $d$.}
\label{fig:LensSpaceKnot}
\end{center}
\end{figure}

\vspace{-1cm}

\begin{remark}
It is known that if the color of a link component is $1$, then the component can be removed from the link without affecting the value of quantum invariants (see \cite[Lemma 4.14]{KM}). Hence $\tau_{M(b,a;1)} = \tau_{L(b,a)}$.
\end{remark}

Let $p$ be any prime divisor of $b$. Due to Corollary \ref{Splitting}, to define $I_M$ it is enough to fix $I^0_M= \pi^p_0(I_M)$ and $I^\bz_M=\pi^p_\bz(I_M)$. 
\\

For $\ve\in \{0,\bz\}$, let $M^\ve(b,a) := M(b,a;d(\ve))$, where $d(0):=1$ and $d(\bz)$ is the smallest odd positive integer such that $|a|d(\bz) \equiv 1 \pmod {b}$. 
First observe that such $d(\bz)$ always exists. Indeed, if $b$ is odd, we can achieve this by adding $b$, otherwise the inverse of any odd number modulo an even number is again odd. Further observe that if $|a|=1$, $d(0)=d(\bz)=1$. 

\begin{lemma}\label{ExistensLensSpaces} 
Suppose $b=\pm p^{l}$ is a prime power. For $\ve \in \{0,\bz\}$, there exists an invertible invariant $I^{\ve}_{M^\ve(b,a)} \in \cR^{p,\ve}_p$ such that 
\[
\tau'_{M^\ve(b,a)}(\xi)\;=\; \ev_{\xi}\left(I^{\ve}_{M^\ve(b,a)}\right) 
\]
where $\ve=0$ if the order of $\xi$ is not divisible by $p$ (and odd if $G=SO(3)$), and $\ve=\bz$ otherwise. Moreover, if $p$ is odd, then $I^{\ve}_{M^\ve(b,a)}$ belongs to and is  invertible in $\cS^{p,\ve}_p$.
\end{lemma}
A proof of Lemma \ref{ExistensLensSpaces} will be given in Chapter \ref{LensSpaces}.

\section{Unified invariant of diagonal manifolds} \label{DefinitionDiagonal}

Remember the definition given in Section \ref{CyclotomicJones} 
\[ 
A(n,k) = \frac{\prod^{k}_{i=0}
\left(q^{n}+q^{-n}-q^i -q^{-i}\right)}{(1-q) \, (q^{k+1};q)_{k+1}}
\]
for non--negative integers $n,k$.

We need the following main technical result of this thesis.

\begin{theorem} \label{Qbk}
Suppose $b=\pm 1$ or $b= \pm p^l$ where $p$  is a  prime and $l$ is positive. For any non--negative integer $k$, there exists an element $Q^G_{b,k} \in \cR_b $ such that for every root $\xi$ of order $r$ (odd, if $G=SO(3)$), one has
\[
\frac{{\sum\limits_{n}}^{\xi,G} \, q^{b\frac{n^2-1}{4}} A(n,k) }{F^G_{U^b}(\xi)}
= \ev_\xi (Q^G_{b,k}).
\]
In addition, if $b$ is odd, $Q^G_{b,k} \in \cS_b $.
\end{theorem}
A proof will be given in Chapter \ref{laplace}.
\\

Let $L_M\sqcup L$ be a framed link in $S^3$ with disjoint sublinks $L_M$ and $L$, with $m$ and $l$ components, respectively. Assume that $L$ is colored by a fixed $\bj=(j_1,\dots, j_l)$ with $j_i$ odd for all $i$. Surgery along the framed link $L_M$ transforms $(S^3,L)$ into  $(M,L)$. Assume now that $L_M$ has diagonal linking matrix with nonzero entries $b_i=p_i^{l_i}$, $p_i$ prime or 1, on the diagonal. Then $M=S^3(L_M)$ is a diagonal rational homology 3--sphere. 
Using \eqref{Jones} and Remark \ref{FiniteSum}, taking into account the framings $b_i$, we have
\[
J_{L_M\sqcup L}(\bn,\bj)\prod_{i=1}^m [n_i] = \sum_{\bk\ge 0} C_{L_{M,0} \sqcup L}
(\bk,\bj) \, \prod_{i=1}^m q^{b_i \frac{n_i^2-1}{4}} A(n_i,k_i)
\]
where $L_{M,0}$ denotes the link $L_M$ with all framings switched to zero. 
By the Definition \eqref{F} of $F^G_{L_M\sqcup L}$, we get
\[ F^G_{L_M\sqcup L}(\xi)= \sum_{\bk \ge 0}
\ev_\xi(C_{L_{M,0} \sqcup L}(\bk,\bj)) \, \prod_{i=1}^m
{\sum_{n_i}}^{\xi,G} \,  q^{b_i \frac{n_i^2-1}{4}} A(n_i,k_i).
\]
From \eqref{0077} and  Theorem \ref{Qbk}, we get
\[
\tau'_{M,L}(\xi) = \ev_\xi \left \{  \prod_{i=1}^m  I_{L(b_i,1)} \,
 \sum_{\bk} C_{L_{M,0} \sqcup L}(\bk,\bj) \, \prod_{i=1}^m  Q_{b_i,k_i} \right \}
\]
where the unified invariant of the lens space $I_{L(b_i,1)}\in \cR_b$, with $\ev_\xi(I_{L(b_i,1)})=\tau'_{L(b_i,1)}(\xi)$, exists by  Lemma \ref{ExistensLensSpaces}. Thus if we define
\[
I_{(M,L)}:= \prod_{i=1}^m  I_{L(b_i,1)} \, \sum_{\bk}
C_{L_{M,0} \sqcup L}(\bk,\bj) \, \prod_{i=1}^m  Q_{b_i,k_i}\, ,
\]
then \eqref{EvI=Tau'} is satisfied. By Theorem \ref{GeneralizedHabiro}, $C_{L_{M,0} \sqcup L}(\bk,\bj)$ is divisible by $(q^{k+1};q)_{k+1}/(1-q)$ which is divisible by $(q;q)_k$ where $k = \max k_i$. It follows that $I^{SO(3)}_{(M,L)} \in \cR_b$ and, if $b$ is odd, $I^G_{(M,L)} \in \cS_b$.

\section{Definition of the unified invariant: general case}\label{DefinitionGeneral}

The general case reduces to the diagonal case by the well--known trick of diagonalization using lens spaces:

\begin{lemma}[Le]\label{diagonalization} 
For every rational homology sphere $M$, there are lens spaces $L(b_i,a_i)$ such that the connected sum of $M$ and these lens spaces is diagonal. Moreover, each $b_i$ is a prime power divisor of $|H_1(M,\Z)|$.
\end{lemma}

\begin{proof}
  See \cite[Proposition 3.2 (a)]{Le}.
\end{proof}

Suppose $(M,L)$ is an arbitrary pair of a rational homology 3--sphere and a link $L$ in it colored by odd numbers $j_1,\dots, j_l$. Let $L(b_i,a_i)$ for  $i=1,\dots, m$  be the lens spaces of Lemma \ref{diagonalization}. To construct the unified invariant of $(M,L)$, we use induction on $m$. If $m=0$, then $M$ is diagonal and $I_{M,L}$  has been defined in Section \ref{DefinitionDiagonal}.

Since $(M,L) \# M(b_1,a_1;d)$ becomes diagonal after adding $m-1$ lens spaces, the unified invariant of $(M,L) \# M(b_1,a_1;d)$ can be defined by  induction for any odd integer $d$. In particular, one can define $I_{M^\ve}$ where $M^\ve := (M,L) \# M^\ve(b_1,a_1)$. Here $\ve = 0$ or $\ve =\bz$ and $b_1$ is a power of a prime $p$ dividing $b$. It follows that  the components $\pi^p_\ve(I^{SO(3)}_{M^\ve}) \in \cR^{p,\ve}_b$ and $\pi^p_\ve(I^{G}_{M^\ve}) \in \cS^{p,\ve}_b$, for $b$ odd, are defined.
By Lemma \ref{ExistensLensSpaces}, $I^\ve_{M^\ve(b_1,a_1)}$ is defined and invertible. We put
\[
I^\ve_{M,L} := I^\ve_{M^\ve} \cdot  (I^\ve_{M^\ve(b_1,a_1)})^{-1}
\]
and due to our construction $I_{M,L}:= (I^0_{M,L}, I^\bz_{M,L})$ satisfies \eqref{EvI=Tau'}. This completes the construction of $I_{M,L}$. 

\begin{remark}
The part $I^0_M= \pi^p_0(I_M)$, when $b=p$, was  defined by T. Le \cite{Le} (up to normalization), where Le considered the case when the order of the roots of unity is coprime to $b$. 
More precisely, the invariant defined in \cite{Le} for $M$ divided by the invariant of $\#_i L(b^{k_i}_i,1)$ (which is invertible in $\cS_{b,0}$, see \cite[Subsection 4.1]{Le} and Remark \ref{Le'sRing} below) coincides with $\pi_0 I_M$ up to a factor $q^{\frac{1-b}{4}}$ by Theorem \ref{MainTheorem}, \cite[Theorem 3]{Le} and Proposition \ref{f=gFollowsEvf=Evg}. Nevertheless, we give a self--contained definition of $I^0_M$ here.
\end{remark}

It remains to prove Lemma \ref{ExistensLensSpaces} (see Chapter \ref{LensSpaces}) and Theorem \ref{Qbk} (see Chapter \ref{laplace}).

\chapter{Roots in $\cS_p$}\label{RootsInSp}

The proof of the main theorem uses the Laplace transform method introduced in Chapter \ref{laplace}. The aim of this chapter is to show that the image of the Laplace transform belongs to $\cR_b$ (respectively $\cS_b$ if $b$ odd), i.e. that certain roots of $q$ exist in $\cR_b$ (respectively $\cS_b$). We achieve this by showing that a certain type of Frobenius endomorphism of $\cS_{b,0}$ is in fact an isomorphism.

\section{On the module $\Z[q]/(\Phi^k_{n}(q))$} 
Since cyclotomic completions of polynomial rings are built from modules like $\Z[q]/(\Phi^k_{n}(q))$,
we first consider these modules. Fix $n,k\ge 1$. Let
\[ 
E := \frac{\Z[q]}{(\Phi^k_{n}(q))} \quad \text { and } \quad 
G := \frac{\Z[e_{n}][x]}{(x^k)}\; .
\]

The following is probably well--known.
\begin{proposition}\label{h_kInjektive}
$\text{ }$
\begin{enumerate}
\item Both $E$ and $G$ are free $\Z$--modules of rank $k \varphi(n)$.

\item The algebra map $h: \Z[q] \to \Z[e_n][x]$ defined by
\[ 
h (q) = e_n + x 
\]
descends to a well--defined algebra homomorphism, also denoted by $h$, from $E$ to $G$.
Moreover, the algebra homomorphism $h: E \to G$ is injective.

\end{enumerate}
\end{proposition}

\begin{proof} 
\begin{enumerate}

\item Since $\Phi^k_n(q)$ is a monic polynomial in $q$ of degree $k \varphi(n)$, it is clear that
\[
E=  \Z[q]/(\Phi^k_n(q))
\]
is a free $\Z$--module of rank $k\varphi(n)$. Since $G = \Z[e_n]\otimes_\Z \Z[x]/(x^k)$, we see  that $G$ is free over $\Z$ of rank $k\varphi(n)$.

\item To prove that $h$ descends to a map $E \to G$, one needs to verify that $h(\Phi^k_n(q))=0$.
Note that 
\[
h(\Phi^k_n(q))= \Phi^k_{n}(x + e_n) = \prod_{(j,n)=1} (x+e_n - e_n^j)^k.
\]
When $j=1$, the factor is $x^k$, which is 0 in $\Z[e_n][x]/(x^k)$. Hence $h(\Phi^k_n(q))=0$.

Now we prove that $h$ is injective. Let $f(q)\in\Z[q]$. Suppose $h(f(q))=0$, or $f(x+e_n)=0$ in
$\Z[e_n][x]/(x^k)$. It follows that $f(x+e_n)$ is divisible by $x^k$; or that $f(x)$ is divisible by
$(x-e_n)^k$. Since $f$ is a polynomial with coefficients in $\Z$, it follows that $f(x)$ is divisible by all Galois conjugates $(x-e_n^j)^k$ with $(j,n)=1$. Then $f$ is divisible by $\Phi^k_n(q)$. In other words, $f=0$ in $E=  \Z[q]/(\Phi^k_n(q))$. 
\end{enumerate}
\end{proof}

\section{Frobenius maps}
\subsection{A Frobenius homomorphism} We use $E$ and $G$ of the previous section. Suppose $b$ is a positive integer coprime with $n$. If $\xi$ is a primitive $n$th root of 1, i.e. $\Phi_n(\xi)=0$, then $\xi^b$ is also a primitive $n$th root of $1$, i.e. $\Phi_n(\xi^b)=0$. It follows that $\Phi_n(q^b)$ is divisible by $\Phi_n(q)$.

Therefore the algebra map $F_b: \Z[q] \to \Z[q]$, defined by $F_b(q)=q^b$, descends to a well--defined algebra map, also denoted by $F_b$, from $E$ to $E$. We want to understand the image $F_b(E)$. 

\begin{proposition} \label{Fb(E)}
The image $F_b(E)$ is a free $\Z$--submodule of $E$ of maximal rank, i.e. 
\[
\rk(F_b(E)) = \rk(E).
\] 
Moreover, the index of $F_b(E)$ in $E$ is $b^{k(k-1)\varphi(n)/2}$.
\end{proposition}

\begin{proof}
Using Proposition \ref{h_kInjektive} we identify $E$ with its image $h(E)$ in $G$.

Let  $\tilde F_b: G \to G$ be the $\Z$--algebra homomorphism defined by $\tilde F_b(e_n) =e_n^b, \tilde F_b(x)= (x+e_n)^b - e_n^b$.
Note that $\tilde F_b(x) = b e_n^{b-1} x + O(x^2)$, hence $\tilde F_b(x^k)=0$. Further, $\tilde F_b$ is a well--defined algebra homomorphism since $\tilde F_b(e_n+x)=e_n^b+(x+e_n)^b-e_n^b=(x+e_n)^b$, and $\tilde F_b$ restricted to $E$ is exactly $F_b$. Since $E$ is a lattice of maximal rank in $G\otimes \Q$, it follows that the index of $F_b$ is exactly the determinant of $\tilde F_b$, acting on $G\otimes \Q$.

The elements $e_n^j x^l$ with $0\le l < k$ and $(j,n)=1$ for $0<j<n$ or $j=0$ form a basis of $G$. Note that
\[
\tilde F_b(e_n^j x^l) = b^le_n^{jb}  e_n^{(b-1)l} x^l + O(x^{l+1}).
\]
Since $(b,n)=1$, the set $e_n^{jb}$ with $(j,n)=1$ is the same as the set $e_n^{j}$ with $(j,n)=1$. Let $f_1: G\to G$ be the $\Z$--linear map defined by $f_1(e_n^{jb} x^l) = e_n^{j} x^l$. Since $f_1$ permutes the basis elements, its determinant  is $\pm 1$. Let $f_2: G\to G$ be the $\Z$--linear map defined by $f_2(e_n^{j} x^l) = e_n^{j} (e_n^{1-b}x)^l$. The determinant of $f_2$ is again $\pm 1$ because, for any fixed $l$, $f_2$ restricts to the automorphism of $\Z[e_n]$ sending $a$ to $ e^s_n a$, each of these maps has a well--defined inverse: $a \mapsto e^{-s}_n a$. 
Now 
\[ 
f_1 f_2 \tilde F_b(e_n^j x^l) = b^l e_n^j x^l + O(x^{l+1})
\]
can be described by an upper triangular matrix with $b^l$'s on the diagonal; its determinant is equal to $b^{k(k-1)\varphi(n)/2}$.
\end{proof}

From Proposition \ref{Fb(E)} we see that if $b$ is invertible, then the index is equal to 1, and we have:

\begin{proposition} \label{onen}
For $k\in \N$ and any $n$ coprime with $b$, the Frobenius homomorphism 
$F_b: \Z[1/b][q]/\left(\Phi^k_n(q)\right) \to \Z[1/b][q]/\left(\Phi^k_n(q)\right)$, defined by $F_b(q)= q^b$, is an isomorphism.
\end{proposition}

\subsection{Frobenius endomorphism of $\cS_{b,0}$}
For finitely many $n_i\in\N_b$ and $k_i\in \N$, the Frobenius endomorphism
\[
F_b :\frac{\Z[1/b][q]}{\left(\prod_i\Phi^{k_i}_{n_i}(q)\right)}\to
\frac{\Z[1/b][q]}{\left(\prod_i\Phi^{k_i}_{n_i}(q)\right)}
\]
sending  $q$ to  $q^b$, is again well--defined. Taking the inverse limit, we get an algebra endomorphism
\[
F_b: \Z[1/b][q]^{\N_b} \to \Z[1/b][q]^{\N_b}.
\]

\begin{theorem}\label{frob}
The Frobenius endomorphism $F_b: \Z[1/b][q]^{\N_b} \to \Z[1/b][q]^{\N_b}$, sending $q$ to $q^b$, is an isomorphism.
\end{theorem}

\begin{proof} For finitely many $n_i\in\N_b$ and $k_i\in \N$, consider the natural algebra homomorphism 
\[
J:\frac{\Z[1/b][q]}{\left(\prod_i\Phi^{k_i}_{n_i}(q)\right)}
\to
\prod_i \frac{ \Z[1/b][q]}{\left(\Phi^{k_i}_{n_i}(q)\right)} .
\]
This map is injective, because  in the unique factorization domain $\Z[1/b][q]$ one has
\[
(\Phi_{n_1}(q)^{k_1} \dots \Phi_{n_s}(q)^{k_s})
= \bigcap_{j=1}^s \Phi_{n_j}(q)^{k_j} \, .
\]
Since the Frobenius homomorphism  commutes with $J$ and is an isomorphism on the target of $J$ by Proposition \ref{onen}, it is an isomorphism on the domain of $J$. Taking the inverse limit, we get the claim. 
\end{proof}

\section{Existence of $b$th root of $q$ in $\cS_{b,0}$}\label{qthroot}

We want to show that there exists a $b$th root of $q$ in $\cS_{b,0}$. First we need the two following Lemmas.

\begin{lemma} \label{0923}
Suppose that $n$ and $b$ are coprime positive integers and $y \in \Q[e_n]$ with $y^b=1$. Then $y =\pm 1$. If $b$ is odd then $y=1$.
\end{lemma}

\begin{proof}
Let $d\mid b$ be the order of $y$, i.e. $y$ is a primitive $d$th root of $1$. Then $\Q[e_n]$ contains $y$ and hence $e_d$. Since $(n,d)=1$, one has $\Q[e_n]\cap \Q[e_d]=\Q$ (see e.g.  \cite[Corollary of IV.3.2]{Lang}). Hence if $e_d \in \Q[e_n]$, then $e_d\in\Q$ and it follows that $d=1$ or $2$. Thus $y=1$ or $y=-1$. If $b$ is odd, then $y$ cannot be $-1$.
\end{proof}

\begin{lemma} \label{0913}
Let $b$ be a positive integer, $T\subset \N_b$, and $y \in \Q[q]^T$ satisfying $y^b=1$. Then $y =\pm 1$. If $b$ is odd then $y=1$. 
\end{lemma}

\begin{proof}  It suffices to show that for any $n_1, n_2 \dots n_m \in T$, the ring $\Q[q]/(\Phi^{k_1}_{n_1}\dots\Phi^{k_m}_{n_m})$ does not contains a $b$th root of $1$ except possibly for $\pm 1$. Using the Chinese remainder theorem, it is enough to consider the case where $m=1$. 

The ring $\Q[q]/(\Phi_{n}^k(q))$ is isomorphic to $\Q[e_n][x]/(x^k)$, by Proposition \ref{h_kInjektive}. If
\[
y= \sum_{j=0}^{k-1} a_j x^j, \quad  a_j \in \Q[e_n]
\]
satisfies  $y^b=1$, then it follows that $a_0^b=1$. By Lemma \ref{0923}, we have $a_0=\pm1$. One can easily see that $a_1=\dots =a_{k-1}=0$. Thus $y=\pm 1$. 
\end{proof}

In contrast with Lemma \ref{0913}, we have the following.
\begin{proposition}\label{cor9}
For any odd positive $b$ and any subset $T\subset \N_b$, the ring $\Z[1/b][q]^{T}$ contains a unique $b$th root of $q$, which is invertible in $\Z[1/b][q]^T$.

For any even positive $b$ and any subset $T\subset \N_b$, the ring $\Z[1/b][q]^{T}$ contains two $b$th roots of $q$ which are invertible in $\Z[1/b][q]^T$; one is the negative of the other. 
\end{proposition}

\begin{proof} 
Let us first consider the case $T=\N_b$. Since $F_b$ is an isomorphism by Theorem \ref{frob}, we can define a $b$th root of $q$ by 
\[
q^{1/b}:=F^{-1}_b (q)  \in \cS_{b,0}\,.
\]
If $y_1$ and $y_2$ are two $b$th root of the same element, then their ratio $y_1/y_2$ is a $b$th root of 1. From Lemma \ref{0913} it follows that  if $b$ is odd, there is only one $b$th root of $q$ in $\Z[1/b][q]^{\N_b}$, and if $b$ is even, there are 2 such roots, one is the minus of the other. We will denote them $\pm q^{1/b}$.

Further it is known that $q$ is invertible in $\Z[q]^\N$ (see \cite{Ha1}). Actually, there is an explicit expression $q^{-1}=\sum_n q^n (q;q)_n $. Hence  $q^{-1}\in \Z[1/b][q]^{\N_b}$, since the natural homomorphism from $\Z[q]^\N$ to $\Z[1/b][q]^{\N_b}$ maps $q$ to $q$.  In a commutative ring, if $x\mid y$ and $y$ is invertible, then so is $x$. Hence any root of $q$ is invertible.

In the general case of $T \subset \N_b$, we use the natural map 
$\Z[1/b][q]^{\N_b}\hookrightarrow \Z[1/b][q]^T$. 
\end{proof}

\begin{remark}\label{Le'sRing}
By Proposition \ref{cor9}, $\cS_{b,0}$ is isomorphic to the ring $\Lambda^{\N_b}_b:=\Z[1/b][q^{1/b}]^{\N_b}$ used in \cite{Le}. 
\end{remark}

\section{Realization of $q^{a^2/b}$ in $\cS_p$}  \label{defxb}

We define  another Frobenius type algebra homomorphism. The difference of the two types of Frobenius homomorphisms is in the target spaces of these homomorphisms. 

Suppose $m$ is a positive integer. Define the algebra homomorphism 
\[
G_m : R[q]^T \to R[q]^{mT} \quad \text{ by } \qquad G_m(q) = q^m.
\]
Since $\Phi_{mr}(q)$ always divides $\Phi_{r}(q^m)$,  $G_m$ is well--defined.
\\

Throughout this section, let $p$ be a prime or $1$. Suppose $b= \pm p^l$ for an $l\in \N$ and let $a$ be an integer. Let $B_{p,j} = G_{p^j}(\cS_{p,0})$. Note that $B_{p,j}\subset \cS_{p,j}$. If $b$ is odd, by Proposition \ref{cor9}  there is a unique $b$th root of $q$ in $\cS_{p,0}$; we denote it by $x_{b;0}$. If $b$ is even, by Proposition \ref{cor9} there are exactly two $b$th root of $q$, namely $\pm q^{1/b}$. We put $x_{b;0}=q^{1/b}$. 
\\

We define the element $z_{b,a} \in \cS_p$ as follows:
\begin{itemize}
\item If $b\mid a$, let $z_{b,a} := q^{a^2/b} \in \cS_p$.

\item If $ b=\pm p^l \nmid a$, then $z_{b,a}\in \cS_p$ is defined by specifying its projections  $\pi_j(z_{b,a}):=z_{b,a;j}\in \cS_{p,j}$  as follows. Suppose $a = p^s e$, with $(e,p)=1$. Then $s<l$. 

For $j >s$ let $z_{b,a;j}:=0$.

For $ 0\le j \le s$ let
$
z_{b,a;j} := [G_{p^j}(x_{b;0})]^{a^2/p^j} = [G_{p^j}(x_{b;0})]^{e^2 \, p^{2s-j}}
\in B_{p,j} \subset \cS_{p,j}.
$
\end{itemize}

$\text{ }$

Similarly, for $b=\pm p^l$ we define the element $x_b \in \cS_p$ as follows:
\begin{itemize}
 \item We put $\pi_0(x_b): =x_{b;0}$. 
 \item For $j<l$, $\pi_j(x_b):= [G_{p^j}(x_{b;0})]^{p^j}$. 
 \item If $j\geq l$, $\pi_j(x_b):=q^{b}$. 
\end{itemize}
Notice that for $c=(b,p^j)$ we have
\be\label{piOfxb}
\pi_j(x_b)=z_{b,c;j}.
\ee

\begin{proposition} \label{eval_z}
Suppose $\xi$ is a root of unity of order $r = c r'$, where $c= (r,b)$. Then 
\[
\ev_{\xi}(z_{b,a}) = \begin{cases} 0 & \text{ if } c \nmid a \\
(\xi^c)^{a_1^2 b'_{*r'}} &  \text{ if } a=ca_1, 
\end{cases}
\]
where $b'_{*r'}$ is the unique element in $\Z/r'\Z$ such that $b'_{*r'} (b/c) \equiv 1 \pmod{r'}$. Moreover, 
\[
\ev_\xi(x_b)=(\xi^c)^{b'_{*r'}}\; .
\]
\end{proposition}

\begin{proof} Let us compute $\ev_\xi(z_{b,a})$. The case of $\ev_\xi(x_b)$ follows then from \eqref{piOfxb}.

If $b\mid a$, then $c\mid a$, and the proof is obvious.

Suppose $b\nmid a$. Let $a=p^s e$ and $c= p^i$. Then $s <l$. Recall that $z_{b,a}=\prod_{j=0}^\infty z_{b,a;j}$. By Lemma \ref{1100}, 
\[
\ev_\xi(z_{b,a}) = \ev_\xi(z_{b,a;i}).
\]

If $c\nmid a$, then $i > s$. By definition, $z_{b,a;i}=0$, hence the statement holds true.  
The case $c\mid a$, i.e.  $ i \le s$, remains. Note that $\zeta = \xi^c$ is a primitive root of order $r'$ and  $(p, r')=1$. Since $z_{b,a;i} \in B_{p,i}$, 
\[
\ev_\xi(z_{b,a;i}) \in \Z[1/p][\zeta].
\]
From the definition of $z_{b,a;i}$ it follows that $(z_{b,a;i})^{b/c}=(q^c)^{a^2/c^2}$, hence after evaluation we have 
\[
[\ev_\xi(z_{b,a;i})]^{b/c} =  (\zeta)^{a^2_1}.
\]
Note also that 
\[
[(\xi^c)^{a^2_1 b'_{*r'}}]^{b/c} = (\zeta)^{a^2_1}.
\]
Using Lemma \ref{0923} we conclude $\ev_\xi(z_{b,a;i})=(\xi^c)^{a^2_1 b'_{*r'}}$ if $b$ is odd, and $\ev_\xi(z_{b,a;i}) =(\xi^c)^{a^2_1 b'_{*r'}}$ or $\ev_\xi(z_{b,a;i}) =-(\xi^c)^{a^2_1 b'_{*r'}}$ if $b$ is even. Since $\ev_{1}(q^{1/b})=1$ and therefore $\ev_{\xi}(q^{1/b})=\xi^{b_{*r}}$ (and not $-\xi^{b_{*r}}$), we get the claim. 
\end{proof}

\chapter{Unified invariant of lens spaces}\label{LensSpaces}

The purpose of this chapter is to prove Lemma \ref{ExistensLensSpaces}. Recall that $M(b,a;d)$ is the lens space $L(b,a)$ together with an unknot $K$ colored by $d$ inside (see Figure \ref{fig:LensSpaceKnot}). In Section \ref{SectionWRTLensSpaces}, we compute  the renormalized quantum invariant of $M(b,a;d)$ for arbitrary $d$. We then define in Section \ref{UnifiedLensSpaces} the unified invariant of $M(b,a;d(\epsilon))$ (see Section \ref{DefinitionLensSpaces} for the definition of $d(\epsilon)$).
\\

Let us introduce the following notation. 
For $a,b\in \Z$, the Dedekind sum (see e.g. \cite{KM-Dedekind}) is defined by
\[
 s(a,b)=\sum_{n=0}^{b-1} \left( \left( \frac{n}{b} \right) \right) \left( \left( \frac{an}{b} \right) \right)
\]
where 
\[
((x))= 
\begin{cases} 
x-\lfloor x\rfloor - \frac{1}{2} &\mbox{if } x\in\mathbb{R}\setminus\mathbb{Z}\\ 0&\mbox{if } x\in\mathbb{Z} 
\end{cases}
\]
and $\lfloor x \rfloor$ denotes the largest integer not greater than $x$. 

For $n,m\in \Z$ coprime and $0<|n|<|m|$, we define $n_{*m}$ and $m_{*n}$ such that
\[
nn_{*m}+mm_{*n}=1, \text{ with } 0<\sn(n)n_{*m}<|m|.
\]
Notice that for $n=1$, $m>1$ we have $1_{*m}=1$ and $m_{*1}=0$.

Let $r$ be a fixed integer denoting the order of $\xi$, a primitive root of unity. If $G=SO(3)$, $r$ is always assumed to be odd.

When we write $\pm$ respectively $\mp$ in a formula, one can either choose everywhere the upper or everywhere the lower signs and the formula holds in both cases.

\section{The quantum invariant of lens spaces with colored unknot inside}
\label{SectionWRTLensSpaces}

\begin{proposition}\label{WRTGLensSpaces}
 Suppose $c=(b,r)$ divides $|a|d\pm 1$. Then
\be\label{tau'SO3}
\tau'^{SO(3)}_{M(b,a;d)}(\xi)= (-1)^{\frac{c+1}{2}\,\frac{\sn(ab)-1}{2}}
\left(\frac{|a|}{c}\right)
\left(\frac{1-\xi^{\pm\sn(a)db_{*r}}}{1-\xi^{\pm\sn(b)b_{*r}}}\right)^{\chi(c)}
\xi^{
4_{*r}u^{SO(3)}-4_{*r}b'_{*r'} \frac{a(\pm a_{*b}-\sn(a)d)^2}{c}
}
\ee
where
\begin{eqnarray}\label{uSO3}
u^{SO(3)}&:=&
12s(1,b)-12 \sn(b)s(a,b) \\
&&\hspace{1.5cm}+\frac{1}{b}\left(a(1-d^2)+2(\mp \sn(a)d-\sn(b))+a(a_{*b}\pm\sn(a)d)^2\right) \in \Z \nonumber 
\end{eqnarray}
and for $b$ odd 
\be\label{tau'SU2}
\tau'^{SU(2)}_{M(b,a;d)}(\xi) = 
(-1)^{\frac{b'+1}{2}\frac{\sn(ab)-1}{2}}
\left(\frac{|a|}{|b'|}\right)
\left(\frac{
1-\xi^{\pm \sn(a)db_{*r}}
}{
1-\xi^{\pm \sn(b)b_{*r}}}\right)^{\chi(c)}
\xi^{\frac{u^{SU(2)}}{4}-\frac{b'_{*r'}a_{*b}(\sn(a)ad\pm 1)^2(\sn(b)b'-1)^2}{4c}}
\ee
where 
\begin{eqnarray}\label{uSU2}
u^{SU(2)}&:=&12s(1,b)-12\sn(b)s(a,b)+\frac{1}{b}(a(1-d^2))\\
&&\hspace{1.5cm}+ \frac{1}{b}(2(\mp\sn(a)d-\sn(b))+a_{*b}(\sn(a)ad\pm 1)^2(\sn(b)b'-1)^2) \in\Z
\nonumber
\end{eqnarray}
and $\chi(c)=1$ if $c=1$ and is zero otherwise.
If $c\nmid (\sn(a)ad\pm 1)$, $\tau'^{G}_{M(b,a;d)}(\xi)=0\; .$
\label{lens}
\end{proposition}

In particular, it follows that $\tau'^{G}_{L(b,a)}(\xi)=0$ if $c\nmid |a|\pm 1$.

\begin{remark}
For $G=SU(2)$, the quantum $SU(2)$ invariant of $M(b,a;d)$ is in general dependent on a $4$th root of $\xi$ (denoted by $\xi^{\frac{1}{4}}$). Here, we have only calculated the (renormalized) quantum invariant for a certain $4$th root of $\xi$, namely $\xi^{\frac{1}{4}}=e_{4r}^{l}$ for $\xi=e_r^{l}$ where $l$ and $r$ are coprime.
For the definition of the unified invariant of lens spaces in Section \ref{UnifiedLensSpaces}, we will choose $d$ such that the quantum invariant is \emph{independent} of the $4$th root of $\xi$.
\end{remark}

The rest of this section is devoted to the proof of Proposition \ref{WRTGLensSpaces}. 

\subsection{The positive case}
To start with, we consider the case when
$b,a>0$. Since two lens spaces $L(b,a_1)$ and $L(b,a_2)$ are homeomorphic if $a_1 \equiv a_2 \pmod{b}$, we can assume $a<b$. Let $b/a$ be given by  a continued fraction
\[
\frac{b}{a}=m_n-\frac{1}{\displaystyle m_{n-1}-\frac{1}{\displaystyle
    		m_{n-2}-\dots \frac{1}{\displaystyle m_2-\frac{1}{\displaystyle m_1}}}}.
\]
We can  assume $m_i\geq 2$ for all $i$ (see \cite[Lemma 3.1]{Je}).

Representing the $b/a$--framed unknot in Figure \ref{fig:LensSpaceKnot}
by a Hopf chain (as e.g. in Lemma 3.1 of \cite{BL}), $M(b,a;d)$ is obtained by integral surgery along the link $L_{M(b,a;d)}$ in Figure \ref{fig:HopfChain}, where the $m_i$ are the framing coefficients and $K_d$ denotes the unknot with fixed color $d$ and zero framing.

\begin{figure}[h]
\begin{center}
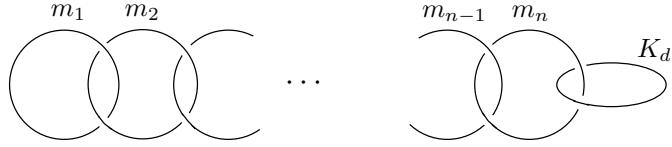
\caption{Surgery link $L_{M(b,a;d)}$ of $M(b,a;d)$ with integral framing.}
\label{fig:HopfChain}
\end{center}
\end{figure}

The colored Jones polynomial of $L_{M(b,a;d)}$ is given by:
\[
J_{L_{M(b,a;d)}}(j_1,\ldots,j_n,d)=
[j_1]
\cdot\prod_{i=1}^{n-1} \frac{[j_i j_{i+1}]}{[j_i]}
\cdot\frac{[j_n d]}{[j_n]}
\]
(see e.g. \cite[Lemma 3.2]{KM}). Applying \eqref{F} and taking the framing into account (let $L_{M(b,a;d),0}$ denote the link $L_{M(b,a;d)}$ with framing zero everywhere), we get
\[
F^G_{L_{M(b,a;d),0}}(\xi)={\sum_{j_i}}^{\xi,G}\prod_{i=1}^{n}
q^{m_i  \frac{j_i^2-1}{4}}
\prod_{i=1}^{n-1} [j_i j_{i+1}]\cdot[j_n d][j_1]
=\ev_\xi\left(\frac{q^{-\frac{1}{4}\sum_{i=1}^{n}m_i}}{(q^{\frac{1}{2}}-q^{-\frac{1}{2}})^{n+1}}\right) \cdot S_n(a,d,\xi)
\]
where
\[
S_n(a,d,\xi)
=
q^{\frac{1}{4}\sum m_ij^2_i}
(q^{\frac{1}{2}j_1}-q^{-\frac{1}{2}j_1})
(q^{\frac{1}{2}j_1j_2}-q^{-\frac{1}{2}j_1j_2})\dots
(q^{\frac{1}{2}j_{n-1}j_n}-q^{-\frac{1}{2}j_{n-1}j_n})
(q^{\frac{1}{2}j_n d}-q^{-\frac{1}{2}j_n d})\, .
\]
Inserting these formulas into the Definition \eqref{eq:DefinitionOfWRT} using $\sigma_+=n$ and $\sigma_-=0$ (compare \cite[p. 243]{KM-Dedekind}) as well as \eqref{eq:F_U_b}, we get
\[
 \tau_{M(b,a;d)}(\xi)=
\frac{
\ev_{\xi}\left(q^{-\frac{1}{4}\sum_{i=1}^{n}m_i+\frac{n}{2}}\right) \cdot S_n(a,d,\xi)
}{
(-2)^n(\gamma_1^G(\xi))^n \ev_{\xi}(q^{\frac{1}{2}}-q^{-\frac{1}{2}})
}.
\]
We restrict us now to the root of unity $e_r=\exp(\frac{2\pi i}{r})$. We look at arbitrary primitive roots of unity in Subsection \ref{ArbitraryRoots}.

Applying \eqref{Gamma1SO3}, \eqref{Gamma1SU2} and
the following
formula for the Dedekind sum (compare \cite[Theorem 1.12]{KM-Dedekind})
\be\label{Dedekind}
3n- \sum_i m_i=-12 s(a,b)+\frac{a+a_{*b}}{b},
\ee
we get 
\be\label{tauM(b,a;d)}
  \tau_{M(b,a;d)}(e_r)=
\ev_{e_r}\left(
\frac{
q^{\frac{1}{4}(-12s(a,b)+\frac{a+a_{*b}}{b})}
}{
q^{\frac{1}{2}}-q^{-\frac{1}{2}}
}
\right)
\cdot 
\frac{
S_n(a,d,e_r)
}{
(-2)^n \sqrt{r}^n 
\epsilon^G(r)^{n}},
\ee
where $\epsilon^{SO(3)}(r)=1$ if $r\con 1 \pmod{4}$, $\epsilon^{SO(3)}(r)=i$ if $r\con 3\pmod{4}$  and $\epsilon^{SU(2)}(r)=1+i$.

Finally, for the renormalized version, we have to divide by
\[
\tau^G_{L(b,1)}(e_r)= 
\ev_{e_r}\left(
\frac{
q^{\frac{1}{4}(-12s(1,b)+\frac{2}{b})}
}{
q^{\frac{1}{2}}-q^{-\frac{1}{2}}
}
\right)
\cdot 
\frac{
S_1(1,1,e_r)
}{
(-2) \sqrt{r}\epsilon^G(r)
}.
\]
This gives
\be\label{tau'M(b,a;d)}
\tau'_{M(b,a;d)}(e_r)=
(-2\sqrt{r}\;\epsilon^G(r))^{1-n} 
\cdot
\frac{
S_n(a,d,e_r)
}{
S_1(1,1,e_r)
}
\cdot
\ev_{e_r}\left(
q^{\frac{1}{4}(-12s(a,b)+\frac{a+a_{*b}}{b}+12s(1,b)-\frac{2}{b})}
\right).
\ee

We put $S_n(d):=S_n(a,d,e_r)$. To calculate $S_n(d)$, we need to look separately at the $SO(3)$ and the $SU(2)$ case.

\subsubsection{The $SO(3)$ case}
We follow the arguments of \cite{LiLi}. 
The $\tau_{M(b,a;d)}(e_r)$ can be computed in the same way as the invariant $\xi_r(L(b,a),A)$ in \cite{LiLi}, after replacing $A^2$ (respectively $A$) by $e_r^{2_{*r}}$ (respectively $e_r^{4_{*r}}$).

Using Lemmas  4.11, 4.12 and 4.20 of \cite{LiLi}\footnote{There are misprints
in Lemma 4.21 of \cite{LiLi}:
$q^*\pm n$ should be replaced by $q^*\mp n$ for
$n=1,2$.}
(and replacing
 $e_r$ by $e_r^{4_{*r}}$, $c_n$ by $c$, $N_{n,1}=p$ by $b$,
$N_{n-1,1}=q$  by $a$, $N_{n,2}=q^*$
 by $a_{*b}$ and $-N_{n-1,2}=p^*$ by $b_{*a}$), we get
\begin{eqnarray}\label{Snd}
S_n(d)&=&(-2)^n (\sqrt{r} \epsilon(r))^n \sqrt{c} \epsilon(c)
\left(\frac{\frac{b}{c}}{\frac{r}{c}}\right)\left(\frac{a}{c}\right)
(-1)^{\frac{r-1}{2}\frac{c-1}{2}}
\\
&&\hspace{4cm}
\cdot\sum_{\pm}\chi^{\pm}(d)
e_r^{-ca4_{*r}b'_{*r'}
\left(\frac{d\mp a_{*b}}{c}\right)^2
\pm 2_{*r}b_{*a}(d\mp a_{*b}) +4_{*r}a_{*b} b_{*a}}\nonumber
\end{eqnarray}
where $\epsilon(x)=1$ if $x \equiv 1 \pmod 4$ and $ \epsilon(x)=i$ if $x \equiv 3 \pmod 4$. 
Further, $\chi^{\pm}(d)=\pm 1$ if $c\mid d\mp a_{*b}$ and is zero otherwise. Since $(a,c)=(a,b)=1$ and $a(d\mp a_{*b})=ad\mp aa_{*b}= ad \mp (1-bb_{*a})$ and $c\mid b$, we have
$c\mid d\mp a_{*b}$ if and only if $c\mid ad\mp 1$.
This implies the last claim of Proposition \ref{WRTGLensSpaces} for $G=SO(3)$.

Note that when $c=1$, both $\chi^\pm(d)$ are nonzero.
If $c>1$ and  $c \mid (d-a_{*b})$, $\chi^+(d)=1$, but $\chi^- (d)=0$. Indeed, for $c$ dividing $d-a_{*b}$, $c \mid (d+a_{*b})$ if and only if $c\mid a_{*b}$, which is impossible, because $c\mid b$ but $(b,a_{*b})=1$. For the same reason, if $c\mid d+a_{*b}$, then $\chi^+(d)=0$ and $\chi^{-}(d)=-1$.

Inserting \eqref{Snd} into \eqref{tau'M(b,a;d)}
we get
\[
\tau'^{SO(3)}_{M(b,a;d)}(e_r)=
\left(\frac{a}{c}\right)
\left(\frac{1-e_r^{\pm db'_{*r'}}}{1-e_r^{\pm b'_{*r'}}}\right)^{\chi(c)}
e_r^{4_{*r}u-4_{*r}b'_{*r'}\frac{a(d\pm a_{*b})^2}{c}}
\]
where
\[
u=
-12s(a,b)+12s(1,b)+
\frac{1}{b}\left(a+a_{*b}
-2
-b_{*a}b(a_{*b}\pm 2d)\right).
\]
Notice that $u\in\Z$. Further observe that by using $aa_{*b}+bb_{*a}=1$, we get
\[
a+a_{*b}-2-b_{*a}b(a_{*b}\pm 2d)=
2(\mp d-1)+a(1-d^2)+a(a_{*b}\pm d)^2.
\]
Since $\tau_{M(b,a;d)}(\xi)=\tau_{M(-b,-a;d)}(\xi)$, this implies the result \eqref{tau'SO3} for  $0<a<b$ and $0>a>b$ for the root of unity $e_r$.

\subsubsection{The SU(2) case}\label{SU(2)Case}

The way we calculate $S_n(d)$ in the $SO(3)$ case can not be adapted to the $SU(2)$ case. We use instead a Gauss sum reciprocity formula following the arguments of \cite{Je}.

We use a well--known result by Cauchy and Kronecker.
\begin{proposition}[Gauss sum reciprocity formula in one dimension]\label{GSRF}
For $m,n, \psi, \varphi \in \Z$ such that $nm$ is even and $\varphi \mid n\psi$
we have
\[
\sum_{\lambda =0}^{n-1} e_{2n}^{m\lambda^2}e_{\varphi}^{\psi \lambda}
=(1+i)\sqrt{\frac{n}{2m}}\sum_{\lambda=0}^{m-1} e_{2m\varphi^2}^{-n(\lambda\varphi+\psi)^2}.
\]
In particular for $n=\varphi$ even, we have
\[
\sum_{\lambda=0}^{n-1}e_{2n}^{m\lambda^2+2\psi\lambda}=(1+i)\sqrt{\frac{n}{2m}}
\sum_{\lambda=0}^{m-1}e_{2mn}^{-(n\lambda+\psi)^2}.
\]
\end{proposition}

\begin{proof}
A proof can be found in e.g. \cite[Chapter IX, Theorem 1]{Ch}. For a generalization of this result see also \cite[Propositions 2.3 and 4.3]{Je} and \cite{DT}.
\end{proof}

\begin{lemma}\label{SndInduction}
For $a,b>0$, we have
\[
S_n(d)= C_n
\sum_{\substack{\gamma =0\\ \gamma\equiv d \hspace{-2mm}\pmod{2r}}}^{2rb-1} 
\left(e_{4abr}^{-(a\gamma-1)^2}-e_{4abr}^{-(a\gamma+1)^2}\right)
\]
where
\[
C_n:=(-2(1+i)\sqrt{r})^n \frac{1}{\sqrt{b}} 
e_{4ar}^{b_{*a}}.
\]
\end{lemma}

\begin{proof}
We prove the claim by induction on $n$. Notice first, that
for $x,y,z \in \Z$ we have
\begin{eqnarray}\label{replace-j}
\sum_{j=0}^{2r-1}e_r^{zj^2}(e_r^{xj}-e_r^{-xj})(e_r^{yj}-e_r^{-yj})
=
2\sum_{j=0}^{2r-1}e_r^{zj^2}(e_r^{(y+x)j}-e_r^{(y-x)j})
\end{eqnarray}
by replacing $j$ by $-j$ in $-\sum_{j=0}^{2r-1}e_r^{zj^2}(e_r^{xj}-e_r^{-xj})e_r^{-yj}$.
Using this and Proposition \ref{GSRF},
we have for $n=1$
\begin{eqnarray*}
S_1(d)&=&\sum_{j=0}^{2r-1}e_{4r}^{bj^2}(e_{2r}^{j}-e_{2r}^{-j})(e_{2r}^{jd}-e_{2r}^{-jd})\\
&=& 2\sum_{j=0}^{2r-1}e_{4r}^{bj^2}(e_{2r}^{j(d+1)}-e_{2r}^{j(d-1)})\\
&=& 2(1+i)\sqrt{\frac{r}{b}}\sum_{j=0}^{b-1}(e_{4br}^{-(2rj+d+1)^2}-e_{4br}^{-(2rj+d-1)^2})\\
&=& 2(1+i)\sqrt{\frac{r}{b}}\sum_{\substack{\gamma =0\\ \gamma\equiv d\hspace{-2mm}\pmod{2r}}}^{2rb-1}
(e_{4br}^{-(\gamma+1)^2}-e_{4br}^{-(\gamma-1)^2}).
\end{eqnarray*}
Now, we set
\[
\frac{\tilde{b}}{\tilde{a}}=m_{n-1}-\frac{1}{m_{n-2}-\cdots\frac{1}{m_2-\frac{1}{m_1}}}.
\]
Notice that $\tilde{b}>\tilde{a}$. Assume the result of the lemma inductively. We have
\begin{eqnarray*}
S_n(d)&=&\sum_{j_n=0}^{2r-1}e_{4r}^{m_nj_n^2}(e_{2r}^{j_n d}-e_{2r}^{-j_n d}) S_{n-1}(j_n)\\
&=&\sum_{j_n=0}^{2r-1} e_{4r}^{m_nj_n^2}(e_{2r}^{j_nd}-e_{2r}^{-j_nd}) C_{n-1}
\sum_{\substack{\gamma=0 \\ \gamma\equiv j_n \hspace{-2mm}\pmod{2r}}}^{2\tilde{b}r-1}
(e_{4\tilde{a}\tilde{b}r}^{-(\tilde{a}\gamma -1)^2}-e_{4\tilde{a}\tilde{b}r}^{-(\tilde{a}\gamma +1)^2}) 
\end{eqnarray*}
We replace $j_n$ by $\gamma$ everywhere and get 
\begin{eqnarray*}
S_n(d)
&=&
C_{n-1}\sum_{\gamma=0}^{2\tilde{b}r-1}
e_{4r}^{m_n\gamma^2}(e_{2r}^{\gamma d}-e_{2r}^{-\gamma d})
(e_{4\tilde{a}\tilde{b}r}^{-\tilde{a}^2\gamma^2+2\tilde{a}\gamma-1}
-e_{4\tilde{a}\tilde{b}r}^{-\tilde{a}^2\gamma^2-2\tilde{a}\gamma-1}) \\
&=&
C_{n-1}
e_{4\tilde{a}\tilde{b}r}^{-1}
\sum_{\gamma=0}^{2\tilde{b}r-1}
e_{4\tilde{b}r}^{(\tilde{b}m_n-\tilde{a})\gamma^2}(e_{2r}^{\gamma d}-e_{2r}^{-\gamma d})
(e_{2\tilde{b}r}^{\gamma}
-e_{2\tilde{b}r}^{-\gamma}) \\
&=&
2C_{n-1}
e_{4\tilde{a}\tilde{b}r}^{-1}
\sum_{\gamma=0}^{2\tilde{b}r-1}
e_{4\tilde{b}r}^{(\tilde{b}m_n-\tilde{a})\gamma^2}
(e_{2\tilde{b}r}^{(\tilde{b}d+1)\gamma }
-e_{2\tilde{b}r}^{(\tilde{b}d-1)\gamma })\\
&=&
2C_{n-1}
e_{4\tilde{a}\tilde{b}r}^{-1}
\sum_{\gamma=0}^{2ar-1}
e_{4ar}^{b\gamma^2}
(e_{2ar}^{(ad+1)\gamma }-e_{2ar}^{(ad-1)\gamma })
\end{eqnarray*}
using first (\ref{replace-j}) and then $\tilde{b}=a$ and $\tilde{b}m_n-\tilde{a}=b$. Applying again 
Proposition \ref{GSRF} gives 
\begin{eqnarray*}
S_n(d)
&=&
2(1+i)C_{n-1}e_{4\tilde{a}\tilde{b}r}^{-1}\sqrt{\frac{ar}{b}}
\sum_{\lambda=0}^{b-1}(e_{4abr}^{-(2ar\lambda+ad+1)}-e_{4abr}^{-(2ar\lambda+ad-1)})\\
&=&
-2(1+i)C_{n-1}e_{4\tilde{a}\tilde{b}r}^{-1}\sqrt{\frac{ar}{b}}
\sum_{\substack{\gamma=0\\ \gamma\equiv d\hspace{-2mm}\pmod{2r-1}}}^{2rb-1}
(e_{4abr}^{-(a\gamma-1)}-e_{4abr}^{-(a\gamma+1)})
\end{eqnarray*}
where we put $\gamma=2r\lambda+d$. Since 
\[
1=\tilde{a}\tilde{a}_{*\tilde{b}}+\tilde{b}\tilde{b}_{*\tilde{a}}
=\tilde{a}_{*\tilde{b}}(am_n-b)+\tilde{b}_{*\tilde{a}}a
=a(\tilde{b}_{*\tilde{a}}+\tilde{a}_{*\tilde{b}}m_n)-b\tilde{a}_{*\tilde{b}}
=aa_{*b}+bb_{*a}
\]
and
\[
0<a(\tilde{b}_{*\tilde{a}}+\tilde{a}_{*\tilde{b}}m_n)
=\tilde{b}\tilde{b}_{*\tilde{a}}+\tilde{b}\tilde{a}_{*\tilde{b}}m_n
=1+\tilde{a}_{*\tilde{b}}(\tilde{b}m_n-\tilde{a})=1+\tilde{a}_{*\tilde{b}}b<
1+\tilde{b}b=1+ab
\]
we have $0<\tilde{b}_{*\tilde{a}}+\tilde{a}_{\tilde{*b}}m_n<b$ and therefore $b_{*a}=-\tilde{a}_{*\tilde{b}}$ . Therefore
\[
 e_{4\tilde{a}r}^{\tilde{b}_{*\tilde{a}}}e_{4\tilde{a}\tilde{b}r}^{-1}=
e_{4\tilde{a}\tilde{b}r}^{-\tilde{a}_{*\tilde{b}}\tilde{a}+1}e_{4\tilde{a}\tilde{b}r}^{-1}=
e_{4\tilde{b}r}^{-\tilde{a}_{*\tilde{b}}}=
e_{4ar}^{b_{*a}}
\]
and we get
\[
-2(1+i)C_{n-1}e_{4\tilde{a}\tilde{b}r}^{-1}\sqrt{\frac{ar}{b}}=
C_n.
\]
\end{proof}

For further calculations on $S_n(d)$ we need the following result.
\begin{lemma}\label{GS}
 For $x,y\in \N$ and $b$ odd with $(b,x)=1$ we have
\[
\sum_{j=0}^{b-1}e_b^{-xj^2-yj}
=\epsilon(b)\left(\frac{-x}{b}\right)\sqrt{b}e_b^{(b-1)^2\frac{x_{*b}y^2}{4}}.
\]
\end{lemma}

\begin{proof}
 This follows from Lemma \ref{GaussSumFormulas} using $e_b^{-1}=e_b^{(b-1)}$ with $(b,(b-1)x)=(b,x)=1$ and $(b-1)_{*b}\equiv -1\pmod{b}$.
\end{proof}

\begin{lemma}\label{SndSU(2)}
For $c=(b,r)$ and $b'=\frac bc$, $r'=\frac rc$ we have
\[
S_n(d)=(-2(1+i)\sqrt{r})^n\sqrt{c}\epsilon(b')\left(\frac{-ar'}{b'}\right)
\sum_{\pm}\chi(d)e_{4rb}^{-a_{*b}-ad^2\mp 2d}
e_{r'b'}^{-\frac{a_{*b}(ad\pm 1)^2(b'-1)^2(b'_{*r'}b'-1)}{4c^2}}
\]
where $\chi(d)=\left\{\begin{array}{ll} 0 & \text{ if } c \nmid ad\pm 1 \\ \mp 1 &\text{ if } c\mid ad\pm 1 \end{array}\right.$.
\end{lemma}

\begin{proof}
We put $\gamma=d+2r\lambda$ and get
\[
S_n(d)=C_n\sum_{\pm} \sum_{\lambda=0}^{b-1} \mp e_{4abr}^{-(a\gamma\pm 1)^2}.
\]
Notice that $(a\gamma \pm 1)^2=a^2d^2+1\pm 2ad + 4ar(ar\lambda^2+ad\lambda\pm \lambda)$ and therefore
\[
S_n(d)=C_n e_{4abr}^{-a^2d^2-1} \sum_{\pm}\sum_{\lambda=0}^{b-1} \mp e_{b}^{-ar\lambda^2-(ad\pm 1)\lambda}e_{2br}^{\mp d}
\]
and
\[
C_n e_{4rab}^{-a^2d^2-1} = (-2(1+i)\sqrt{r})^n \frac{1}{\sqrt{b}}e_{4rb}^{-a_{*b}-ad^2}.
\]
We use \cite[Theorem 2.1]{LiLi} as well as Lemma \ref{GS} to get
\begin{eqnarray*}
\mp\sum_{\lambda=0}^{b-1} e_{b}^{-ar\lambda^2-(ad\pm 1)\lambda}e_{2br}^{\mp d}
&=& \mp e_{2br}^{\mp d}\cdot c\sum_{\substack{\lambda=0\\c\mid ad\pm 1}}^{b'-1}
e_{b'}^{-ar'\lambda^2-\frac{(ad\pm 1)}{c}\lambda}\\
&=& \chi(d) c e_{2br}^{\mp d} \epsilon(b')\left(\frac{-ar'}{b'}\right)\sqrt{b'}
e_{b'r'}^{a_{*b}(1-b'b'_{*r})\left(\frac{ad\pm 1}{c}\right)^2\frac{(b'-1)^2}{4}}
\end{eqnarray*}
where we used that $a_{*b}\equiv a_{*b'}\pmod{b'}$, $(ar')_{*b'}\equiv a_{*b'}r'_{*b'}\pmod{b'}$ and 
$
e_{b'}^{r'_{*b'}}=e_{b'r'}^{r'_{*b'}r'}=e_{b'r'}^{1-b'b'_{*r'}}.
$
\end{proof}

Lemma \ref{SndSU(2)} implies the last claim of Proposition \ref{WRTGLensSpaces} for $G=SU(2)$.
\\

Since
\[
\left(\frac{
1-e_{4rb}^{\pm 4d}e_{rb}^{\pm a_{*b}ad(b-1)^2(b_{*r}b-1)}
}{
1-e_{rb}^{\pm 1}e_{rb}^{\pm (b-1)^2(b_{*r}b-1)}}\right)^{\chi(c)}
=
\left(\frac{
1-e_{rb}^{\pm a_{*b}ad(b_{*r}b-1)-d}
}{
1-e_{r}^{\pm b_{*r}}}\right)^{\chi(c)}
=\left(\frac{
1-e_r^{\pm db_{*r}}
}{
1-e_{r}^{\pm b_{*r}}}\right)^{\chi(c)},
\]
inserting the formula of $S^{SU(2)}_n(d)$ of Lemma \ref{SndSU(2)} into \eqref{tau'M(b,a;d)}, we get
\[
 \tau'^{SU(2)}_{M(b,a;d)}(e_r)
=
\left(\frac{a}{b'}\right)
\left(\frac{
1-e_r^{\pm db_{*r}}
}{
1-e_{r}^{\pm b_{*r}}}\right)^{\chi(c)}
\hspace{-1mm}
e_{4r}^{-12s(a,b)+12s(1,b)+\frac{1}{b}(a+a_{*b}-2)}
e_{4br}^{-a_{*b}-ad^2\mp 2d}
e_{br}^{\frac{-a_{*b}(ad\pm 1)^2(b'-1)^2(b'_{*r'}b'-1)}{4}}
.
\]
Therefore
\[
\tau'^{SU(2)}_{M(b,a;d)}(e_r)=
\left(\frac{a}{b'}\right)
\left(\frac{
1-e_r^{\pm db_{*r}}
}{
1-e_{r}^{\pm b_{*r}}}\right)^{\chi(c)}
e_{r}^{\frac{u}{4}-\frac{b'_{*r'}a_{*b}(ad\pm 1)^2(b'-1)^2}{4c}}
\]
where $u:=-12s(a,b)+12s(1,b)
+\frac{1}{b}(a(1-d^2)+ 2(\mp d-1)+a_{*b}(ad\pm 1)^2(b'-1)^2)$.
This implies \eqref{tau'SU2} for $0<a<b$ or $0>a>b$ at $e_r$. 

\subsection{The negative case}
To compute $\tau'^{G}_{M(-b,a;d)}(e_r)$, observe that, since $L(b,a)$ and $L(-b,a)$ are homeomorphic with opposite orientation, $\tau^{G}_{M(-b,a;d)}(\xi)=\tau^{G}_{M(b,-a;d)}(\xi)$ is equal to the complex conjugate of $\tau^{G}_{M(b,a;d)}(\xi)$ due to Theorem \ref{PropertiesOfWRTInvariant}. The ratio
\[
\tau'^{SO(3)}_{M(-b,a;d)}(\xi)=\frac{\;\;\overline{\tau^{SO(3)}_{M(b,a;d)}(\xi)}\;\;}
{\tau^{SO(3)}_{L(b,1)}(\xi)}
\]
can be computed analogously to the positive case. Using $\overline{\epsilon(c)}=(-1)^{\frac{c-1}{2}}\epsilon(c)$,
 we have for $a,b>0$
\[
\tau'^{SO(3)}_{M(-b,a,d)}(e_r) =(-1)^{\frac{c+1}{2}}
\left(\frac{a}{c}\right)
\left(\frac{1- e_r^{\mp db_{*r}}}{1-e_r^{\pm b_{*r}}}\right)^{\chi(c)}
e_r^{4_{*r}\tilde{u}^{SO(3)}+4_{*r}b'_{*r'}\frac{a(d\pm a_{*b})^2}{c}}
\]
where
\begin{eqnarray*}
\tilde{u}^{SO(3)}&=&12s(a,b)+12s(1,b)+\frac{1}{b}
\left(
-a-a_{*b}-2+b_{*a}b(a_{*b}\pm 2d)
\right)\\
&=&12s(a,b)+12s(1,b)+\frac{1}{b}
\left(
2(\pm d-1)-a(1-d^2)-a(a_{*b}\pm d)^2
\right)
\end{eqnarray*}
and
\begin{eqnarray*}
 \tau'^{SU(2)}_{M(-b,a:d)}(e_r)
&=&
(-1)^{\frac{b'+1}{2}}\left(\frac{a}{b'}\right)
\left(\frac{1-e_r^{\mp db_{*r}}}{1-e_r^{\pm b_{*r}}}\right)^{\chi(c)}
e_r^{\frac{\tilde{u}^{SU(2)}}{4}+\frac{b'_{*r'}a_{*b}(ad\pm 1)^2(b'-1)^2}{4c}}
\end{eqnarray*}
where
\[
\tilde{u}^{SU(2)}=12 s(a,b)+12 s(1,b)+ \frac{1}{b}(-a(1-d^2)+2(\pm d-1)-a_{*b}(ad\pm 1)^2(b'-1)^2).
\]
Using $s(a,b)=s(a,-b)=-s(-a,b)$, we get the claim of Proposition \ref{WRTGLensSpaces} at $e_r$ if either $a$ or $b$ is negative. 

\subsection{Arbitrary primitive roots of unity}\label{ArbitraryRoots}
To get the result of Proposition \ref{WRTGLensSpaces} for an arbitrary primitive root of unity $\xi$ of order $r$, notice that we can regard $\tau'^{G}$ as a map 
\[
\tau'^{G}:\{\xi^{\frac{1}{4}} \in \mathbb{C} \mid \xi=e_r^l, (r,l)=1\} \to \Q(e_{4r})
\]
and notice that for $\xi=e_r^l$ for some $l$ coprime to $r$, we have the Galois transformation
\[
 \varphi:\Q(e_{4r}) \to \Q(\xi^{\frac{1}{4}}), \hspace{1cm} e_{4r}\mapsto e_{4r}^l
\]
which is a ring isomorphism and maps $e_r$ to $\xi$. Therefore 
$\tau'^{G}_{M(b,a;d)}(\xi)=\tau'^{G}_{M(b,a;d)}(\varphi(e_r))=\varphi(\tau'^{G}_{M(b,a;d)}(e_r))$ and we get \eqref{tau'SO3} and \eqref{tau'SU2} in general.

\qed

\begin{example*} For $b>0$, we have
\[
\tau'^{SO(3)}_{L(-b,1)}(\xi)
=
(-1)^{\frac{c+1}{2}-{\chi(c)}}\,
\xi^{2_{*r}(b-3)+b_{*r}\chi(c)}
\quad 
\text{ and }
\quad
\tau'^{SU(2)}_{L(-b,1)}(\xi)
=(-1)^{\frac{b'-1}{2}+\chi(c)}\xi^{\frac{b-3}{2}+b_{*r}\chi(c)}.
\]
\end{example*}

\section{Proof of Lemma \ref{ExistensLensSpaces}} \label{UnifiedLensSpaces}
Since $L(-b,a)$ and $L(b,-a)$ are homeomorphic, we can assume $b$ to be positive, i.e. $b=p^l$ for a prime $p$. We have to define the unified invariant of   $M^\ve(b,a) := M(b,a;d(\ve))$ where $d(0)=1$ and $d(\bz)$ is the smallest odd positive integer such that $\sn(a)ad(\bz) \equiv 1 \pmod {b}$. 

Recall from Section \ref{qthroot} that we denote the unique positive $b$th root of $q$ in $S_{p,0}$ by $q^{\frac{1}{b}}$. For $p\neq 2$, we define the unified invariant $I^{G}_{M^\ve(b,a)}\in \cS_p$ by specifying its projections
\[
\pi _j I^{SO(3)}_{M^\ve(b,a)}
:=
\,
\begin{cases}
q^{3 s(1,b)-3\sn(b)\, s(a,b)}
&\text{if } j=0, \;\ve =0
\\&\\
(-1)^{\frac{p^j+1}{2}\,\frac{\sn(a)-1}{2}}
\left(\frac{|a|}{p}\right)^j \,
q^{\frac{u'^{SO(3)}}{4}}
&\text{if } 0<j<l, \;\ve=\bz
\\&\\
(-1)^{\frac{p^l+1}{2}\,\frac{\sn(a)-1}{2}}
\left(\frac{|a|}{p}\right)^l q^{\frac{u'^{SO(3)}}{4}}
&\text{if } j\geq l, \;\ve=\bz
\end{cases}
\]
where $u'^{SO(3)}:=u^{SO(3)}-\frac{a(a_{*b}-\sn(a)d(\bz))^2}{b}$ and $u^{SO(3)}$ is defined in  \eqref{uSO3} and 
\[
\pi_j I^{SU(2)}_{M^{\epsilon}(b,a)}:=\left\{\begin{array}{ll}				(-1)^{\frac{b+3}{2}\frac{\sn(a)-1}{2}}\left(\frac{|a|}{p}\right)^{l}q^{3s(1,b)-3\sn(b)s(a,b)}&\text{if } j=0, \epsilon=0\\ (-1)^{\frac{p^{l-j}+1}{2}\frac{\sn(a)-1}{2}}\left(\frac{|a|}{p}\right)^{l-j}q^{\frac{u'^{SU(2)}}{4}} & \text{if } 0<j<l, \varepsilon=\bar{0}\\
				(-1)^{\frac{\sn(a)-1}{2}}q^{\frac{u'^{SU(2)}}{4}} & \text{if } j>l, \varepsilon=\bar{0}
                             \end{array}\right.
\]
where $u'^{SU(2)}:=u^{SU(2)}-\frac{a_{*b}(\sn(a)ad-1)^2(\sn(b)b'-1)^2}{b}$ and $u^{SU(2)}$ is defined in \eqref{uSU2}. 

For $G=SO(3)$ and $p=2$, only  $\pi_0 I^{SO(3)}_{M(b,a)}\in \cS_{2,0}=\cR_2$ is non--zero and it is defined to be $q^{3s(1,b)-3  s(a,b)}$. 

The $I^G_{M^{\varepsilon}(b,a)}$ is well--defined due to Lemma \ref{0622} below, i.e. all powers of $q$ in $I^G_{M^{\ve}(b,a)}$ are integers for $j>0$ or lie in $\frac{1}{b}\Z$ for $j=0$. Unlike the invariant for arbitrary $d$, there is no dependency on the $4$th root of $q$. Further, for b odd (respectively even) $I^G_{M^\ve(b,a)}$ is invertible in $\cS_{p}^{p,\varepsilon}$ (respectively $\cR_p^{p,\varepsilon}$) since $q$ and $q^{\frac{1}{b}}$ are invertible in these rings. 

In particular, for odd $b=p^l$, we have $I^G_{L(b,1)}=1$ and
\[
\pi _j I^{SO(3)}_{L(-b,1)}
=
\,
\begin{cases}
q^{\frac{b-3}{2}+\frac{1}{b}}
&\text{if } j=0
\\&\\
(-1)^{\frac{p^j+1}{2}}q^{\frac{b-3}{2}}
&\text{if } 0<j<l, \; p \text{ odd} \\&\\
(-1)^{\frac{p^l+1}{2}}q^{\frac{b-3}{2}}
&\text{if } j\geq l, \; p \text{ odd} \, 
\end{cases}
\]
and 
\[
\pi _j I^{SU(2)}_{L(-b,1)}
=
\,
\begin{cases}
(-1)^{\frac{-b+3}{2}}q^{\frac{b-3}{2}+\frac{1}{b}}
&\text{if } j=0
\\&\\
(-1)^{\frac{p^{l-j}+1}{2}}q^{\frac{b-3}{2}}
&\text{if } 0<j<l, \; p \text{ odd} \\&\\
q^{\frac{b-3}{2}}
&\text{if } j\geq l, \; p \text{ odd} \, .
\end{cases}
\]

It is left to show that for any $\xi$ of order $r$ coprime with $p$, we have
\[
\ev_\xi ( I^G_{M^0(b,a)})=\tau'^{G}_{M^0(b,a)}(\xi)\,
\]
and, if $r=p^j k$ with $j>0$, then
\[
\ev_\xi ( I^G_{M^{\bz}(b,a)})=\tau'^{G}_{M^\bz(b,a)}(\xi)\,.
\]
For $\ve=0$, this follows directly from Propositions \ref{eval_z} and \ref{WRTGLensSpaces} with $c=d=1$ using
\[
\frac{1-\xi^{-\sn(a)b_{*r}}}{1-\xi^{-b_{*r}}}=\begin{cases} 1 &\text{if } \sn(a)=1 \\ 
							-\xi^{-\sn(a)b_{*r}} &\text{if } \sn(a)=-1
                                             \end{cases}.
\]
For $\ve=\bz$, we have $c=(p^j,b)>1$ and we get the claim by using Proposition \ref{WRTGLensSpaces} and for the $SO(3)$ case
\be \label{0623} 
\xi^{\frac{a(a_{*b}-\sn(a)d(\bz))^2}{b}}=
\xi^{c\,\frac{a(a_{*b}-\sn(a)d(\bz))^2}{bc}}=
\xi^{bb'_{*r'}\,\frac{a(a_{*b}-\sn(a)d(\bz))^2}{bc}}=
\xi^{b'_{*r'}\,\frac{a(a_{*b}-\sn(a)d(\bz))^2}{c}},
\ee
respectively for the $SU(2)$ case 
\begin{equation} \label{0625} 
\xi^{\frac{a_{*b}(\sn(a)ad-1)^2}{b}\cdot\frac{(b'-1)^2}{4}}=
\xi^{c\,\frac{a_{*b}(\sn(a)ad-1)^2}{bc}\cdot\frac{(b'-1)^2}{4}}=
\xi^{bb'_{*r'}\,\frac{a_{*b}(\sn(a)ad-1)^2}{bc}\cdot\frac{(b'-1)^2}{4}}=
\xi^{b'_{*r'}\,\frac{a_{*b}(\sn(a)ad-1)^2}{c}\cdot\frac{(b'-1)^2}{4}},
\end{equation}
where for the second equalities in \eqref{0623} and \eqref{0625} we use $c\equiv bb'_{*r'}\pmod{r}$. For $G=SO(3)$, notice that due to part \emph{(b)} of Lemma \ref{0622} below, $b$ and $c$ divide $a_{*b}-\sn(a)d(\bz)$  and therefore all powers of $\xi$ in \eqref{0623} are integers. For $G=SU(2)$, we can see that all powers of $\xi$ in \eqref{0625} are integers using part \emph{(c)} of Lemma \ref{0622} and the fact that $b'$ is odd and therefore $4 \mid (b'-1)^2$.
\qed
\\

The following Lemma is used in the  proof of Lemma \ref{ExistensLensSpaces}.

\newpage
\begin{lemma} \label{0622}
We have
\begin{itemize}
\item[(a)] $ 3 s(1,b) - 3\sn(b)\, s(a,b) \in \frac{1}{b}\Z$,
\item[(b)] $b\mid a_{*b}-\sn(a)d(\bz)$ and therefore $u'\in\Z$, 
\item[(c)] $b\mid \sn(a)ad-1$ and therefore $u'\in\Z$, and
\item[(d)] $4\mid u'$ for $d=d(\bz)$.
\end{itemize}
\end{lemma}

\begin{proof}
For claim \emph{(a)}, using the identity 
\[
 12bs(a,b)\equiv (b-1)(b+2)-4a(b-1) + 4 \sum_{j<\frac{b}{2}}\left\lfloor\frac{2aj}{b} \right\rfloor\pmod{8}
\]
(e.g. see \cite[Theorem 3.9]{Ap}), we get
\[
 12bs(a,b) - 12 bs(1,b)\equiv 4(1-a)(b-1) + 4 \sum_{j<\frac{b}{2}}\left\lfloor\frac{2aj}{b} \right\rfloor-\left\lfloor\frac{2j}{b} \right\rfloor\pmod{8}
\]
which is divisible by 4 in $\Z$. 

Claim \emph{(b)} follows from the fact that $(a,b)=1$ and 
\[
a(a_{*b}-\sn(a)d)=1-\sn(a)ad-bb_{*a}\equiv 0\pmod{b},
\]
since $d$ is chosen such that $\sn(a)ad\equiv 1\pmod{b}$, which also proves the Claim \emph{(c)}. For Claim \emph{(d)}, notice that for odd $d$ we have
\[
4\mid (1-d^2) \;\;\text{ and }\;\;
4\mid 2(\sn(a)d-\sn(b)).
\]
\end{proof}

\newpage

\chapter{Laplace transform} \label{laplace}

This chapter is devoted to the proof of Theorem \ref{Qbk} by using Andrew's identity. Throughout this chapter, let $p$ be a prime or $p=1$ and $b= \pm p^l$ for some $l\in \N$. 

\section{Definition of Laplace transform}
The Laplace transform is a
 $\Z[q^{\pm 1}]$--linear map
defined by
\begin{eqnarray*}
\cL_{b}: \Z[z^{\pm 1}, q^{\pm 1}] &\to& \cS_p \\
z^a &\mapsto& z_{b,a}.
\end{eqnarray*}
In particular, we put $\cL_{b;j}:=\pi_j \circ \cL_{b}$ and have $\cL_{b;j}(z^a)=z_{b,a;j} \in \cS_{p,j}$.

Further, for any
  $f \in \Z[z^{\pm 1},q^{\pm 1}]$ and $n\in \Z$, we define
\[
\hat f:=f|_{z=q^n}\in \Z[q^{\pm n}, q^{\pm 1}]\,.
\]

\begin{lemma}
Suppose  $f \in \Z[z^{\pm 1},q^{\pm 1}]$. Then
for any root of unity $\xi$ of order $r$ (odd for $G=SO(3)$),
\[
{\sum_{n}}^{\xi,G} q^{b\frac{n^2-1}{4}}  \hat f = \gamma^G_{b}(\xi) \,
\ev_\xi(\cL_{-b}(f)).
\]
\label{1001}
\end{lemma}

\begin{proof}
It is sufficient to consider the case $f=z^a$.
Then, by the same arguments as in
the proof of \cite[Lemma 1.3]{BL},
we
have
\be\label{four}
{\sum_{n}}^{\xi,G} q^{b\frac{n^2-1}{4}}\,
q^{ na} = \begin{cases}
0
&\text{if $c\nmid  a$}\\
(\xi^c)^{-{a_1^2 b'_*}}\, \gamma^G_b(\xi)  & \text{if $a=ca_1$}.
 \end{cases}
\ee
The result follows now from Proposition \ref{eval_z}.
\end{proof}

\section{Proof of Theorem \ref{Qbk}}
Recall that
\[
A(n,k) = \frac{\prod^{k}_{i=0}
\left(q^{n}+q^{-n}-q^i -q^{-i}\right)}{(1-q) \, (q^{k+1};q)_{k+1}}.
\]
We have to show that there exists an element $Q_{b,k} \in \cR_b$ (respectively $Q_{b,k} \in \cS_b$ if $b$ odd), such that for every root of unity $\xi$ of order $r$ (odd if $G=SO(3)$), one has
\[
\frac{{\sum_n}^{\hspace{-1.8mm}\xi,G } \; q^{b\frac{n^2-1}{4}} A(n,k) }{F_{U^b}(\xi)}
= \ev_\xi (Q_{b,k}).
\]
Applying Lemma \ref{1001} to $F_{U^b}(\xi)={\sum\limits_ n}^{\xi,G} q^{b\frac{n^2-1}{4}}
[n]^2$, we get for $c=(b,r)$
\be\label{1122}
F_{U^b}(\xi)=2\gamma^G_b(\xi) \;\ev_{\xi}\left(\frac{(1-x_{-b})^{\chi(c)}}{(1-q^{-1})(1-q)}\right)
,\ee
where, as usual, $\chi(c)=1$ if $c=1$ and zero otherwise. We will prove that for an odd prime $p$ and any number $j\geq 0$, there exists an element $Q_k(q,x_b,j) \in \cS_{p,j}$ such that 
\be\label{imp}
\frac{1}{(q^{k+1};q)_{k+1}}
\,\cL_{b;j}\left( \prod_{i=0}^k (z+z^{-1} - q^i -q^{-i})   \right)
=
2\, Q_k(q^{\sn(b)}, x_{b},j).
\ee
If $p=2$ we will prove the claim for $j=0$ only, since $\cS_{2,0}\simeq\cR_{2}$.

\begin{remark}
The case $p=\pm 1$ was already done e.g. in \cite{BBL}.
\end{remark}

Theorem \ref{Qbk} follows then from Lemma \ref{1001} and \eqref{1122}
where $Q_{b,k}$ is defined by its projections
\[
\pi_j Q_{b,k}:=\;
\frac{1-q^{-1}}{(1-x_{-b})^{\chi(p^j)}}\; Q_k(q^{-\sn(b)},x_{-b},j).
\]
We split the proof of (\ref{imp})
into two parts. In the first part we will show that
there exists an element $Q_{k}(q,x_b,j)$ such
 that  Equality \eqref{imp}
holds. In the second part we show that $Q_k(q,x_b,j)$
lies in $\cS_{p,j}$.

\subsection{Part 1: Existence of $Q_k(q,x_b,j)$,  $b$ odd case}
Assume $b=\pm p^l$ with $p\not=2$.
We split the proof into several lemmas.

\begin{lemma}\label{S_{b;j}(k,q)}
For $x_{b;j}:=\pi_j(x_b)$ and $c=(b,p^j)$,
\[
\cL_{b;j}\left( \prod_{i=0}^k (z+z^{-1} - q^i -q^{-i})   \right)=
2\, (-1)^{k+1} \, \qbinom{2k+1}{k} \, S_{b;j}(k,q)
\]
where
\begin{equation}\label{unif-s}
S_{b;j}(k,q):=1+\sum_{n=1}^{\infty}\frac{q^{(k+1)cn}(q^{-k-1};q)_{cn}}
{(q^{k+2};q)_{cn}}
(1+q^{cn}) x_{b;j}^{n^2}.
\end{equation}
\end{lemma}

Observe that for  $n> \frac{k+1}{c}$, the term $(q^{-k-1};q)_{cn}$ is zero and therefore the sum in \eqref{unif-s} is finite.

\begin{proof}
Since $\cL_{b}$ is invariant under $z \to z^{-1}$, one has
\[
\cL_b\left(\prod_{i=0}^k (z+z^{-1} - q^i -q^{-i})\right) =
-2\cL_{b} (z^{-k}(zq^{-k};q)_{2k+1}),
\]
and the $q$--binomial theorem (e.g. see \cite{GR}, II.3) gives
\begin{equation}\label{qbinomial}
z^{-k}(zq^{-k};q)_{2k+1}=
(-1)^k \sum_{i=-k}^{k+1}(-1)^i\qbinom{2k+1}{k+i}z^i.
\end{equation}
Notice that $\cL_{b;j}(z^a)\not=0$ if and only if $c \mid a$. Applying $\cL_{b;j}$ to the RHS of (\ref{qbinomial}), only the terms with
$c \mid i$ survive and therefore
\[
\cL_{b;j}\left(z^{-k}(zq^{-k};q)_{2k+1}\right)=
(-1)^{k} \sum_{n=-\lfloor k/c \rfloor}^{\lfloor (k+1)/c \rfloor} (-1)^{cn}
\qbinom{2k+1}{k+cn} z_{b,cn;j}.
\]
Separating the case $n=0$ and combining positive and negative $n$, this is equal to
\[
(-1)^{k}\qbinom{2k+1}{k}
+(-1)^k \sum_{n=1}^{\lfloor (k+1)/c \rfloor} (-1)^{cn}\left(\qbinom{2k+1}{k+cn}+\qbinom{2k+1}{k-cn}\right) z_{b,cn;j}
\]
where we use the convention that $\qbinom{x}{-1}$ is zero for positive $x$.
Further,
\[
\qbinom{2k+1}{k+cn} +\qbinom{2k+1}{k-cn} =\frac{\{k+1\}}{\{2k+2\}}\qbinom{2k+2}{k+cn+1}(q^{cn/2}+q^{-cn/2})
\]
and
\[
\frac{\{k+1\}}{\{2k+2\}}\qbinom{2k+2}{k+cn+1}\qbinom{2k+1}{k}^{-1}=
(-1)^{cn}q^{(k+1)cn+\frac{cn}{2}}\frac{(q^{-k-1};q)_{cn}}{(q^{k+2};q)_{cn}}
.
\]
Using $z_{b,cn;j}=(z_{b,c;j})^{n^2}=x_{b;j}^{n^2}$, we get the result.
\end{proof}

To define $Q_k(q,x_b,j)$, we will need Andrew's identity (3.43) of \cite{And}:
\begin{eqnarray*}
&&\hspace{-6mm}\sum_{n\geq 0} (-1)^n\alpha_n t^{-\frac{n(n-1)}{2}+sn+Nn} \frac{(t^{-N})_n}{(t^{N+1})_n}
\prod_{i=1}^{s}\frac{(b_i)_n (c_i)_n}{b_i^nc_i^n(\frac{t}{b_i})_n(\frac{t}{c_i})_n}
=\\
&&
\hspace{-4mm}
\frac{(t)_N(\frac{q}{b_sc_s})_N}{(\frac{t}{b_s})_N(\frac{t}{c_s})_N}
\sum_{n_s\geq \cdots\geq n_2\geq n_1 \geq 0}
\beta_{n_1}
\frac{t^{n_s}(t^{-N})_{n_s}(b_s)_{n_s}(c_s)_{n_s}}{(t^{-N} b_s c_s)_{n_s}}
\prod_{i=1}^{s-1}
\frac{t^{n_i}}{b_i^{n_{i}} c_i^{n_{i}}}
\frac{(b_i)_{n_{i}}(c_i)_{n_{i}}}{(\frac{t}{b_i})_{n_{i+1}}(\frac{t}{c_i})_{n_{i+1}}}
\frac{(\frac{t}{b_i c_i})_{n_{i+1}-n_i}}{(t)_{n_{i+1}-n_i}}\;.
\end{eqnarray*}
Here and in what follows we use the notation $(a)_n:=(a;t)_n$ .
The special Bailey pair $(\alpha_n,\beta_n)$ is chosen as follows
\[
\begin{array}{rclcrcl}
\alpha_0&=&1, &&\alpha_n&=&(-1)^nt^{\frac{n(n-1)}{2}}(1+t^n)\\
\beta_0&=&1, &&\beta_n&=&0 \;\;\;\text{ for } n\geq 1.
\end{array}
\]

\begin{lemma}\label{LHS}
$S_{b;j}(k,q)$ is equal to the LHS of Andrew's identity
with the parameters fixed below.
\end{lemma}

\begin{proof}
Since
\[
 S_{b;j}(k,q)=S_{-b;j}(k,q^{-1}),
\]
it is enough to look at the case $b>0$. Define $b':=\frac{b}{c}$ and
let $\omega$ be a $b'$th primitive root of unity. For simplicity,
put $N:=k+1$ and $t:=x_{b;j}$.
Using the following identities
\begin{eqnarray*}
(q^{y};q)_{cn}&=&\prod_{l=0}^{c-1}(q^{y+l};q^c)_n\;,\\
(q^{yc};q^c)_n&=&\prod_{i=0}^{b'-1}(\omega^i t^y;t)_n\;,
\end{eqnarray*}
where the later is true due to $t^{b'}=x_{b;j}^{b'}=q^c$ for all $j$, and
choosing a $c$th root of $t$ denoted by $t^{\frac{1}{c}}$,
we can see that
\[
S_{b;j}(k,q)=1+\sum_{n=1}^{\infty}
\prod_{i=0}^{b'-1}\prod_{l=0}^{c-1}
\frac{
(\omega^i t^{\frac{-N+l}{c}})_{n}
}{
(\omega^i t^{\frac{N+1+l}{c}})_{n}
}
(1+t^{b'n})t^{n^2+b'Nn}.
\]

Now we choose the parameters for Andrew's identity as follows.
We put
 $a:=\frac{c-1}{2}$, $d:=\frac{b'-1}{2}$ and $m:=\lfloor \frac{N}{c}\rfloor$.
For $l\in\{1,\ldots,c-1\}$, there exist unique $u_l,v_l\in\{0,\ldots,c-1\}$, such that
$u_l\equiv N+l\pmod{c}$ and $v_l\equiv N-l\pmod{c}$. Note that $v_l=u_{c-l}$. We define
$U_l:=\frac{-N+u_l}{c}$ and $V_l:=\frac{-N+v_l}{c}$. Then
$U_l, V_l \in \frac{1}{c}\Z$ but $U_l+V_l \in \Z$.
We define
\[
\begin{array}{rclrcll}
b_l&:=&t^{U_l}, &c_l&:=&t^{V_l} & \text{for }\; l=1,\ldots, a, \\
b_{a+i}&:=&\omega^it^{-m},
&c_{a+i}&:=&\omega^{-i}t^{-m}
& \text{for }\;i=1,\ldots,d, \\
b_{a+ld+i}&:=&\omega^it^{U_l},
&c_{a+ld+i}&:=&\omega^{-i}t^{V_l}
& \text{for }\;i=1,\ldots,d \text{ and } l=1,\ldots, c-1, \\
b_{g+i}&:=&-\omega^{i}t, &c_{g+i}&:=&-\omega^{-i}t &\text{for }\;i=1,\ldots,d,\\
b_{s-1}&:=&t^{-m}, &c_{s-1}&:=&t^{N+1}, &\\
b_s&\rightarrow&\infty,&c_s&\rightarrow&\infty,&\\
\end{array}
\]
where $g=a+cd$ and $s=(c+1)\frac{b'}{2}+1$.

We now calculate the LHS of Andrew's identity.
Using the notation
\[
(\omega^{\pm 1} t^x)_n=(\omega t^x)_n(\omega^{-1} t^x)_n
\]
and the identities
\[
\lim_{c\to \infty} \frac{(c)_n}{c^n}=(-1)^n t^{\frac{n(n-1)}{2}} \;\;\text{ and }\;\; \lim_{c\to\infty} \left(\frac{t}{c}\right)_n=1,
\]
we get
\begin{eqnarray*}
LHS&=&1+
\sum_{n\geq 1}
t^{n(n-1+s+N-y)}\;
(1+t^n)\frac{(t^{-N})_n}{(t^{N+1})_n}
\cdot
\prod_{l=1}^{a}
	\frac{
		(t^{U_l})_n (t^{V_l})_n
	}{
		(t^{1-U_l})_n (t^{1-V_l})_n
	}	
\cdot
\prod_{i=1}^{d}
	\frac{
		(\omega^{\pm i} t^{-m})_n
	}{
		(\omega^{\pm i} t^{1+m})_n
	}
\\	&& \hspace{3cm}
\cdot\prod_{i=1}^{d}\prod_{l=1}^{c-1}
	\frac{
		(\omega^i t^{U_l})_n
		(\omega^{-i}t^{V_l})_n
	}{
	(\omega^{-i} t^{1-U_l})_n
	(\omega^{i}t^{1-V_l})_n	
	}
\cdot
\prod_{i=1}^{d}
	\frac{
		(-\omega^{\pm i}t)_n
	}{
		(-\omega^{\pm i})_n
	}
\cdot
\frac{
	(t^{-m})_n (t^{N+1})_n
}{
	(t^{1+m})_n (t^{-N})_n
}
\end{eqnarray*}
where
 \[
 y:=
\sum_{l=1}^{a} (U_l+V_l)+\sum_{i=1}^{d}\sum_{l=1}^{c-1}(U_l+V_l)-m(2d+1)+2d+1+N.
\]
Since
$\sum_{l=1}^{c-1}(U_l+V_l)=2\sum_{l=1}^{a}(U_l+V_l)=2(-N+m+\frac{c-1}{2})$ and
$2d+1=b'$, we have
\[
n-1+s+N-y=n+Nb'.
\]
Further,
\[
\prod_{i=1}^{d}	\frac{(-\omega^{\pm i}t)_n}{(-\omega^{\pm i})_n}
=
\prod_{i=1}^{b'-1}\frac{1+\omega^i t^n}{1+\omega^i}=\frac{1+t^{b'n}}{1+t^n}	
\]
and
\begin{eqnarray*}
&&
\prod_{l=1}^{a}
	\frac{
		(t^{U_l})_n (t^{V_l})_n
	}{
		(t^{1-U_l})_n (t^{1-V_l})_n
	}	
\cdot
\prod_{i=1}^{d}
	\frac{
		(\omega^{\pm i} t^{-m})_n
	}{
		(\omega^{\pm i} t^{1+m})_n
	}
\cdot
\prod_{i=1}^{d}\prod_{l=1}^{c-1}
	\frac{
		(\omega^i t^{U_l})_n
		(\omega^{-i}t^{V_l})_n
	}{
	(\omega^{-i} t^{1-U_l})_n
	(\omega^{i}t^{1-V_l})_n	
	}
\cdot
\frac{
	(t^{-m})_n
}{
	(t^{1+m})_n
}
\\
&&\hspace{105mm}=
\prod_{i=0}^{b'-1}\prod_{l=0}^{c-1}
\frac{
(\omega^i t^{\frac{-N+l}{c}})_{n}
}{
(\omega^i t^{\frac{N+1+l}{c}})_{n}
}.
\end{eqnarray*}
Taking all the results together, we see that the LHS is equal to
$S_{b;j}(k,q)$.

\end{proof}

Let us now calculate the RHS of Andrew's identity with parameters chosen as above.
For simplicity, we put $\delta_j:=n_{j+1}-n_j$.
Then the RHS is given by
\begin{eqnarray*}
RHS
&=&
(t)_N\sum_{n_s\geq \cdots\geq n_2 \geq n_1=0}
\frac{
t^{x}\cdot (t^{-N})_{n_s}(b_s)_{n_s}(c_s)_{n_s}
}{
\prod_{i=1}^{s-1} (t)_{\delta_{i}}(t^{-N}b_sc_s)_{n_s}
}
\cdot
\frac{
(t^{-m})_{n_{s-1}}(t^{N+1})_{n_{s-1}}
(t^{m-N})_{\delta_{s-1}}
}{
(t^{m+1})_{n_s}(t^{-N})_{n_s}
}
\\
&&\hspace{7mm}
\cdot
\prod_{l=1}^{a}
\frac{
(t^{U_l})_{n_l}(t^{V_l})_{n_l}
(t^{1-U_l-V_l})_{\delta_{l}}
}{
(t^{1-U_l})_{n_{l+1}}(t^{1-V_l})_{n_{l+1}}
}
\cdot
\prod_{i=1}^{d}
\frac{
(\omega^{\pm i}t^{-m})_{n_{a+i}}
(t^{2m+1})_{\delta_{a+i}}
}{
(\omega^{\pm i}t^{m+1})_{n_{a+i+1}}
}
\frac{
(-\omega^{\pm i}t)_{n_{g+i}}
(t^{-1})_{\delta_{g+i}}
}{
(-\omega^{\pm i})_{n_{g+i+1}}
}
\\
&&\hspace{7mm}
\cdot
\prod_{i=1}^{d}\prod_{l=1}^{c-1}
\frac{
(\omega^it^{U_l})_{n_{a+ld+i}}
(\omega^{-i}t^{V_l})_{n_{a+ld+i}}
(t^{1-U_l-V_l})_{\delta_{a+ld+i}}
}{
(\omega^{-i}t^{1-U_l})_{n_{a+ld+i+1}}
(\omega^{i}t^{1-V_l})_{n_{a+ld+i+1}}
}
\end{eqnarray*}
where
\begin{eqnarray*}
x
&=&\sum_{l=1}^{a}(1-U_l-V_l)\,n_l
+\sum_{i=1}^{d}(2m+1)\,n_{a+i}
\\&&\hspace{2cm}
+\sum_{i=1}^{d}\sum_{l=1}^{c-1}(1-U_l-V_l)\, n_{a+ld+i}
-\sum_{i=1}^{d}n_{g+i}+(m-N)\,n_{s-1}+n_s.
\end{eqnarray*}
For $c=1$ or $d=0$,
we use the convention that empty products  are equal to 1 and empty sums are equal to zero.

Let us now have a closer look at the RHS.
Notice that
\[
\lim_{b_s,c_s\to\infty}\frac{
(b_s)_{n_s}(c_s)_{n_s}
}{
(t^{-N}b_sc_s)_{n_s}
}
=
(-1)^{n_s}t^{\frac{n_s(n_s-1)}{2}}t^{Nn_s}.
\]
The term $(t^{-1})_{\delta_{g+i}}$ is zero unless
 $\delta_{g+i}\in \{0,1\}$.
Therefore, we get
\[
\prod_{i=1}^{d}\frac{
(-\omega^{\pm i}t)_{n_{g+i}}
}{
(-\omega^{\pm i})_{n_{g+i+1}}
}
=
\prod_{i=1}^{d}
(1+\omega^{\pm i}t^{n_{g+i}})^{1-\delta_{g+i}}.
\]
Due to the term $(t^{-m})_{n_s}$, we have
$n_s\leq m$ and therefore $n_i\leq m$ for all $i$.
Multiplying the numerator and denominator of each term of the RHS by
\begin{eqnarray*}
&&
\prod_{l=1}^{a} (t^{1-U_l+n_{l+1}})_{m-n_{l+1}}(t^{1-V_l+n_{l+1}})_{m-n_{l+1}}
\prod_{i=1}^{d} (\omega^{\pm i}t^{m+1+n_{a+i+1}})_{m-n_{a+i+1}}
\\
&&\hspace{35mm}\cdot\prod_{i=1}^{d}\prod_{l=1}^{c-1}
(\omega^{-i}t^{1-U_l+n_{a+ld+i+1}})_{m-n_{a+ld+i+1}}
(\omega^{i}t^{1-V_l+n_{a+ld+i+1}})_{m-n_{a+ld+i+1}}
\end{eqnarray*}
gives in the denominator
$\prod_{i=0}^{b'-1}\prod_{l=1}^{c-1}(\omega^it^{1-U_l})_m
\cdot
\prod_{i=1}^{b'-1}(\omega^it^{m+1})_m$.
This is equal to
\[
\prod_{l=1}^{c-1}(t^{b'(1-U_l)};t^{b'})_m
\cdot \frac{(t^{b'(m+1)};t^{b'})_m}{(t^{m+1};t)_m}
=
\frac{(q^{N+1};q)_{cm}}{(t^{m+1};t)_m}.
\]
Further,
\[
(t)_N(t^{N+1})_{n_{s-1}}=(t)_{N+n_{s-1}}=(t)_{m}(t^{m+1})_{N-m+n_{s-1}}.
\]
The term $(t^{-N+m})_{\delta_{s-1}}$ is zero unless $\delta_{s-1}\leq N-m$
and therefore
\[
\frac{(t^{m+1})_{N-m+n_{s-1}}}{(t^{m+1})_{n_s}}
=
(t^{m+1+n_s})_{N-m-\delta_{s-1}}.
\]
Using the above calculations, we get
\begin{equation}\label{RHS}
RHS=
\frac{(t;t)_{2m}}{(q^{N+1};q)_{cm}}
\cdot T_k(q,t)
\end{equation}
where
\begin{eqnarray*}
T_k(q,t)&:=&
\sum_{n_s\geq \cdots\geq n_2 \geq n_1=0}
(-1)^{n_s}t^{x'}\cdot
(t^{-m})_{n_{s-1}}\cdot
(t^{m+1+n_s})_{N-m-\delta_{s-1}}
\cdot\frac{
(t^{-N+m})_{\delta_{s-1}}
}{
\prod_{i=1}^{s-1} (t)_{\delta_i}
}
\\
&&\hspace{3mm}
\cdot
\prod_{l=1}^{a}(t^{1-U_l-V_l})_{\delta_l}
\cdot
\prod_{i=1}^{d} (t^{2m+1})_{\delta_{a+i}}(t^{-1})_{\delta_{g+i}}
\cdot
\prod_{i=1}^{d}\prod_{l=1}^{c-1}(t^{1-U_l-V_l})_{\delta_{a+ld+i}}
\\
&&\hspace{3mm}
\cdot\prod_{l=1}^{a}
(t^{U_l})_{n_l}(t^{V_l})_{n_l}
(t^{1-U_l+n_{l+1}})_{m-n_{l+1}}
(t^{1-V_l+n_{l+1}})_{m-n_{l+1}}
\cdot
\prod_{i=1}^{d}
(1+\omega^{\pm i}t^{n_{g+i}})^{1-\delta_{g+i}}
\\
&&\hspace{3mm}
\cdot
\prod_{i=1}^{d}
(\omega^{\pm i}t^{-m})_{n_{a+i}}
(\omega^{\pm i}t^{m+1+n_{a+i+1}})_{m-n_{a+i+1}}
\cdot
\prod_{i=1}^{d}
\prod_{l=1}^{c-1}
(\omega^{i}t^{U_l})_{n_{a+ld+i}}
(\omega^{-i}t^{V_l})_{n_{a+ld+i}}
\\
&&\hspace{3mm}
\cdot
\prod_{i=1}^{d}
\prod_{l=1}^{c-1}
(\omega^{-i}t^{1-U_l+n_{a+ld+i+1}})_{m-n_{a+ld+i+1}}
(\omega^{i}t^{1-V_l+n_{a+ld+i+1}})_{m-n_{a+ld+i+1}}
\end{eqnarray*}
and
$x':=x+\frac{n_s(n_s-1)}{2}+Nn_s$.

We define the element $Q_{k}(q,x_b,j)$   by
\[
Q_{k}(q,x_b,j):=
\left((-1)^{k+1}q^{-\frac{k(k+1)}{2}}\right)
^{\frac{1+\sn(b)}{2}}
\left(q^{(k+1)^2}\right)^{\frac{1-\sn(b)}{2}}
\frac{(x_{b;j};x_{b;j})_{2m}}{(q;q)_{N+cm}}
\;T_k(q,x_{b;j}).
\]

By Lemmas \ref{S_{b;j}(k,q)} and \ref{LHS}, Equation (\ref{RHS}) and the following Lemma \ref{bPosbNeg},
we  see that this element
satisfies Equation (\ref{imp}).

\begin{lemma}\label{bPosbNeg}
The following formula holds.
\[
(-1)^{k+1}\qbinom{2k+1}{k}(q^{k+1};q)_{k+1}^{-1}=
(-1)^{k+1}\frac{q^{-k(k+1)/2}}{(q;q)_{k+1}}
=
\frac{q^{-(k+1)^2}}{(q^{-1};q^{-1})_{k+1}}
\]
\end{lemma}

\begin{proof}
This is an easy calculation using
\[
(q^{k+1};q)_{k+1}=(-1)^{k+1} q^{(3k^2+5k+2)/4}\frac{\{2k+1\}!}{\{k\}!}.
\]
\end{proof}

\subsection{Part 1: Existence of $Q_k(q,x_b,j)$, $b$ even case.}
Let  $b=\pm 2^l$. We have to prove Equality (\ref{imp}) only for $j=0$, i.e.
we have to show
\[
\frac{1}{(q^{k+1};q)_{k+1}}
\,\cL_{b;0}\left( \prod_{i=0}^k (z+z^{-1} - q^i -q^{-i})   \right)
=
2\, Q_k(q^{\sn(b)}, x_{b},0).
\]
The calculation works similarly to the odd case. Note that we have $c=1$ here.
This case was already done in \cite{BL} and \cite{Le}. Since their approaches are slightly different and for the sake of completeness, we will give the parameters for Andrew's identity and the formula for $Q_{k}(q,x_{b},0)$ nevertheless.

We put $t:=x_{b;0}$, $d:=\frac{b}{2}-1$, $\omega$ a $b$th root of unity and choose a primitive square root $\nu$ of $\omega$. Define the parameters of Andrew's identity by
\[
\begin{array}{rclrcll}
b_{i}&:=&\omega^it^{-N},
&c_{i}&:=&\omega^{-i}t^{-N}
&  \;\text{for }\;i=1,\ldots,d, \\
b_{d+i}&:=&-\nu^{2i-1}t, &c_{d+i}&:=&-\nu^{-(2i-1)}t &\;\text{for }\;i=1,\ldots,d+1,\\
b_{b}&:=&-t^{-N},
&c_{b}&:=&-t^0=-1, &\\
b_{s-1}&:=&t^{-N}, &c_{s-1}&:=&t^{N+1}, &\\
b_s&\rightarrow&\infty,&c_s&\rightarrow&\infty,&
\end{array}
\]
where $s=b+2$.
Now we can define the element
\[
Q_{k}(q,x_b,0):=
\left((-1)^{k+1}q^{-\frac{k(k+1)}{2}}\right)
^{\frac{1+\sn(b)}{2}}
\left(q^{(k+1)^2}\right)^{\frac{1-\sn(b)}{2}}
\frac{(x_{b;0};x_{b;0})_{2N}}{(q;q)_{2N}}
\frac{1}{(-x_{b;0};x_{b;0})_N} T_k(q,x_{b;0})
\]
where
\begin{eqnarray*}
T_k(q,t)
&:=&
\sum_{n_{s-1}\geq \cdots \geq n_1=0}
(-1)^{n_{s-1}}t^{x''}
\cdot
\frac{
\prod_{i=1}^{d}
(t^{2N+1})_{\delta_{i}}
\cdot\prod_{i=1}^{d+1}(t^{-1})_{\delta_{d+i}}
\cdot
(t^{N+1})_{\delta_b}
}{
\prod_{i=1}^{s-2}(t)_{\delta_i}
}
\\
&&\hspace{13mm}
\cdot
(t^{-N})_{n_{s-1}}
\cdot (-t^{N+1+n_{s-1}})_{N-n_{s-1}}
\cdot(-t^{-N})_{n_{b}}
\cdot (-t)_{n_{b}-1}
\cdot (-t^{n_{s-1}+1})_{N-n_{s-1}}
\\
&&\hspace{13mm}
\cdot
\prod_{i=1}^{d}
(\omega^{\pm i}t^{-N})_{n_{i}}
(\omega^{\pm i}t^{N+1+n_{i+1}})_{N-n_{i+1}}
\cdot
\prod_{i=1}^{d+1}(1+\nu^{\pm(2i-1)}t^{n_{d+i}})^{1-\delta_{d+i}}
\end{eqnarray*}
and
$
x'':=\sum_{i=1}^{d}(2N+1)n_i-\sum_{i=1}^{d+1}n_{d+i}+\frac{n_{s-1}(n_{s-1}-1)}{2}+(N+1)(n_b+n_{s-1})$.
We use the notation $(a;b)_{-1}:=\frac{1}{1-ab^{-1}}$.

\subsection{Part 2: $Q_{k}(q,x_b,j) \in\cS_{p,j}$\;.}
We have to show  that $Q_{k}(q,x_b,j) \in\cS_{p,j}$,
where $j \in \N\cup\{0\}$ if $p$ is odd, and $j=0$ for $p=2$.
This follows from the next two lemmas.

\begin{lemma}
For $t=x_{b;j}$,
\[
T_k(q,t)\in \Z[q^{\pm 1}, t^{\pm1}].
\]
\end{lemma}

\begin{proof}
Let us first look at the case $b$ odd and positive. Since for $a\not=0$, $(t^a)_n$ is always divisible by $(t)_n$, it is easy to see that the denominator of each term of $T_k(q,t)$ divides its numerator. Therefore we proved that $T_k(q,t)\in\Z[t^{\pm 1/c},\omega]$. Since
\be\label{Sbj}
S_{b;j}(k,q)=
\frac{(t;t)_{2m}}{(q^{N+1};q)_{cm}}
\cdot T_k(q,t),
\ee
there are $f_0,g_0\in\Z[q^{\pm 1},t^{\pm 1}]$ such that $T_k(q,t)= \frac{f_0}{g_0}$. This implies that $T_k(q,t)\in \Z[q^{\pm 1}, t^{\pm 1}]$ since $f_0$ and $g_0$ do not depend on $\omega$ and the $c$th root of $t$.

The proofs for the even and the negative case work analogously.
\end{proof}

\begin{lemma}
For $t=x_{b;j}$,
\[
\frac{(t;t)_{2m}}{(q;q)_{N+cm}}
\frac{1}{((-t;t)_{N})^{\lambda}}
\in\cS_{p,j}
\]
where $\lambda=1$ and $j=0$ if $p=2$,  and $\lambda=0$ and
$j\in \N\cup \{0\}$ otherwise.
\end{lemma}

\begin{proof}
Notice that
\[
(q;q)_{N+cm}=\widetilde{(q;q)}_{N+cm}(q^c;q^c)_{2m}
\]
where we use the notation
\[
\widetilde{(q^a;q)}_{n}:=
\prod_{\substack{j=0\\c\nmid (a+j)}}^{n-1}(1-q^{a+j}).
\]
We have to show
that
\[
\frac{(q^c;q^c)_{2m}}{(t;t)_{2m}}
\cdot
\widetilde{(q;q)}_{N+cm}\cdot
((-t;t)_{N})^{\lambda}
\]
is invertible in $\Z[1/p][q]$ modulo any ideal  $(f)=(\prod_{n} \Phi^{k_n}_n (q))$ where $n$ runs through a subset of $p^j\N_p$. Recall that in a commutative ring $A$, an element $a$ is invertible in $A/(d)$ if and only if $(a)+(d)=(1)$. If $(a)+(d)=(1)$ and $(a)+(e)=(1)$, we get $(a)+(de)=(1)$. Hence, it is enough to consider  $f=\Phi_{p^j n}(q)$ with $(n,p)=1$. For any $X\in\N$, we have
\begin{eqnarray}\label{denomin1}
{\widetilde{(q;q)}_X}&=&\prod_{\text{\scriptsize{$\begin{array}{c}i=1 \\ c\nmid i \end{array}$}}}^{X} \prod_{d\mid i} \Phi_d(q),
\\\label{denomin2}
(-t;t)_X&=&\frac{(t^2;t^2)_X}{(t;t)_X}=\prod_{i=1}^{X}\prod_{d\mid i}\Phi_{2d}(t),
\\\label{denomin3}
\frac{(q^c;q^c)_X}{(t;t)_X}
&=&\frac{(t^{b'};t^{b'})_X}{(t;t)_X}=
\frac{
\prod_{i=1}^{X}\prod_{d\mid ib'}\Phi_d(t)
}{
\prod_{i=1}^{X}\prod_{d\mid i}\Phi_d(t),
}
\end{eqnarray}
for $b'=b/c$. Recall that $(\Phi_r(q),\Phi_a(q))=(1)$ in $\Z[1/p][q]$ if either $r/a$ is not a power of a prime or a power of $p$. For $r=p^j n$ odd and $a$ such that $c\nmid a$, one of the conditions is always satisfied. Hence  \eqref{denomin1} is invertible in $\cS_{p,j}$. If $b=c$ or $b'=1$, \eqref{denomin2} and \eqref{denomin3} do not contribute. For $c<b$, notice that $q$ is a $cn$th primitive root of unity in $\Z[1/p][q]/(\Phi_{cn}(q)) = \Z[1/p][e_{cn}]$. Therefore $t^{b'}=q^{c}$ is an $n$th primitive root of unity. Since $(n,b')=1$, $t$ must be a primitive $n$th root of unity in $\Z[1/p][e_{cn}]$ too, and hence $\Phi_n(t) = 0$ in that ring. Since for $j$ with $(j, p) > 1$, $(\Phi_j(t),\Phi_n(t)) = (1)$ in $\Z[1/p][t]$, $\Phi_j(t)$ is invertible in $\Z[1/p][e_{cn}]$, and therefore \eqref{denomin2} and \eqref{denomin3} are invertible too.
\end{proof}

\newpage

\pagestyle{fancy}

\setlength{\headheight}{15.4pt}
\renewcommand{\chaptermark}[1]{\markboth{\MakeUppercase{Appendix}\ \thechapter. \ #1}{}}
\renewcommand{\sectionmark}[1]{\markright{\thesection. \ #1}{}}
\fancyhead{}
\fancyhead[RE]{\slshape \leftmark}
\fancyhead[LE]{\thepage}
\fancyhead[LO]{\slshape \rightmark}
\fancyhead[RO]{\thepage}
\fancyfoot{}
\renewcommand{\headrulewidth}{0pt} 
\renewcommand{\footrulewidth}{0pt}

\appendix

\numberwithin{equation}{section}
\numberwithin{figure}{section}

\chapter{Proof of Theorem \ref{GeneralizedHabiro}}
The appendix is devoted to the proof of Theorem \ref{GeneralizedHabiro}, a generalization of the  deep integrality result of Habiro, namely Theorem 8.2 of \cite{Ha}. The existence of this generalization and some ideas of the proof were kindly communicated to us by Habiro.

\section{Reduction to a result on values of the colored Jones polynomial}
We will use the notation of Section \ref{Uh(sl2)}.

Let $V_n$ be the unique $(n+1)$--dimensional irreducible $U_h$--module. In \cite{Ha}, Habiro defined a new basis $\tilde{P}'_k$, $k=0,1,2,\dots$,  for the Grothendieck ring of finite--dimensional $U_h(sl_2)$--modules with
\[
\tilde{P}_k':=\frac{v^{\frac{1}{2}k(1-k)}}{\{k\}!} \, \prod_{i=0}^{k-1}
(V_1-v^{2i+1}-v^{-2i-1}).
\]

Put $\tilde{P}'_{\bk}=\{\tilde{P}'_{k_1},\ldots,\tilde{P}'_{k_m}\}$. It follows from Lemma 6.1 of \cite{Ha}  that we will have identity \eqref{Jones} of Theorem \ref{GeneralizedHabiro} if we substitute \[
C_{L\sqcup L'}(\bk,\bj)=
J_{L\sqcup L'}\,(\tilde{P}'_{\bk},\bj)\,
\prod_{i}(-1)^{k_i}q^{k^2_i+k_i+1}\, .
\]

Hence, to prove Theorem \ref{GeneralizedHabiro}, it is enough to show the following.
\begin{Atheorem}\label{main-integrality}
Suppose $L\sqcup L'$ is a colored framed link in $S^3$ such that $L$ has 0 linking matrix and $L'$ has odd colors. Then  for $k=\max\{k_1,\dots, k_m\}$, we have
\[
J_{L\sqcup L'}(\tilde{P}'_{\bk},\bj) \in \frac{(q^{k+1};q)_{k+1}}{1-q}
\,\,\Z[q^{\pm 1}].
\]
\end{Atheorem}

In the case $L'=\emptyset$, this statement was proven
in \cite[Theorem 8.2]{Ha}. Since our proof is a modification
of the original one, we first sketch Habiro's original proof for the reader's convenience.

\section{Sketch of  the proof of Habiro's  integrality theorem}

Corollary 9.13 in \cite{Ha-b} states the following. 
\begin{Aproposition}\label{thmASL}{\em (Habiro)}
If the linking matrix of a bottom tangle $T$ is zero then $T$ can be presented as $T=W B^{\otimes k}$, where $k\geq 0$ and $W\in \modB(3k,n)$ is obtained by horizontal and vertical pasting of finitely many copies of $1_\modb$, $\psi_{\modb,\modb}$, $\psi^{-1}_{\modb,\modb}$, and 
\begin{equation*}
    \def\sss{1.5em}e
    \eta _\modb =\incl{\sss}{bot0},\quad
    \def\sss{2em}
    \mu _\modb =\incl{\sss}{bomu},\quad
    \gamma_+=\incl{2.5em}{gamma+},\quad
    \gamma_-=\incl{2.5em}{gamma-}.
\end{equation*}
\end{Aproposition}

Let  $K=v^{H}=e^\frac{hH}{2}$. Habiro introduced the integral version $\U_q$, which is the $\Z[q,q^{-1}]$--subalgebra of $U_h$ freely spanned by $\tilde{F}^{(i)}K^{j}e^k$ for $i,k\geq 0, j\in\Z$, where
\[
\tilde{F}^{(n)}=\frac{F^nK^n}{v^{\frac{n(n-1)}{2}}[n]!}\;\;\;\;\text{ and
 }\;\;\;\; e=(v-v^{-1})E.
\]
There is a $\Z/2\Z$--grading, $\U_q=\U^0_q\oplus \U^1_q$, where $\U^{0}_q $ (respectively $\U_q^{1}$) is spanned
by $\tilde{F}^{(i)}K^{2j}e^k$ (respectively $\tilde{F}^{(i)}K^{2j+1}e^k$). We call this the $\ve$--grading and $\U^{0}_q $ (respectively $\U_q^{1}$) the even (respectively odd) part.
 
The two--sided ideal $\F_p$ in $\U_q$ generated by $e^p$ induces a filtration on $(\U_q)^{\otimes n}$, $n\geq 1$, by
\[
\F_p((\U_q^{})^{\otimes n})=
\sum^n_{i=1} (\U_q^{})^{\otimes i-1}\otimes
\F_p(\U_q^{})\otimes (\U_q^{})^{\otimes n-i}
\subset (\U_q^{})^{\otimes n}\, .
\]
Let $(\tilde{\U}_q)^{\tilde\otimes n}$ be the image of the homomorphism
\[
\lim_{\overleftarrow{\hspace{1mm}{p\geq 0}\hspace{1mm}}} \;\;
\frac{(\U_q)^{\otimes n}}{\F_p((\U_q)^{\otimes n})} \to
U_h^{\hat \otimes n}
\]
where $\hat\otimes$ is the $h$--adically completed tensor product. By using $\F^\ve_p(\U_q^\ve):=\F_p(\U_q)\cap \U^{\ve}_q$ one defines $(\tilde{\U}^{\ve}_q)^{\tilde\otimes n}$ for $\ve\in\{0,1\}$ in a similar fashion. 

By definition (Section 4.2 of \cite{Ha}), the universal $sl_2$ invariant  $J_T$ of an $n$--component  bottom tangle $T$ is an element of $U_h^{\hat\otimes n}$.
Theorem 4.1 in \cite{Ha} states  that, in fact, for any bottom tangle $T$ with zero linking matrix, $J_T$ is even, i.e. 
\be\label{even}
J_T\in (\tilde \U^{0}_q)^{\tilde \otimes n}\, . 
\ee

Further, using the fact that $J_K$ of a $0$--framed bottom knot $K$ (i.e. a 1--component bottom tangle) belongs to the center of $\tilde \U^0_q$, Habiro showed that
\[
J_K=\sum_{n\geq 0}  (-1)^n q^{n(n+1)}\frac{(1-q)}{(q^{n+1};q)_{n+1}}\,  J_K(\tilde P'_n)\, \sigma_n
\]
where
\[
\sigma_n=\prod^n_{i=0}(C^2-(q^i+2+ q^{-i}))\, \quad {\rm with}\quad
C=(v-v^{-1})\tF^{(1)}K^{-1}e+vK+v^{-1}K^{-1}
\]
is the quantum Casimir operator. The $\sigma_n$ provide a basis for the even part of the center. From this, Habiro deduced that $J_K(\tilde P'_n)\in \frac{  (q^{n+1};q)_{n+1}}{(1-q)} \Z[q,q^{-1}]$. 

The case of $n$--component bottom tangles reduces to the 1--component case by partial trace, using certain integrality of traces of even element (Lemma 8.5 of \cite{Ha}) and the fact that $J_T$ is invariant under the adjoint action. 

The proof of \eqref{even} uses Proposition \ref{thmASL}, which allows us to build any bottom tangle $T$ with zero linking matrix from simple parts, i.e. $T=W(B^{\otimes k})$.

On the other hand, the construction of the universal invariant $J_T$ extends to the braided functor $J:\modB\to \operatorname{ Mod}_{U_h}$ from $\modB$ to the category of $U_h$--modules. 
This means that $J_{W(B^{\otimes k})}=J_W(J_{B^{\otimes k}})$. Therefore, in order to show \eqref{even}, we need to check that $J_B \in (\tilde \U^0_q)^{\tilde \otimes 3} $ and then verify that
$J_W$ maps the even part to itself. The first check can be done by a direct computation \cite[Section 4.3]{Ha}. The last verification is the content of Corollary 3.2 in \cite{Ha}.

\section{Strategy of the Proof of Theorem \ref{main-integrality}}

\subsection{Generalization of Equation  \eqref{even}}
To prove Theorem \ref{main-integrality}, we need a generalization of Equation \eqref{even} or Theorem 4.1 in \cite{Ha} to tangles with closed components. To state the result, we first introduce two gradings.

Suppose $T$ is an $n$--component bottom tangle in a cube, homeomorphic to the 3--ball  $D^3$. Let $\tcS(D^3\setminus T)$ be the $\Z[q^{\pm 1/4}]$--module freely generated by the isotopy classes of
framed unoriented colored links in $D^3\setminus T$, including the empty link. For such a link $L\subset D^3\setminus T$ with $m$--components colored by $n_1,\dots, n_m$, we define our new gradings as follows. First provide the components of $L$ with arbitrary orientations. Let $l_{ij}$ be the linking number between the $i$th component of $T$ and the $j$th component of $L$, and $p_{ij}$ be the linking number between the $i$th and the $j$th components of $L$. For $X=T\sqcup L$, we put \begin{equation}\label{3306}
\gr_\ve(X):= (\ve_1,\dots, \ve_n)\in (\Z/2\Z)^n \quad  \text{where}\quad
\ve_i := \sum_{j} l_{ij} n'_{j} \pmod 2, \quad \text{ and}
\end{equation}
\[
\gr_q(L) := \sum_{1\le i,j\le m }p_{ij} n_i'n_j' + 2\sum_{1\le j\le m}(p_{jj}+1) n_j' \pmod 4
\quad \text{where} \quad n_i':= n_i-1.
\]
It is easy to see that the definitions do not depend on the orientation of $L$.

The meaning of $\gr_q(L)$ is the following: The colored Jones polynomial of $L$,
a priori  a Laurent polynomial of $q^{1/4}$, is actually a Laurent polynomial of $q$ after dividing by $q^{\gr_q(L)/4}$; see \cite{LeDuke} for this result and its generalization to other Lie algebras. 

We further extend both gradings to $\tcS(D^3\setminus T)$ by
\[
\gr_\ve(q^{1/4})=0, \quad \gr_q(q^{1/4}) =1 \pmod 4.
\]

Recall that the universal invariant $J_X$ can also be defined when $X$ is the union of a bottom tangle and a colored link (see \cite[Section 7.3]{Ha-b}). In \cite{Ha-b}, it is proved that $J_X$ is adjoint invariant. The generalization of Theorem 4.1 of \cite{Ha} is the following.  
\begin{Atheorem}\label{MMM}
Suppose $X=T \sqcup L$ where $T$ is a $n$--component bottom tangle with zero linking matrix and $L$ is a framed unoriented colored link with $\gr_\ve(X)= (\ve_1,\dots,\ve_n)$. Then 
\[
J_X \in  q^{{\gr_q(L)/4}}\, \left ( \tilde \U_q^{\ve_1} \tilde \otimes
\dots \tilde \otimes \tilde \U_q^{\ve_n}\right ).
\]
\end{Atheorem}

\begin{Acorollary}\label{MMM1}
Suppose $L$ is colored by a tuple of odd numbers, then
\[
J_X\in (\tilde \U_q^{0})^{\tilde\otimes n}\, .
\]
\end{Acorollary}

Since $J_X$  is invariant under the adjoint action, Theorem \ref{main-integrality} follows from Corollary \ref{MMM1} by repeating Habiro's arguments. Hence it remains to prove Theorem \ref{MMM}. In the proof we will need the notion of a {\em good morphism}.

\subsection*{Good morphisms}
Let $I_m\in \modB(m,m)$ be the identity morphism of $\modb ^{\otimes m}$ in the cube $C$. A framed link $L$ in the complement $C \setminus I_m$ is {\em good} if $L$ is geometrically disjoint from all the up arrows of $\modb ^{\otimes m}$, i.e. there is a plane dividing the cube into two  halves, such that all the up arrows are in one half, and all the down arrows and $L$ are in the other. Equivalently, there is a diagram in which all the up arrows are above all components of $L$. The union $W$ of $I_m$ and a  colored framed good link $L$ is called a {\em good morphism}. If $Y$ is any bottom tangle so that we can compose $X =WY$, then it is easy to see that $\gr_\ve(X)$ does not depend on $Y$, and we define $\gr_\ve(W):=\gr_\ve(X)$. Also define $\gr_q(W):= \gr_q(L)$. 

As in the case with $L=\emptyset$, the universal invariant extends to a map $J_W: \U_h^{\otimes m} \to \U_h^{\otimes m}$.

\subsection{Proof of Theorem \ref{MMM}}
The strategy here is again analogous to the Habiro case. In Proposition \ref{TWWB} we will decompose $X$ into simple parts: the top is a bottom tangle with zero linking matrix, the next is a good morphism, and the bottom is a morphism obtained by pasting copies of $\mu_\modb$. Since any bottom tangle with zero linking matrix satisfies Theorem \ref{MMM} and $\mu_\modb$ is the product in $\U_q$, which preserves $\gr_\ve$ and $\gr_q$, it remains to show that any good morphism preserves the gradings. This is done in Proposition \ref{MM2} below.
\qed

\begin{Aproposition}\label{TWWB}
Assume $X=T\sqcup L$ where $T$ is a $n$--component bottom tangle with zero linking matrix and $L$ is a link. Then there is a presentation $X =W_2 W_1 W_0$ where $W_0$ is a bottom tangle with zero linking matrix, $W_1$ is a good morphism, and $W_2$ is obtained by pasting  copies of $\mu_\modb$.
\end{Aproposition}

\begin{proof}
Let us first define $\tilde \gamma_\pm\in \modB(i, i+1)$ for any $i\in {\mathbb N}$ as follows. \begin{equation*}
    \def\sss{3em}
  \tilde\gamma_+  :=\incl{\sss}{tildegamma+}\quad
    \def\sss{3em}
    \tilde\gamma_- :=\incl{\sss}{tildegamma-}.
\end{equation*}
If a copy of $\mu_\modb$ is directly above $\psi^{\pm 1}_{\modb, \modb}$ or $\gamma_\pm$, one can move $\mu_\modb$ down by isotopy and represent the result by pasting copies of $\psi^{\pm 1}_{\modb, \modb}$ and $\tilde \gamma_\pm$. It is easy to see that after the isotopy, $\gamma_\pm$ gets replaced by $\tilde\gamma_\pm$ and $\psi^{\pm 1}_{\modb, \modb}$ by two copies of $\psi^{\pm 1}_{\modb, \modb}$.

Using Proposition \ref{thmASL} and reordering the basic morphisms so that the $\mu$'s are at the bottom, one can see that $T$ admits the following presentation:
\[
T= W_2 \tilde W_1 (B^{\otimes k})
\]
where $B$ is the Borromean tangle, $W_2$ is obtained by pasting copies of $\mu_\modb$ and $\tilde{W}_1$ is obtained by pasting copies of $\psi^{\pm 1}_{\modb,\modb}$, $\tilde\gamma_{\pm}$ and $\eta_\modb$.

Let $P$ be the horizontal plane separating $\tilde W_1$ from $W_2$. Let $P_+$ ($P_-$) be the upper (respectively lower) half--space. Note that $W_0=\tilde W_1(B^{\otimes k})$ is a bottom tangle with zero linking matrix lying in $P_+$  and does not have any minimum points. Therefore the pair $(P_+,W_0)$ is homeomorphic to the pair $(P_+,\; l \text{ trivial arcs})$. Similarly, $W_2$ does not have any maximum points; hence $L$ can be isotoped  off $P_-$ into $P_+$. Since the pair $(P_+,W_0)$ is homeomorphic to the pair $(P_+,\; l\text{ trivial arcs})$ one can isotope $L$ in $P_+$ to the bottom end points of down arrows. We obtain the desired presentation. 
\end{proof}

\begin{Aproposition}\label{MM2} For every good morphism $W$, the operator $J_{W}$ preserves $\gr_\ve$ and $\gr_q$ in the following sense. If $x\in \U_q^{\ve_1} \otimes \dots \otimes \U_q^{\ve_m}$, then 
\[ 
J_{W}(x) \in q^{\gr_q(W)/4} \left( \U_q^{\ve_1'} \otimes \dots \otimes
 \U_q^{\ve_m'}\right) \quad \text{where} \quad (\ve'_1,\dots, \ve'_m) =
 (\ve_1,\dots, \ve_m) + \gr_{\ve}(W).
\]
\end{Aproposition}
The rest of the appendix is devoted to the proof of Proposition \ref{MM2}.

\subsection{Proof of Proposition \ref{MM2}}\label{ProofMM2}
We proceed as follows. Since $J_X$ is invariant under cabling and skein relations, and by Lemma \ref{A8} below, both relations preserve our gradings, we consider the quotient of $\tilde S(D^3\setminus T)$ by these relations. It is known as a skein module of $D^3\setminus T$. For  $T=I_n$, this module has a natural algebra structure with good morphisms forming a subalgebra. By Lemma \ref{A9} 
(see also Figure \ref{Wgamma}), the basis elements $W_\gamma$ of this subalgebra are labeled by $n$--tuples $\gamma=(\gamma_1,\dots, \gamma_n)\in (\Z/2\Z)^n$. It is clear that if the proposition holds for $W_{\gamma_1}$ and $W_{\gamma_2}$, then it holds for $W_{\gamma_1} W_{\gamma_2}$. Hence 
it remains to check the claim for $W_\gamma$'s. This is done in Corollary \ref{corA10} for basic good morphisms corresponding to the $\gamma$ whose non--zero $\gamma_j$'s are consecutive. Finally, any $W_\gamma$ can be obtained by pasting a basic good morphism  with a few copies of $\psi^{\pm}_{\modb,\modb}$. Since $J_{\psi^\pm}$ preserves gradings (compare (3.15), (3.16) in \cite{Ha}), the claim follows from Lemmas \ref{A9}, \ref{A8} and Proposition \ref{3308} below.
\qed

\subsection*{Cabling and skein relations}
Let us introduce the following relations in $\tcS(D^3\setminus T)$.
\\

\noindent\textbf{Cabling relations:}
\begin{enumerate}
\item Suppose $n_i=1$ for some $i$. The first cabling relation is $L=\tilde L$ where $\tilde L$ is obtained from $L$ by removing the $i$th component. 
\item Suppose $n_i \ge 3$ for some $i$. The second cabling relation is $L= L'' -L'$ where $L'$ is the link $L$ with the color of the $i$th component switched to $n_i-2$, and $L''$ is obtained from $L$ by replacing the $i$th component with two of its parallels, which are colored with $n_i-1$ and $2$.
\end{enumerate}
$\text{ }$

\noindent\textbf{Skein relations:}
\begin{enumerate}
\item The first skein relation is $U=q^{\frac 12}+q^{-\frac 12}$ where $U$ denotes the
unknot with framing zero and color 2.
\item Let $L_R$, $L_V$ and $L_H$ be unoriented  framed links with color 2 everywhere which are identical except in a disc where they are as shown in Figure \ref{Skein}. Then the second skein relation is $L_R = q^{\frac 14}L_V + q^{-\frac 14}L_H$ if the two strands in the crossing come from different components of $L_R$, and $L_R = \epsilon(q^{\frac 14}L_V - q^{-\frac 14}L_H)$ if the two strands come from the same component of $L_R$, producing a crossing of sign $\epsilon=\pm 1$ (i.e. appearing as in $L_{\epsilon}$ of Figure \ref{Skein} if $L_R$ is oriented).
\end{enumerate}

\begin{figure}[ht]
\begin{center}
\mbox{\hspace{3cm}
\includegraphics[height=1.9cm]{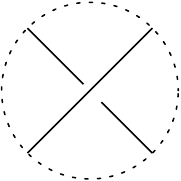}\hspace{7mm}
\includegraphics[height=1.9cm]{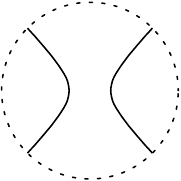}\hspace{7mm}
\includegraphics[height=1.9cm]{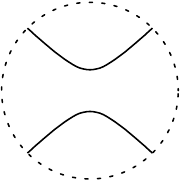}\hspace{7mm}
\includegraphics[height=1.9cm]{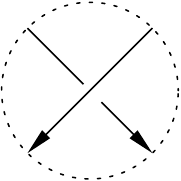}\hspace{7mm}
\includegraphics[height=1.9cm]{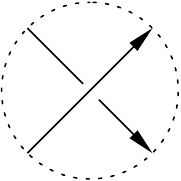}\hspace{7mm}
}
\hspace{-15mm}
\caption{\hspace{10mm} 
$L_R$ \hspace{19mm} 
$L_V$ \hspace{19mm} 
$L_H$ \hspace{19mm} 
$L_{+}$ \hspace{19mm} 
$L_{-}$
}
\label{Skein}
\end{center}
\end{figure}

We denote by $S(D^3\setminus T)$ the quotient of $\tcS(D^3\setminus T)$ by these relations. It is known as the  {\it skein module} of $D^3\setminus T$ (compare \cite{Pr}, \cite{SikPr} and \cite{Bullock}). Recall that the ground ring is $\Z[q^{\pm 1/4}]$.
\\

Using the cabling relations, we can reduce all colors of $L$ in $S(D^3\setminus T)$ to be 2. Note that the skein module $S(C\setminus I_n)$ has a natural algebra structure, given by putting one cube on top of the other. Let us denote by $A_n$ the subalgebra of this skein algebra generated by good morphisms.
\\

For a set $\gamma=(\gamma_1,\dots, \gamma_n)\in (\Z/2\Z)^n$ let $W_\gamma$ be a simple closed curve encircling the end points of those downward arrows with $\gamma_i=1$. See Figure \ref{Wgamma} for an example.
\begin{figure}[ht!]
\begin{center}
\includegraphics[height=2.7cm]{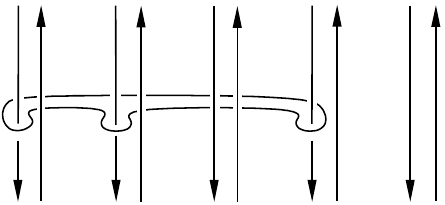}
\caption{\hspace{0mm} The element $W_{(1,1,0,1,0)}$.
\label{Wgamma}}
\end{center}
\end{figure}

Similarly to the case of Kauffman bracket skein module \cite{Bullock}, one can easily prove the following.
\begin{Alemma}\label{A9}
The algebra $\mathcal A_n$ is generated by $2^n$  curves $W_\gamma$.
\end{Alemma}

Using linearity, we can extend the definition of $J_X$ to $X =T \sqcup L$ where $L$ is {\em any element} of $\tcS(D^3\setminus T)$. It is known that $J_X$ is invariant under the cablings and skein relations (Theorem 4.3 of  \cite{KM}), hence $J_X$ is defined for $L \in S(D^3\setminus T)$. Moreover, we have the following.

\begin{Alemma} \label{A8}
Both gradings $\gr_\ve$ and $\gr_q$ are preserved under the cabling and skein relations. 
\end{Alemma}

\begin{proof}
The statement is obvious for the $\ve$--grading.
For the $q$--grading, notice that 
\[
\gr_q(L)=2\sum_{1\leq i<j\leq m}p_{ij}n_i' n_j' +
\sum_{1\leq j \leq m}p_{jj}n_j'^2+2\sum_{1\leq j\leq m}(p_{jj}+1)n_j',
\]
and therefore $\gr_q(L'')\equiv \gr_q(L')\equiv\gr_q(L)\pmod{4}$. This takes care of the cabling relations. 

Let us now assume that all colors of $L$ are equal to 2 and therefore
\[
\gr_q(L)=2\sum_{1\leq i< j\leq m} p_{ij}+3\sum_{i=1}^{m} p_{ii}+2m.
\]
The statement is obvious for the first skein relation. For the second skein relation, choose an arbitrary orientation on L. Let us first assume that the two strands in the crossing depicted in Figure \ref{Skein} come from the same component of $L_R$ and that the crossing is positive. Then, $L_V$ and $L_H$ have one positive self--crossing less, and $L_V$ has one link component more than $L_R$. Therefore
\begin{eqnarray*}
\gr_q(q^{\frac 14}L_V)&=&\gr_q(L_R)-3+2+1 \equiv \gr_q{L_R} \pmod{4} \quad\text{and} \\
\gr_q(q^{-\frac 14}L_H)&=&\gr_q(L_R)-3-1 \equiv \gr_q{L_R} \pmod{4}.
\end{eqnarray*}
It is obvious that this does not depend on the orientation of $L_R$. If the crossing of $L_R$ is negative or the two strands do not belong to the same component of $L_R$, the proof works in a similar way.
\end{proof}

\subsection*{Basic good morphisms}
Let  $\hat Z_n$ be $W_\gamma$ for $\gamma=(1,1,\dots, 1)\in (\Z/2\Z)^n$. We will also need the tangle $Z_n$ obtained from $\hat Z_n$ by removing the last up arrow.

\begin{equation*}\label{Zn}
    \def\sss{3em}
	 \hat Z_n  =\incl{\sss}{ZnVers2}\quad\quad
 \def\sss{3em}
	  Z_n  =\incl{\sss}{Zn}
\end{equation*}
Let $J_{Z_n}$ be the universal quantum invariant of $Z_n$, see \cite{Ha}.

\begin{Aproposition} \label{3308}
One has a presentation
\[
J_{\hat Z_n} = 
\sum z^{(n)}_{i_1} \otimes \sum z^{(n)}_{i_2} \otimes \dots \otimes
\sum z^{(n)}_{i_{2n}},
\]
such that  $z^{(n)}_{i_{2j-1}}\, z^{(n)}_{i_{2j}}\in v\; \U^1_q$ for every $j=1,\dots, n$.
\end{Aproposition}

\begin{Acorollary}\label{corA10}
$J_{\hat Z_n}$ satisfies Proposition \ref{MM2}.
\end{Acorollary}

\begin{proof}
Assume  $x\in \U_q^{\ve_1} \otimes \dots \otimes \U_q^{\ve_n}$, we have 
\[
J_{\hat Z_n}(x) = 
\sum z^{(n)}_{i_1} x_1  z^{(n)}_{i_2} \otimes \dots \otimes
\sum z^{(n)}_{i_{2n-1}} x_n z^{(n)}_{i_{2n}}.
\]
Hence, by Proposition \ref{3308} we get
\[
J_{\hat Z_n}(x) \in q^{1/2} \left( \U_q^{\ve_1'} \otimes \dots \otimes
 \U_q^{\ve_n'}\right), \quad \text{where} \quad (\ve'_1,\dots, \ve'_m) =
 (\ve_1,\dots, \ve_n) + (1,1,\dots, 1).
\]
The claim follows from the fact that $\gr_\ve(\hat Z_n)=(1,1,\dots,1)$ and $\gr_q(L)=2$.
\end{proof}

\section{Proof of Proposition \ref{3308}}
The statement holds true for $J_{\hat Z_1}= C\otimes \id_\uparrow$. 
Now Lemma 7.4 in \cite{Ha-b} states that applying $\Delta$ to the $i$th component of the universal quantum invariant of a tangle is the same 
as duplicating the $i$th component. Using this fact, we represent
\[ 
J_{Z_{n+1}} = (\id^{\otimes 2(n-1)}\otimes  {\Phi}) \left(  J_{Z_{n}} \right),
\]
where $\Phi$ is defined as follows. For $x \in \U_q$ with  $\Delta (x)=\sum x_{(1)}\otimes x_{(2)}$, we put
\[
\Phi(x) :=   \sum_{(x), m,n} x_{(1)} \otimes \beta_m S(\beta_n)
\otimes \alpha_n \, x_{(2)} \alpha_m
\]
where the $R$--matrix is given by $R =\sum_l {\alpha_l \otimes \beta_l}$. See Figure below for a picture. 
\begin{figure}[!ht]
\begin{center}
\mbox{
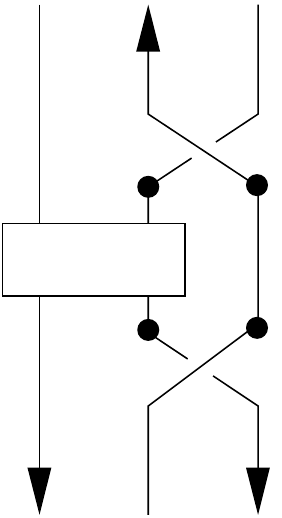}
\end{center}
\end{figure}

We are left with the computation of the $\ve$--grading of each component of $\Phi(x)$.

In $\U_q$, in addition to the $\ve$--grading, there is also the $K$--grading, defined by $|K|=|K^{-1}|=0, |e|=1, |F|=-1$. In general, the co--product $\Delta$ does not preserve the $\ve$--grading. However, we have the following. 
\begin{Alemma} \label{3304}
Suppose $x\in \U_q$ is homogeneous in both $\ve$--grading and $K$--grading. 
We have a presentation
\[
\Delta(x) = \sum_{(x)} x_{(1)} \otimes x_{(2)}
\]
where each $x_{(1)}$, $x_{(2)}$ is homogeneous with respect to the $\ve$--grading and $K$--grading. In addition,
\[
\gr_\ve(x_{(2)})= \gr_\ve(x)=\gr_\ve(x_{(1)} \, K^{-|x_{(2)}|}).
\]
\end{Alemma}

\begin{proof} If the statements hold true for $x, y \in \U_q$, then they hold true for $xy$. Therefore, it is enough to check the statements for the generators $e$, $\tF^{(1)},$ and $K$, for which they follow from explicit formulas of the co--product.
\end{proof}

\begin{Alemma} \label{3303}
Suppose $x \in \U_q$ is homogeneous in both $\ve$--grading and $K$--grading. There is a presentation
\[
\Phi(x) =\sum x_{i_0}\otimes x_{i_1}\otimes  x_{i_2}
\]
such that each $x_{i_0}$ is homogeneous in both $\ve$--grading and $K$--grading, and $\gr_\ve(x_{i_2})= \gr_\ve( x_{i_0}\, x_{i_1})= \gr_\ve(x)$.
\end{Alemma}

\begin{proof}
We put $D=\sum D' \otimes D'':=v^{\frac 12 H\otimes H}$. Using (see e.g. \cite{Ha}) 
\[
R=D\left(\sum_{n} q^{\frac 12 n(n-1)}\tF^{(n)}K^{-n}\otimes e^n\right),
\]
we get
\begin{eqnarray*}
\Phi(x)&=&\sum_{(x),n,m}
q^{\frac{1}{2}\left(m(m-1)+n(n-1)\right)}
x_{(1)}\otimes
D''_2 e^m S(D''_1e^{n})\otimes
D'_1\tF^{(n)}K^{-n}x_{(2)}D'_2\tF^{(m)}K^{-m}\\
&=&
\sum_{(x),n,m}(-1)^n q^{-\frac 12 m(m+1)-n(|x_{(2)}|+1)}
x_{(1)}\otimes e^{m}e^n K^{-|x_{(2)}|}\otimes \tF^{(n)}x_{(2)}\tF^{(m)}
\end{eqnarray*}
where we used $(\id \otimes S)D=D^{-1}$ and $D^{\pm1}(1\otimes x)=(K^{\pm|x|}\otimes x)D^{\pm 1}$ for homogeneous 
\linebreak
$x\in \U_q$ with respect to the $K$--grading. Now, the claim follows from Lemma \ref{3304}.
\end{proof}

By induction, using the fact that $C\in v \;\U^1_q$, Lemma \ref{3303} implies Proposition \ref{3308}. 

\newpage

\pagestyle{fancy}
\setlength{\headheight}{15.4pt}
\renewcommand{\chaptermark}[1]{\markboth{\MakeUppercase{\chaptername}\ \thechapter. \ #1}{}}
\renewcommand{\sectionmark}[1]{\markright{\thesection. \ #1}{}}
\fancyhead{}
\fancyhead[RE]{\slshape \leftmark}
\fancyhead[LE]{\thepage}
\fancyhead[LO]{\slshape \rightmark}
\fancyhead[RO]{\thepage}
\fancyfoot{}
\renewcommand{\headrulewidth}{0pt} 
\renewcommand{\footrulewidth}{0pt}

\end{document}